\RequirePackage{fix-cm}
\documentclass[smallextended]{svjour3}       
\smartqed  
\usepackage{graphicx}
\usepackage[utf8]{inputenc} 
\usepackage[T1]{fontenc}    
\usepackage{url}            
\usepackage{booktabs}       
\usepackage{amsmath,amssymb,amsfonts}       
\usepackage{nicefrac}       
\usepackage{microtype}      
\usepackage{bm}
\usepackage{float}
\usepackage{relsize}
\usepackage{graphicx}
\usepackage{color}
\usepackage{multirow}
\usepackage[bookmarks]{hyperref}
\usepackage{lineno}
\usepackage{stmaryrd}
\usepackage{xcolor}
\usepackage{caption}
\usepackage{subcaption}
\usepackage[textsize=small]{todonotes}
\usepackage{hyperref}
\usepackage{comment}
\usepackage{ragged2e}
\usepackage{ulem}
\newcommand{\half}{{\frac{1}{2}}}   
\hypersetup{
	colorlinks=true, 
	linkcolor=blue, 
	urlcolor=red, 
	citecolor=magenta,
	linktoc=all 
}

\newcommand{\df}[2]{\frac{\partial #1}{\partial #2}} 

\DeclareMathOperator\trace{trace}

\newcommand{\vel}{\bm{v}}
\newcommand{\p}{\mathcal{P}}

\newcommand{\con}{\bm{U}}
\newcommand{\prim}{\bm{W}}

\newcommand{\evar}{\bm{V}} 

\newcommand{\Old}[1]{} 

\newcommand{\f}{\bm{F}}
\newcommand{\fx}{\bm{F}^x}
\newcommand{\fy}{\bm{F}^y}

\newcommand{\s}{\bm{S}}

\newcommand{\B}{\bm{B}}
\newcommand{\Bx}{\bm{B}^x}
\newcommand{\By}{\bm{B}^y}

\newcommand{\detP}{\textrm{det}(\p)}

\newcommand{\rev}[1]{\color{black}#1 \color{black}}

\newtheorem{defn}{Definition}
\newtheorem{prop}{Proposition}
\begin{document}

\title{Entropy stable schemes for the shear shallow water model Equations}


\author{Anshu Yadav
	\and
       Deepak Bhoriya
       \and
       Harish Kumar\footnote{Corresponding Author}
       \and
       Praveen Chandrashekar
}

\institute{	Anshu Yadav  \at
	Department of Mathematics\\
	Indian Institute of Technology, Delhi\\
	New Delhi -- 110016, India \\
	\email{mkmanubalia@gmail.com}
	\and
	Deepak Bhoriya  \at
	Physics Department,\\
	University of Notre Dame, IN, USA\\
	\email{dbhoriy2@nd.edu}
	\and
	Harish Kumar  \at
	Department of Mathematics\\
	Indian Institute of Technology, Delhi\\
	New Delhi -- 110016, India \\
	\email{hkumar@iitd.ac.in}
	\and
	Praveen Chandrashekar \at
	Centre for Applicable Mathematics\\
	Tata Institute of Fundamental Research\\
	Bangalore -- 560065, India\\
	\email{praveen@math.tifrbng.res.in}
}

\date{Received: date / Accepted: date}

\maketitle

\begin{abstract}
The shear shallow water model is an extension of the classical shallow water model to include the effects of vertical shear. It is a system of six non-linear hyperbolic PDE with non-conservative products. We develop a high-order entropy stable finite difference scheme for this model in one dimension and extend it to two dimensions on rectangular grids. The key idea is to rewrite the system so that non-conservative terms do not contribute \rev{to} the entropy evolution. Then, we first develop an entropy conservative scheme for the conservative part, which is then extended to the complete system using the fact that the non-conservative terms do not contribute to the entropy production. The entropy dissipative scheme, which leads to an entropy inequality, is then obtained by carefully adding dissipative flux terms. The proposed schemes are then tested on several one and two-dimensional problems to demonstrate their stability and accuracy.
\end{abstract}

\section{Introduction}
The system of equations describing multi-dimensional shear shallow water~(SSW) flows was derived by Teshukov~\cite{Teshukov2007}. This system provides an approximation for shallow water flows by including the effects of vertical shear, which are neglected in the classical shallow water (Saint-Venant) model. It is derived from the incompressible Euler equations by a depth averaging process that gives rise to second-order velocity fluctuations, which are retained in the model but ignored in the classical model. Additional equations that account for the second-order fluctuations are also derived where third-order fluctuations arise but are neglected within the order of the approximations. The resulting system of equations has a very close resemblance to the Ten-moment Gaussian closure model of gas dynamics~\cite{Levermore1998}, except for the presence of some additional terms arising from gravitational effects. In particular, the entropy function of the two models is \rev{the} same since the non-conservative terms in \rev{the} SSW model, which are purely due to gravitational effects, do not make any contribution to the entropy equation.

Being a non-conservative hyperbolic system, the numerical solution of the SSW model is challenging since the notion of weak solution requires the choice of a path which is usually not known. The correct path depends on the physical regularization mechanism and even when the correct path is known, the construction of a numerical scheme that converges to the weak solution is hard since the solution is sensitive to the numerical viscosity~\cite{Abgrall2010}. In practice, a linear path is assumed in state space and some path conservative methods are developed which build some information of the waves present in the Riemann solution. For the SSW model, such methods have been developed following HLL-type ideas in~\cite{Gavrilyuk2018,bhole2019fluctuation,Chandrashekar2020}. The first two works split the model into some sub-systems and developed Riemann solver type methods, while the last one treats it in a unified manner by writing it in the form of the Ten-moment system. An exact Riemann solver has been developed in~\cite{Nkonga2022} for the linear path, and comparisons of the path conservative HLL-type numerical methods have been performed. The work in~\cite{Busto2021} proposes a slightly different model of the shear shallow water problem and develops a thermodynamically consistent scheme.

In the present work, we take a different approach to the construction of numerical methods, which is based on entropy consistency ideas~\cite{tadmor2003entropy,ismail2009affordable,fjordholm2012arbitrarily,chandrashekar2013kinetic}. The main technique is to first construct an entropy conservative scheme following the ideas of Tadmor and then add dissipative terms~\cite{ismail2009affordable,chandrashekar2013kinetic} that lead to an entropy inequality. For conservative systems, constructing the entropy conservative scheme is based on finding a central numerical flux that satisfies a certain jump condition~\cite{tadmor2003entropy}, see Theorem~\ref{thm:entconcond}. The SSW model is non-conservative, but the equation has conservative and non-conservative terms. The conservative terms have \rev{the} same structure as the Ten-moment equations of gas dynamics. Since the non-conservative terms do not contribute to the entropy, the ideas from conservative systems can be used to construct an entropy conservative scheme. This is the approach taken in the present work in the finite difference context where high-order accuracy is also achieved by following the ideas in LeFloch~\cite{leFloch2002}.  For conservation laws, there is a close relationship between the existence of a convex entropy function and the symmetrization of the equations, see \cite{Godlewski1996}, Theorem 3.2. This property does not hold for general non-conservative systems; for the SSW model, we have a convex entropy function and an entropy conservation law for smooth solutions, but the equations cannot be symmetrized. The failure to symmetrize is due to the non-conservative terms related to gravitational effects but since they do not contribute to the entropy equation, we still have an entropy equation satisfied by smooth solutions. A framework to construct entropy stable schemes for non-conservative hyperbolic systems is presented in~\cite{Castro2013}, which uses the idea of path-consistent schemes and fluctuation splitting. Our approach is, however, different from this as we exploit the conservation form and the special structure of the non-conservative terms, which do not contribute to the entropy.  The scheme is first developed in one dimension and extended to two dimensions on logically rectangular meshes. \rev{The stability and accuracy of the proposed schemes are demonstrated on several test cases in one and two dimensions. We have also compared the computed solutions with the exact solutions for several test cases. For the roll wave test cases, we have compared the computed solutions with the roll waves observed in some experimental studies in one and two dimensions.}

The rest of the paper is organized as follows. Section~\ref{sec:ssweq} presents the non-conservative SSW model in \rev{a form where the conservative terms are similar to the Ten-moment equations}. The entropy function and the entropy equation are discussed in Section~\ref{sec:entropy}. The semi-discrete entropy conservative and dissipative schemes are constructed in Section~\ref{sec:semid}, which is also extended to higher order accuracy, and the entropy condition is demonstrated. Section~\ref{sec:fulld} discusses the fully discrete scheme obtained by adding a time integration scheme. Section~\ref{sec:res} presents numerical results obtained from the proposed schemes in one and two dimensions, and Section~\ref{sec:sum} provides a summary. In the appendices, we examine the symmetrizability issue of the SSW model and derive the entropy scaled eigenvectors which are used to construct the entropy stable dissipative fluxes.

\section{Equations of shear shallow water model}
\label{sec:ssweq}
The shear shallow water model has been recently studied in \cite{Chandrashekar2020} and expressed in an almost conservative form for \rev{the} evolution of the water depth $h$, the depth average momentum $h\bm{v}$ and the energy tensor $\mathcal{E}=\{\mathcal{E}_{11},\mathcal{E}_{12},\mathcal{E}_{22}\}$. It is a system of non-linear, non-conservative hyperbolic partial differential equations. In $2-$D, following \cite{Chandrashekar2020}, the governing equations of the shear shallow water model (SSW) can be expressed as,
\begin{equation}
	\label{eq:ssw}
	\df{\con}{t} + \df{\fx}{ x} + \df{\fy}{ y} + \Bx \df{h}{ x} + \By \df{h}{y} =  \s
\end{equation}
where
\begin{equation*}
	\con=\begin{pmatrix}
		h\\
		h v_1\\
		h v_2\\
		\mathcal{E}_{11}\\
		\mathcal{E}_{12}\\
		\mathcal{E}_{22}
	\end{pmatrix}, \quad
	\fx=\begin{pmatrix}
		h v_1\\
		h (v_{1}^2+\p_{11})\\
		h (v_1 v_2+\p_{12})\\
		\frac{1}{2}h v_1(v_1^2+{3} \p_{11})\\
		\frac{1}{2}h (v_1^2 v_2+2 v_1 \p_{12}+v_2\p_{11})\\
		\frac{1}{2}h (v_1 v_2^2+2 v_2\p_{12}+v_1 \p_{22})
	\end{pmatrix},
\end{equation*}
\begin{equation*}
	\fy=\begin{pmatrix}
		h v_2\\
		h (v_1 v_2+\p_{12})\\
		h (v_{2}^2+{\p_{22}})\\
		\frac{1}{2}h(v_1^2 v_2+2 v_1 \p_{12}+v_2\p_{11})\\
		\frac{1}{2}h(v_1 v_2^2+2v_2\p_{12}+v_1 \p_{22})\\
		\frac{1}{2}h (v_2^3+{3}v_2 \p_{22}),
	\end{pmatrix},
	\Bx=\begin{pmatrix}
		0\\
		gh\\
		0\\
		ghv_1\\
		\frac{1}{2}ghv_2\\
		0
	\end{pmatrix}, \quad
	\By=\begin{pmatrix}
		0\\
		0\\
		gh\\
		0\\
		\frac{1}{2} ghv_1\\
		ghv_2
	\end{pmatrix}, 
\end{equation*}
\begin{equation*}
	\s=\begin{pmatrix}
		0\\
		-gh\frac{\partial b}{\partial x}-C_f|\vel|v_1\\
		-gh\frac{\partial b}{\partial y} -C_f|\vel|v_2\\
		-\alpha |\vel|^3 \p_{11}-ghv_1\frac{\partial b}{\partial x}-C_f|\vel|v_1^{2}\\
		-\alpha |\vel|^3 \p_{12}-\frac{1}{2}ghv_2\frac{\partial b}{\partial x} -\frac{1}{2}ghv_1\frac{\partial b}{\partial y}-C_f|\vel|v_1 v_{2}\\
		-\alpha |\vel|^3\p_{22} -ghv_2\frac{\partial b}{\partial y}-C_f|\vel|v_2^{2}
	\end{pmatrix}.
\end{equation*}
In the above set of equations, $\bm{v}=(v_1,v_2)$ is the velocity vector, $g>0$ is the acceleration due to gravity, $b\equiv b\left( x, y\right)$ is the bottom topography, $C_f$ is the Chezy coefficient and $\alpha$ is given by the following relation \cite{Gavrilyuk2018,richard2013classical},
\[
\alpha = \max\left(0, C_r \frac{T - \phi h^2}{T^2} \right), \qquad T = \trace(\p) = \p_{11} + \p_{22}, \qquad C_r > 0,
\]
where, $\p=\{\p_{11},\p_{12},\p_{22}\}$ is the Reynolds stress tensor, which is symmetric, positive definite, and arises due to depth averaging. The quantities $C_f,~ C_r,~\phi$ are model constants and must be determined from experiments. The above system is closed with the equation of state,
\[
\mathcal{E}= \frac{h}{2}\left ( \rule{0mm}{4mm} \vel\otimes\vel + \p \right ).
\]
Next, we define the set of primitive variables $\prim$,
$$
\prim=(h,v_1,v_2,\p_{11},\p_{12},\p_{22})^\top
$$

For the solution to be physically acceptable, we need \rev{the} water depth $h$ and \rev{the} symmetric stress tensor $\p$ to be positive. Hence, we consider the following set $\Omega$ of physically admissible solutions,
\begin{align*}
	\Omega=\{\con\in \mathbb{R}^6 |~h>0,~x^\top\p x>0, \forall x\in \mathbb{R}^2 \setminus \{(0,0)\}\}.
\end{align*}
\rev{Now for the solutions of the homogeneous case (i.e., $\s=0$) in $\Omega$, the system \eqref{eq:ssw} is hyperbolic for the states $\con\in \Omega$ with the following set of eigenvalues,
	\[
	\lambda_1 = v_d - \sqrt{gh + 3 \p_{dd}}, \quad \lambda_2 = v_d - \sqrt{\p_{dd}},    \quad \lambda_3 = \lambda_4 = v_d,\]
	\[  \lambda_5 = v_d+ \sqrt{\p_{dd}}, \quad     \lambda_6 = v_d + \sqrt{g h + 3 \p_{dd}}.
	\]
	Here, $d\in \{1,2\}$ and the  indices $\{1,2\}$ denote the $x-$direction and $y-$direction respectively.  The first and last eigenvalues correspond to genuinely non-linear characteristic fields in the sense of Lax~\cite{Godlewski1996}, while the remaining eigenvalues correspond to linearly degenerate characteristic fields~\cite{Gavrilyuk2018}.
}
In \rev{the} $x$-direction, the matrix of right eigenvectors in terms of primitive variables $\prim$ is given by
\begin{align*}
	{R}_{\prim}^{x}=\begin{pmatrix}
		h (A^2-C^2)                        & 0                        & -h               & 0 & 0                       & h (A^2-C^2)                       \\
		-A(A^2-C^2) & 0                        & 0                & 0 & 0                       & A(A^2-C^2) \\
		-2A\p_{12} & -C & 0                & 0 & C & 2A\p_{12} \\
		2C^2(A^2-C^2)                      & 0                        &g h+ \p_{11} & 0 & 0                       & 2C^2(A^2-C^2)                     \\
		\p_{12}(A^2+C^2)         & C^2        & \p_{12} & 0 & C^2       & \p_{12}(A^2+C^2)  \\
		4\p_{12}^2                       & 2\p_{12}        &   0             & 1 & 2\p_{12}      & 4\p_{12}^2
	\end{pmatrix}.
\end{align*}
where $A = \sqrt{gh + 3 \p_{11}}$ and $C = \sqrt{\p_{11}}$. One can get the matrix of right eigenvectors in conservative variables by pre-multiplying the above matrix by the Jacobian matrix $\dfrac{\partial \mathbf{U}}{\partial \bm{W}}$ for the change of variable.
\section{Entropy analysis}
\label{sec:entropy}

Solutions of a nonlinear hyperbolic system can be discontinuous even for very smooth initial data. This leads us to the consideration of weak solutions, which, however, may not be unique. Hence, an additional criterion is considered to select the physically relevant solution among all weak solutions in terms of the entropy condition.

For the SSW model~\eqref{eq:ssw}, we follow~\cite{berthon2006numerical,biswas2021entropy,sen_entropy_2018,berthon2015entropy} to define the entropy $\eta$ and the entropy fluxes \rev{$(q^x, q^y)$} as follows 
\begin{align}\label{entropy-pair}
	\eta=\eta(\con) = -h s, \qquad q^x=-h v_1s, \qquad q^y=-hv_2s
\end{align}
where
\[
s=\log \left( \dfrac{\det\p}{h^2} \right)=\log \left( \dfrac{\p_{11}\p_{22}-\p_{12}^2}{h^2} \right)
\]
For the homogeneous case, we will now prove the entropy equation. We proceed in one dimension as the two and three-dimensional cases are similar. The proof is similar to the entropy equality proof for the Ten-Moment equations presented in \cite{berthon2015entropy,sen_entropy_2018}.
\rev{\begin{prop}
 \label{prop:entropy}
	Smooth solutions of \eqref{eq:ssw} without the source term satisfy the following entropy equality,
	\begin{equation}
		\partial_t s+v_1 \partial_x s=0. \label{s}
	\end{equation}
	As a corollary, for any smooth function H(s), we have,
	\begin{equation}
		\partial_t(h H(s))+ \partial_x (hv_1 H(s))=0. \label{h_s}
	\end{equation}
	In particular, smooth solutions will satisfy the entropy equality,
	\begin{equation}
		\partial_t\eta+ \partial_xq^x=0. \label{entropy_pair_equality}
	\end{equation}
\end{prop}}
\begin{proof}
	First, we will prove the equality~\eqref{s}. Assuming $\bm{U}$ is a smooth solution of the system~\eqref{eq:ssw} for \rev{the} homogeneous case, we subtract the kinetic energy contributions from the energy equations to obtain the following equations in terms of the stress components,
	\begin{align*}
		\partial_t \p_{11}+v_1\partial_x \p_{11}+2\p_{11}\partial_x v_1 & =0, & \\
		\partial_t \p_{12}+v_1\partial_x \p_{12}+\p_{12}\partial_x v_1+\p_{11}\partial_x v_2 & =0, & \\
		\partial_t \p_{22}+v_1\partial_x \p_{22}+2\p_{12}\partial_x v_2 & =0.
	\end{align*}
	Using the definition, $\det\p=\p_{11}\p_{22}-\p_{12}^2$, we apply the chain rule and use the above set of equations to obtain,
	\begin{align}
		\partial_t \det\p+v_1\partial_x\det\p+2\det\p\partial_x v_1=0. \label{detp}
	\end{align}
	Now	using \eqref{detp} and \rev{the} water depth equation, $\partial_t h + v_1 \partial_x h + h \partial_x v_1 = 0$, we get,
	\begin{align*}
		\partial_t s +v_1\partial_x s &=\frac{1}{\det \mathcal{P}} \partial_t \detP  - \frac{2}{h}\partial_t h+ v_1\frac{1}{\det \mathcal{P}} \partial_x \detP  - v_1 \frac{2}{h}\partial_x h\\
		&=-\frac{1}{\det\p}\left(v_1\partial_x \det\p +2 \det\p \partial_x v_1 \right) + \frac{1}{\det\p}\left(v_1\partial_x \det\p \right) \\
		&+\frac{2}{h}\left( h\partial_x v_1 +v_1\partial_x h\right) -v_1\frac{2}{h}\partial_x h\\
		&=-2\partial_x v_1 + 2\partial_xv_1\\
		&=0
	\end{align*}
	The relations \eqref{h_s} and \eqref{entropy_pair_equality} can now be obtained using a simple application of the chain rule on the Eqn.~\eqref{s}.
\end{proof}
From the proof of Proposition \ref{prop:entropy}, we observe that the non-conservative terms containing the gravitational effects do not make any contribution to the entropy evolution. In fact, this also follows from the fact that $\eta'(\con) \Bx(\con) = \eta'(\con) \By(\con) = 0$. This motivates the following definition of entropy function for non-conservative systems.
\begin{defn}
	A convex function $\eta(\con)$ is said to be an entropy function for the system
	\begin{equation*}
		\df{\con}{t} + \df{\fx}{ x} + \df{\fy}{ y} + \tilde\Bx \df{\con}{ x} + \tilde\By \df{\con}{y} =  0
	\end{equation*}
	if there exist smooth functions \rev{$q^x(\con)$} and \rev{$q^y(\con)$} such that
	\begin{equation*}
		{q^x}'(\con) = \eta'(\con){\fx}'(\con), \qquad {q^y}'(\con) = \eta'(\con){\fy}'(\con)
	\end{equation*}
	and
	\[
	\eta'(\con) \tilde\Bx(\con) = \eta'(\con) \tilde\By(\con) = 0
	\]
	The functions ($\eta, q^x, q^y$) form an entropy-entropy flux pair.
\end{defn}
The SSW model can be put in the above form with the matrix $\tilde\Bx$ containing the vector $\Bx$ in its first column and similarly, the matrix $\tilde\By$ containing the vector $\By$ in its first column, and all other columns being zero. We have seen above that the SSW model has the entropy pair $(\eta, q)$ and additionally satisfies the conservation law \eqref{entropy_pair_equality} for smooth solutions in the absence of source terms, while for discontinuous solutions, we can demand the entropy inequality
\begin{equation}\label{eq:entropyineq}
	\partial_t\eta+ \partial_xq^x \le 0
\end{equation}
to hold in the sense of distributions. In the next Section, we will develop semi-discrete numerical schemes that satisfy a discrete entropy inequality~\eqref{eq:entropyineq}.  There is a close connection between the existence of an entropy pair and the symmetrization of a system of conservation laws. Since the SSW model is a  non-conservative hyperbolic system, we investigate the symmetrizability of the system in detail in Appendix~\ref{symmetrizability}. Based on the discussion in  Appendix~\ref{symmetrizability}, we conclude this section with the following remark.
\begin{remark}
	The existence of an entropy pair does not guarantee the symmetrizability of the system in the case of non-conservative hyperbolic systems. In particular, \rev{the} SSW system~\eqref{eq:ssw} is not symmetrizable.
\end{remark}
\section{Semi-discrete numerical schemes}
\label{sec:semid}
We can rewrite the SSW model~\eqref{eq:ssw} as follows,
\begin{equation}
	\label{eq:ssw1}
	\df{\con}{t} + \df{\fx}{ x} + \df{\fy}{ y} +
	\B^{nc} =\s,
\end{equation}
where $\B^{nc}= \Bx \df{h}{ x} + \By\df{h}{y}$.
In this Section, we will first develop semi-discrete schemes for the homogeneous part of the system \eqref{eq:ssw1}. The discretization of the source term is then discussed in Section~\ref{subsec:semidisscheme}. We discretize the domain $D=(x_a,x_b) \times (y_a,y_b)$ uniformly into cells $I_{ij}$ with mesh size of $\Delta x \times \Delta y$, where \rev{$\Delta x= \frac{x_b-x_a}{N_x}$ and $\Delta y= \frac{y_b-y_a}{N_y}$. We define the grid points by $x_i=x_a+i \Delta x $, $y_j=y_a+j \Delta y$, , with $0 \le i \le N_x$ and $0 \le j \le N_y$. We also define cell interfaces as $x_{i+1/2}= \frac{x_i+x_{i+1}}{2}$, $y_{j+1/2}= \frac{y_j+y_{j+1}}{2}$.} Then a general semi-discrete conservative finite difference scheme has the following form,
\begin{align}
	\frac{d }{dt}\con_{i,j}(t)
	+ &
	\frac{1}{\Delta x}\left({\fx}_{i+\frac{1}{2},j}(t)-{\fx}_{i-\frac{1}{2},j}(t)\right)\nonumber \\
	+ & \frac{1}{\Delta y}  \left({\fy}_{i,j+\frac{1}{2}}(t)-{\fy}_{i,j-\frac{1}{2}}(t)\right)+\B^{nc}_{i,j}(\con)=0, \label{scheme}
\end{align}
where ${\fx}_{i+\frac{1}{2},j}$ and ${\fy}_{i,j+\frac{1}{2}}$ are the numerical fluxes consistent with the continuous fluxes ${\fx}$ and ${\fy}$, respectively. The derivatives $\frac{\partial h}{\partial x},~\frac{\partial h}{\partial y}$ in the non-conservative term ${\B}_{i,j}^{nc}={\B^x(\con_{i,j})}\big(\frac{\partial h}{\partial x}\big)_{i,j}+{\B^y(\con_{i,j})}\big(\frac{\partial h}{\partial y}\big)_{i,j}$ are approximated by suitable order central difference approximations.\\

The semi-discrete scheme \eqref{scheme} is said to be an entropy stable scheme if the computed solution satisfies the following entropy inequality,
\begin{equation*}
	\frac{d}{dt}  \eta(\con_{i,j})  +\frac{1}{\Delta x} \left( \hat{q^x}_{i+\frac{1}{2},j} - \hat{q^x}_{i-\frac{1}{2},j}\right)+\frac{1}{\Delta y}\left( \hat{q^y}_{i,j+\frac{1}{2}} - \hat{q^y}_{i,j-\frac{1}{2}}\right) \le 0,
\end{equation*}
for some numerical entropy fluxes $\hat{q^y}$ and $\hat{q^y}$  consistent with the fluxes ${q}^x$ and $q^y$, respectively.
The procedure for construction of an entropy stable scheme involves first constructing an entropy conservative scheme.  We say the semi-discrete scheme \eqref{scheme} is an entropy conservative scheme if the computed solution satisfies the following entropy equality
\begin{equation*}
	\frac{d}{dt}  \eta(\con_{i,j})  +\frac{1}{\Delta x} \left( \tilde{q^x}_{i+\frac{1}{2},j} - \tilde{q^x}_{i-\frac{1}{2},j}\right)+\frac{1}{\Delta y}\left( \tilde{q^y}_{i,j+\frac{1}{2}} - \tilde{q^y}_{i,j-\frac{1}{2}}\right) = 0,
\end{equation*}
for some numerical entropy fluxes $\tilde{q^x}$ and $\tilde{q^y}$  consistent with the fluxes ${q}^x$ and $q^y$, respectively. Hence, first, we discuss the construction of entropy conservative scheme.
\subsection{Entropy conservative schemes}
For the construction of numerical flux that leads to an entropy conservative scheme, we define the entropy variable $\evar=\frac{\partial \eta}{\partial \con}$ and entropy potential $\psi^{k}=\evar^\top \bm{F^k}-q^{k},~k\in \{x,y\}$.
A simple calculation results in,
\begin{align}\label{entvar}
	\renewcommand{\arraystretch}{2}
	\evar:=\frac{\partial \eta}{\partial \con}=\begin{pmatrix}
		4-s- \dfrac{1}{\detP}\left(\p_{11} v_2^2+\p_{22} v_1^2-2\p_{12}v_1 v_2\right)\\
		\dfrac{2 (\p_{22}v_1-\p_{12}v_2)}{\detP}\\
		\dfrac{2  (\p_{11}v_2-\p_{12}v_1)}{\detP}\\
		-\dfrac{2  \p_{22}}{\detP}\\
		\dfrac{4  \p_{12}}{\detP}\\
		-\dfrac{2  \p_{11}}{\detP}
	\end{pmatrix}
\end{align}
The entropy potentials are given by,
\begin{align*}
	\psi^x=\evar^\top \fx-q^x= 2 h v_1, \qquad \psi^y=\evar^\top \fy-q^y= 2 h v_2.
\end{align*}
We now recall the following theorem, which provides us a procedure for the construction of entropy conservative fluxes, ${\tilde{\fx}}$ and ${\tilde{\fy}}$. For \rev{a given} variable $a$, we introduce the notations $[\![{\cdot}]\!]$ for the jump and $\bar{\cdot}$ for the arithmetic average in the following way,
\begin{equation*}
	[\![a]\!]_{i+\frac{1}{2},j}=a_{i+1,j}-a_{i,j},  \qquad \bar{a}_{i+\frac{1}{2},j}=\frac{1}{2}(a_{i+1,j}+a_{i,j}),
\end{equation*}
\begin{equation*}
	[\![a]\!]_{i,j+\frac{1}{2}}=a_{i,j+1}-a_{i,j}, \qquad
	\bar{a}_{i,j+\frac{1}{2}}=\frac{1}{2}(a_{i,j+1}+a_{i,j}).
\end{equation*}
%

\begin{theorem}[{Tadmor} \cite{tadmor1987numerical}]\label{thm:entconcond}
	Let ${\tilde{\fx}}$ and ${\tilde{\fy}}$ be the consistent numerical fluxes, which satisfy
	\begin{equation}
		[\![\evar ]\!]^{\top}_{i+\frac{1}{2},j}\,{\tilde{\fx}}_{i+\frac{1}{2},j}=[\![\psi^x]\!]_{i+\frac{1}{2},j}, \ \quad \
		[\![\evar ]\!]^{\top}_{i,j+\frac{1}{2}}\,{\tilde{\fy}}_{i,j+\frac{1}{2}}=[\![\psi^y]\!]_{i,j+\frac{1}{2}},
		\label{tadmor_thm}
	\end{equation}
	then the scheme~\eqref{scheme} with the numerical fluxes ${\tilde{\fx}}$ and ${\tilde{\fy}}$ is second-order accurate and entropy conservative, i.e., the computed solutions satisfy the discrete entropy equality
	\begin{equation*}
		\frac{d}{dt}  \eta(\con_{i,j})  +\frac{1}{\Delta x} \left( \tilde{q^x}_{i+\frac{1}{2},j} - \tilde{q^x}_{i-\frac{1}{2},j}\right)+\frac{1}{\Delta y}\left( \tilde{q^y}_{i,j+\frac{1}{2}} - \tilde{q^y}_{i,j-\frac{1}{2}}\right) = 0,
	\end{equation*}
	corresponding to the numerical entropy fluxes,
	\begin{equation*}
		\tilde{q^x}_{i+\frac{1}{2},j}=\bar{\evar}_{i+\frac{1}{2},j}^{\top}{\tilde{\fx}}_{i+\frac{1}{2},j}-\bar{\psi^x}_{i+\frac{1}{2},j} \qquad
		\text{and} \qquad
		\tilde{q^y}_{i,j+\frac{1}{2}}=\bar{\evar}_{i,j+\frac{1}{2}}^{\top}{\tilde{\fy}}_{i,j+\frac{1}{2}}-\bar{\psi^y}_{i,j+\frac{1}{2}}.
	\end{equation*}
\end{theorem}
First, we consider the $x$-directional identity \eqref{tadmor_thm} to get the conservative flux in \rev{the} $x$-direction. Note that we have \rev{a} single algebraic equation with $6$ unknowns ${\tilde{\fx}}=[\tilde{f}_1,\,\tilde{f}_2,\,\tilde{f}_3,\,\tilde{f}_4,\,\tilde{f}_5,\,\tilde{f}_6]^\top$. Therefore, we cannot have \rev{a} unique solution for the algebraic equation~\eqref{tadmor_thm}. In \cite{ismail2009affordable,chandrashekar2013kinetic}, \rev{the} authors have presented a procedure to find an inexpensive entropy conservative flux. For the SSW model \eqref{eq:ssw1}, we follow the approach presented in \cite{chandrashekar2013kinetic} to construct an entropy conservative flux in the next sub-section~\eqref{entropy_con_flux}.
%
%
\subsubsection{Entropy conservative flux} \label{entropy_con_flux}
We first consider the $x-$directional case. Following \cite{tadmor1987numerical}, we need to find an entropy conservative flux ${\tilde{\fx}}=[\tilde{f}_1,\,\tilde{f}_2,\,\tilde{f}_3,\,\tilde{f}_4,\,\tilde{f}_5,\,\tilde{f}_6]^\top$ satisfying the identity:
\begin{align}
	[\![ \evar]\!]^\top \cdot {\tilde{\fx}}=[\![ \psi^x]\!].\label{confluxx}
\end{align}
For simplicity, we ignore the indices $i$ and define
\begin{equation*}
	D = {\det(\p)}=\p_{11}\p_{22}-\p_{12}^2, \quad
	\beta_{11}=\dfrac{\p_{11}}{D}, \quad
	\beta_{12}=\dfrac{\p_{12}}{D},
\end{equation*}
\begin{equation*}
	\beta_{22}=\dfrac{\p_{22}}{D}, \quad
	D_{\beta}=\beta_{11}\beta_{22}-\beta_{12}^2
\end{equation*}
We also define the logarithmic average,  $a^{\ln}=\dfrac{[\![a]\!]}{[\![\ln a]\!]}$. As \rev{the} conservative flux is the same as the flux of Ten-Moment equations (where \rev{the} water depth is replaced by density), we use the entropy conservative flux derived in \cite{sen_entropy_2018} for Ten-Moment equations. The expression of the numerical flux is,

		\begin{equation*}
			{\tilde{\fx}}=
			\begin{pmatrix}
				h^{ln} \bar{v}_1  \\
				\bar{v}_1\tilde{f}_1+\frac{\bar{\beta}_{11} \bar{h}}{\bar{\beta}_{11}\bar{\beta}_{22}-\left(\bar{\beta}_{12}\right)^2}                             \\
				\bar{v}_2\tilde{f}_1+\frac{\bar{\beta}_{12} \bar{h}}{\bar{\beta}_{11}\bar{\beta}_{22}-\left(\bar{\beta}_{12}\right)^2}                             \\
				\frac{1}{2}\left(\dfrac{\bar{\beta}_{11}}{D_\beta^{ln}}-\overline{\left(v_1\right)^2}\right)\tilde{f}_1+\bar{v}_1\tilde{f}_2                       \\
				\frac{1}{2}\bigg(\left(\dfrac{\bar{\beta}_{12}}{D_\beta^{ln}}-\overline{v_1 v_2}\right)\tilde{f}_1+\bar{v}_1\tilde{f}_3+\bar{v}_2\tilde{f}_2\bigg) \\
				\frac{1}{2}\left(\dfrac{\bar{\beta}_{22}}{D_\beta^{ln}}-\overline{\left(v_2\right)^2}\right)\tilde{f}_1+\bar{v}_2\tilde{f}_3
			\end{pmatrix}.
		\end{equation*}
		The $y-$directional entropy conservative flux is given by, $\bm{\tilde{\bm{F}}^y}=[\tilde{g}_1,\,\tilde{g}_2,\,\tilde{g}_3,\,\tilde{g}_4,\,\tilde{g}_5,\,\tilde{g}_6]^\top$ as
		\begin{equation*}
			{\tilde{\fy}}=
			\begin{pmatrix}
				h^{ln} \bar{v}_2        \\
				\bar{v}_1\tilde{g}_1+\frac{\bar{\beta}_{12} \bar{h}}{\bar{\beta}_{11}\bar{\beta}_{22}-\left(\bar{\beta}_{12}\right)^2}                             \\
				\bar{v}_2\tilde{g}_1+\frac{\bar{\beta}_{22} \bar{h}}{\bar{\beta}_{11}\bar{\beta}_{22}-\left(\bar{\beta}_{12}\right)^2}                             \\
				\frac{1}{2}\left(\dfrac{\bar{\beta}_{11}}{D_\beta^{ln}}-\overline{\left(v_1\right)^2}\right)\tilde{g}_1+\bar{v}_1\tilde{g}_2                       \\
				\frac{1}{2}\bigg(\left(\dfrac{\bar{\beta}_{12}}{D_\beta^{ln}}-\overline{v_1 v_2}\right)\tilde{g}_1+\bar{v}_1\tilde{g}_3+\bar{v}_2\tilde{g}_2\bigg) \\
				\frac{1}{2}\left(\dfrac{\bar{\beta}_{22}}{D_\beta^{ln}}-\overline{\left(v_2\right)^2}\right)\tilde{g}_1+\bar{v}_2\tilde{g}_3
			\end{pmatrix}.
		\end{equation*}
		Note that these are two-point fluxes, i.e., they depend on two states. One can easily observe that the above fluxes ${\tilde{\fx}}$ and ${\tilde{\fy}}$ are consistent with the exact fluxes $\fx$ and $\fy$, respectively, when the two states are equal.
		
		
		\subsection{Higher order entropy conservative schemes}
		The entropy conservative fluxes presented above are only second-order accurate. To get higher-order accurate conservative fluxes, we follow the approach of \cite{leFloch2002}. They have constructed $2p^{th}$, \rev{$p\in\mathbb{N}$}, order accurate entropy conservative flux by choosing specific linear combinations of the second-order accurate entropy conservative fluxes. In particular, the $x$-directional entropy conservative flux for the $4^{th}$-order ($p=2$) scheme is given by
		\begin{eqnarray}
			{\tilde{\fx}}_{i+\frac{1}{2},j}^4=\frac{4}{3}{\tilde{\fx}}_{i+\frac{1}{2},j}(\con_{i,j},\con_{i+1,j})\nonumber\\-\frac{1}{6} \bigg( {\tilde{\fx}}_{i+\frac{1}{2},j} (\con_{i-1,j},\con_{i+1,j})+
			{\tilde{\fx}}_{i+\frac{1}{2},j}(\con_{i,j},\con_{i+2,j}) \bigg).\label{4thorder_x}
		\end{eqnarray}
		\rev{A} similar expression can be derived for the $y$-directional $4^{th}$-order flux
		\begin{eqnarray}
			{\tilde{\fy}}_{i,j+\frac{1}{2}}^4=\frac{4}{3}{\tilde{\fy}}_{i,j+\frac{1}{2}}(\con_{i,j},\con_{i,j+1})\nonumber\\-\frac{1}{6} \bigg( {\tilde{\fy}}_{i,j+\frac{1}{2}} (\con_{i,j-1},\con_{i,j+1})+
			{\tilde{\fy}}_{i,j+\frac{1}{2}}(\con_{i,j},\con_{i,j+2}) \bigg).\label{4thorder_y}
		\end{eqnarray}
		The scheme~\eqref{scheme} with the numerical fluxes ${\tilde{\fx}}^4$ and ${\tilde{\fy}}^4$ is fourth order accurate and entropy conservative.
		
		
\subsection{Entropy stable schemes}
		As the entropy needs to decay at shocks, the entropy conservative schemes designed above will produce oscillations at the shock. Hence, we need an appropriate entropy dissipation process, resulting in the entropy inequality. We follow \cite{tadmor1987numerical} to define the modified fluxes ${\hat{\fx}},~{\hat{\fy}}$ as follows:
		\begin{equation}
			\begin{aligned}
				{\hat{\fx}}_{i+\frac{1}{2},j} ={\tilde{\fx}}_{i+\frac{1}{2},j} - \frac{1}{2} \textbf{D}^{x}_{i+\frac{1}{2},j}[\![ \evar]\!]_{i+\frac{1}{2},j},
				\\
				{\hat{\fy}}_{i,j+\frac{1}{2}} = {\tilde{\fy}}_{i,j+\frac{1}{2}} - \frac{1}{2} \textbf{D}^{y}_{i,j+\frac{1}{2}}[\![ \evar]\!]_{i,j+\frac{1}{2}},
				\label{es_numflux}
			\end{aligned}
		\end{equation}
		where $\textbf{D}^x_{i+\frac{1}{2},j}$ and $\textbf{D}^y_{i,j+\frac{1}{2}}$ are symmetric positive definite matrices. Then we have the following Lemma:
		\begin{lemma}[Tadmor \cite{tadmor1987numerical}] The numerical scheme \eqref{scheme} with the modified numerical fluxes \eqref{es_numflux} is entropy stable, i.e., the computed solution satisfies,
			\begin{equation*}
				\frac{d}{dt}  \eta(\con_{i,j})  +\frac{1}{\Delta x} \left( \hat{q^x}_{i+\frac{1}{2},j} - \hat{q^x}_{i-\frac{1}{2},j}\right)+\frac{1}{\Delta y}\left( \hat{q^y}_{i,j+\frac{1}{2}} - \hat{q^y}_{i,j-\frac{1}{2}}\right) \le 0,
			\end{equation*}
			with consistent numerical entropy flux functions,
			\begin{equation*}
				\begin{aligned}
					\hat{q^x}_{i+\frac{1}{2},j}=  \tilde{q^x}_{i+\frac{1}{2},j} - \frac{1}{2}\bar{\evar}_{i+\frac{1}{2},j}^{\top}  \textbf{D}_{i+\frac{1}{2},j}^{x}[\![ \evar]\!]_{i+\frac{1}{2},j}
				\end{aligned} \end{equation*}
			
					and \begin{equation*}
						\begin{aligned}
							\qquad \hat{q^y}_{i,j+\frac{1}{2}}=    \tilde{q^y}_{i,j+\frac{1}{2}} -  \frac{1}{2}\bar{\evar}_{i,j+\frac{1}{2}}^{\top}  \textbf{D}_{i,j+\frac{1}{2}}^{y}[\![ \evar]\!]_{i,j+\frac{1}{2}}.\end{aligned} \end{equation*}
		\end{lemma}
		Here, we use \textit{Rusanov's type} diffusion operators for the matrix $\textbf{D}$, given by,
		\begin{equation}  \label{diffusiontype}
			\textbf{D}_{i+\frac{1}{2},j}^{x} = \tilde{R}^x_{i+\frac{1}{2},j} \Lambda_{i+\frac{1}{2},j}^x \tilde{R}_{i+\frac{1}{2},j}^{x \top} \qquad \text{and} \qquad \textbf{D}_{i,j+\frac{1}{2}}^y = \tilde{R}_{i,j+\frac{1}{2}}^y \Lambda_{i,j+\frac{1}{2}}^y \tilde{R}_{i,j+\frac{1}{2}}^{y \top},
		\end{equation}
		where ${\tilde{R}^d},\, d \in \{x,y\},$ are matrices of the scaled entropy right eigenvectors of the jacobian $\dfrac{\partial \f^d}{\partial \con}$, and ${\Lambda^d},\, d \in \{x,y\},$ are $6\times 6$ diagonal matrices of the form
		$$
  {\Lambda^d}=  \left( \max_{1 \leq k \leq 6} |\lambda_k^d|\right) \mathbf{I}_{6 \times 6},  \ \ \, d \in \{x,y\}.
  $$
		Here $\{\lambda_k^d: 1 \leq k \leq 6 \}$ is the set of eigenvalues of the jacobian $\frac{\partial\f^d}{\partial \con}$. The procedure to obtain the scaled right eigenvector matrices $\tilde{R}$ is given in~\cite{barth1999numerical}. We follow \cite{barth1999numerical},\cite{sen_entropy_2018} to derive expressions for the scaling matrices in Appendix~\eqref{scaledrev}.
		
		Now, with the choice of diffusion operator \eqref{diffusiontype}, the numerical scheme \eqref{scheme} with the numerical flux \eqref{es_numflux} is entropy stable.
		
		
		\subsection{Higher order entropy stable schemes}
		The entropy stable scheme \eqref{scheme} discussed above with the numerical flux \eqref{es_numflux} contains the jump terms $[\![\evar]\!]_{i+\frac{1}{2},j}$ and $[\![\evar]\!]_{i,j+\frac{1}{2}}$ which are of first order accuracy. Therefore, the resultant scheme cannot be expected to be more than first-order accurate. The natural idea to increase the order of accuracy is to approximate the jump terms using higher-order polynomial reconstructions. However, straightforward reconstruction cannot be shown to preserve entropy stability. Therefore, instead of reconstructing \rev{the} entropy variable $\evar_{i,j}$ we follow \rev{the} reconstruction procedure of \cite{fjordholm2012arbitrarily} to reconstruct the scaled entropy variables $\bm{\mathcal{{V}}}_{i,j}$, defined as
		\[\bm{\mathcal{{V}}}_{i,j}^{x,\pm}\,=\, {R}^{x^{\top}}_{i\pm\frac{1}{2},j}\evar_{i,j}.\]
		If $\bm{\mathcal{\tilde{V}}}_{i,j}^{x,\pm}$ denotes the $k$-th order reconstructed values of $\bm{\mathcal{V}}_{i,j}^{x,\pm}$ in the $x$-direction, then,
		$$\bm{\tilde{V}}_{i,j}^{x,\pm}\,=\, \left\lbrace \tilde{ R}^{x^{\top}}_{i\pm\frac{1}{2},j}\right\rbrace ^{(-1)}\bm{\mathcal{\tilde{{V}}}}_{i,j}^{x,\pm},$$
		are the corresponding $k$-th order reconstructed values for $\bm{V}_{ij}$. Hence, the modified numerical flux is given by,
		\begin{equation}
			{\hat{\fx}}^{k}_{i+\frac{1}{2},j}\,=\,{\tilde{\fx}}^{2p}_{i+\frac{1}{2},j}\,-\,\frac{1}{2}\,\textbf{D}_{i+\frac{1}{2},j}^x[\![ \bm{\tilde{V}}^x]\!]_{i+\frac{1}{2},j}^k
			\label{eq:entropy_stable_flux}
		\end{equation}
		where $[\![ \bm{\tilde{V}}^x]\!]_{i+\frac{1}{2},j}^k$ stands for,
		$$
		[\![ \bm{\tilde{V}}^x]\!]_{i+\frac{1}{2},j}^k = \bm{\tilde{V}}^{x,-}_{i+1,j}\,-\,\bm{\tilde{V}}^{x,+}_{i,j}
		$$
		and $p \in \mathbb{N}$ is chosen as
		\begin{itemize}
			\item $p=k/2$ if $k$ is even,
			\item $p=(k+1)/2$ if $k$ is odd,
		\end{itemize}
		where $k$ is the accuracy of the time integration scheme.  As in \cite{fjordholm2012arbitrarily}, the sufficient condition for the numerical flux \eqref{eq:entropy_stable_flux} to be entropy stable is that the reconstruction process for ${\bm{\mathcal{V}}}$ must satisfy the sign preserving property, i.e., the sign of the reconstructed jumps at any face must be same as the sign of the original jumps; for example, for \rev{a} reconstruction along \rev{the} $x$-direction, we need the following
		\begin{equation}\label{eq:recsign}
			\textrm{sign} \left( \bm{\mathcal{V}}^{x,-}_{i+1,j} - \bm{\mathcal{V}}^{x,+}_{i,j} \right) = \textrm{sign} \left(  \bm{\mathcal{V}}_{i+1,j} - \bm{\mathcal{V}}_{i,j} \right)
		\end{equation}
		to hold for each component.  Consequently, we use \textit{minmod} reconstruction for the second order scheme, which satisfies this property and denotes it by O2$\_$ES. Following \cite{fjordholm2013eno}, for the higher order schemes, we use the \textit{ENO} based reconstruction. In particular, for the third-order scheme, we use \rev{the} fourth-order entropy conservative flux \eqref{4thorder_x} and \rev{the} third-order $\text{ENO}$ reconstruction to obtain the expression for the $x-$directional flux as
		\begin{equation*}
			{\hat{\fx}}^{3}_{i+\frac{1}{2},j}\,=\,{\tilde{\fx}}^4_{i+\frac{1}{2},j}\,-\,\frac{1}{2}\,\textbf{D}_{i+\frac{1}{2},j}^x[\![ \bm{\tilde{V}}^x]\!]_{i+\frac{1}{2},j}^3
		\end{equation*}
		and denote it by O3$\_$ES. Similarly, for the fourth-order scheme, we use the fourth-order entropy conservative flux \eqref{4thorder_x} and a fourth-order ENO reconstruction and denote the scheme by O4$\_$ES. Note that the extension to two dimensions is \rev{straightforward}.
		
		
		\subsection{Semi-discrete entropy stability}\label{subsec:semidisscheme}
		We now proceed to show that the scheme \eqref{scheme} with \rev{the} numerical flux \eqref{eq:entropy_stable_flux} is entropy stable.
		\begin{theorem}
			The semi-discrete schemes O2\_ES, O3\_ES, and O4\_ES designed above are entropy stable, i.e., they satisfy,
			\begin{equation*}
				\frac{d}{dt}  \eta(\con_{i,j})  +\frac{1}{\Delta x} \left( \hat{q^x}_{i+\frac{1}{2},j} - \hat{q^x}_{i-\frac{1}{2},j}\right)+\frac{1}{\Delta y}\left( \hat{q^y}_{i,j+\frac{1}{2}} - \hat{q^y}_{i,j-\frac{1}{2}}\right) \le 0,
				\label{eq:semi_disc_ent_inq}
			\end{equation*}
			where $ \hat{q^x}$ and $ \hat{q^y}$ are the consistent numerical entropy fluxes.
		\end{theorem}
		\begin{proof}
			Following \cite{fjordholm2012arbitrarily,tadmor1987numerical}, we have
			\begin{align*}
				& \frac{d}{dt}  \eta(\con_{i,j})  +\frac{1}{\Delta x} \left( \hat{q^x}_{i+\frac{1}{2},j} - \hat{q^x}_{i-\frac{1}{2},j}\right)+\frac{1}{\Delta y}\left( \hat{q^y}_{i,j+\frac{1}{2}} - \hat{q^y}_{i,j-\frac{1}{2}}\right) \\
				= & - \half \left(  \bm{\mathcal{V}}_{i,j} - \bm{\mathcal{V}}_{i-1,j} \right)^\top \Lambda^x_{i-1/2,j} \left( \bm{\mathcal{V}}^{x,-}_{i,j} - \bm{\mathcal{V}}^{x,+}_{i-1,j} \right) \\
				& - \half \left(  \bm{\mathcal{V}}_{i+1,j} - \bm{\mathcal{V}}_{i,j} \right)^\top \Lambda^x_{i+1/2,j} \left( \bm{\mathcal{V}}^{x,-}_{i+1,j} - \bm{\mathcal{V}}^{x,+}_{i,j} \right) \\
				& - \half \left(  \bm{\mathcal{V}}_{i,j} - \bm{\mathcal{V}}_{i,j-1} \right)^\top \Lambda^y_{i,j-1/2} \left( \bm{\mathcal{V}}^{y,-}_{i,j} - \bm{\mathcal{V}}^{y,+}_{i,j-1} \right) \\
				& - \half \left(  \bm{\mathcal{V}}_{i,j+1} - \bm{\mathcal{V}}_{i,j} \right)^\top \Lambda^y_{i,j+1/2} \left( \bm{\mathcal{V}}^{y,-}_{i,j+1} - \bm{\mathcal{V}}^{y,+}_{i,j} \right) \\
				& -\evar_{i,j}^\top\B^x_{i,j}\left(\frac{\partial h}{\partial x}\right)_{i,j}-\evar_{i,j}^\top\B^y_{i,j}\left(\frac{\partial h}{\partial y}\right)_{i,j} \\
				\rev{= & - \half \left(  \bm{\mathcal{V}}_{i,j} - \bm{\mathcal{V}}_{i-1,j} \right)^\top \Lambda^x_{i-1/2,j} \left( \bm{\mathcal{V}}^{x,-}_{i,j} - \bm{\mathcal{V}}^{x,+}_{i-1,j} \right) \\
				& - \half \left(  \bm{\mathcal{V}}_{i+1,j} - \bm{\mathcal{V}}_{i,j} \right)^\top \Lambda^x_{i+1/2,j} \left( \bm{\mathcal{V}}^{x,-}_{i+1,j} - \bm{\mathcal{V}}^{x,+}_{i,j} \right) \\
				& - \half \left(  \bm{\mathcal{V}}_{i,j} - \bm{\mathcal{V}}_{i,j-1} \right)^\top \Lambda^y_{i,j-1/2} \left( \bm{\mathcal{V}}^{y,-}_{i,j} - \bm{\mathcal{V}}^{y,+}_{i,j-1} \right) \\
				& - \half \left(  \bm{\mathcal{V}}_{i,j+1} - \bm{\mathcal{V}}_{i,j} \right)^\top \Lambda^y_{i,j+1/2} \left( \bm{\mathcal{V}}^{y,-}_{i,j+1} - \bm{\mathcal{V}}^{y,+}_{i,j} \right)\;\; (\text{Using } \evar_{i,j}^\top\B^x_{i,j}=\evar_{i,j}^\top\B^y_{i,j}=0)}
			\end{align*}
			\rev{Using the sign property~\eqref{eq:recsign} of the reconstruction process, the jumps in scaled entropy variables and their reconstructed jumps have the same signs. Also, matrices $\Lambda^x$ are $\Lambda^y$ are diagonal with positive entry. Hence, each term on the right side of the above equality is negative. This results in the inequality \eqref{eq:semi_disc_ent_inq}.}
		\end{proof}
		
		The general semi-discrete finite difference scheme for the system \eqref{eq:ssw1} has the following form,
		\begin{eqnarray}
			\frac{d }{dt}\con_{i,j}(t)
			+
			\frac{1}{\Delta x}\left({\fx}_{i+\frac{1}{2},j}(t)-{\fx}_{i-\frac{1}{2},j}(t)\right)\nonumber \\
			+
			\frac{1}{\Delta y}  \left({\fy}_{i,j+\frac{1}{2}}(t)-{\fy}_{i,j-\frac{1}{2}}(t)\right)+\B^{nc}_{i,j}(\con)=\s_{i,j}, \label{scheme1}
		\end{eqnarray}
		where $\s_{i,j}=\s(\con_{i,j})$. Then, we have the following remark:
		
		\begin{remark}
			The semi-discrete scheme \eqref{scheme1} with the numerical flux \eqref{eq:entropy_stable_flux} satisfies \rev{the} following inequality,
			\begin{equation*}
				\frac{d}{dt}  \eta(\con_{i,j})  +\frac{1}{\Delta x} \left( \hat{q^x}_{i+\frac{1}{2},j} - \hat{q^x}_{i-\frac{1}{2},j}\right)+\frac{1}{\Delta y}\left( \hat{q^y}_{i,j+\frac{1}{2}} - \hat{q^y}_{i,j-\frac{1}{2}}\right) \le 4 \alpha_{i,j}|\bm{v}_{i,j}|^3.
			\end{equation*}
			Following \cite{fjordholm2012arbitrarily,tadmor1987numerical}, we have
			\rev{\begin{align*}
				& \frac{d}{dt}  \eta(\con_{i,j})  +\frac{1}{\Delta x} \left( \hat{q^x}_{i+\frac{1}{2},j} - \hat{q^x}_{i-\frac{1}{2},j}\right)+\frac{1}{\Delta y}\left( \hat{q^y}_{i,j+\frac{1}{2}} - \hat{q^y}_{i,j-\frac{1}{2}}\right) \\
				\le & -\bm{V}_{i,j}\top\B^x_{i,j}\left(\frac{\partial h}{\partial x}\right)_{i,j}-\evar_{i,j}^\top\B^y_{i,j}\left(\frac{\partial h}{\partial y}\right)_{i,j}+\evar_{i,j}^\top\s_{i,j}(\con), & \\
			\rev{	\le & \evar_{i,j}^\top\s_{i,j}(\con)=4 \alpha_{i,j}|\bm{v}_{i,j}|^3.}
			\end{align*}}
		\end{remark}

\section{Fully discrete scheme}
		\label{sec:fulld}
		Let the initial time be $t^0$ and let $\con^n$ be the discrete solution at time $t^n$. The semi-discrete scheme~\eqref{scheme} can be expressed as
		\begin{equation}
			\frac{d }{dt}\con_{i,j}(t)
			=
			\mathcal{L}_{i,j}(\con(t)) -{\B^x(\con_{i,j}(t))}\left(\frac{\partial h}{\partial x}\right)_{i,j}-{\B^y(\con_{i,j}(t))}\left(\frac{\partial h}{\partial y}\right)_{i,j}+\s(\con_{i,j}(t))
			\label{fullydiscrete}
		\end{equation}
		where,
		\begin{equation*}
			\mathcal{L}_{i,j}(\con(t))
			=
			- \frac{1}{\Delta x}  \left(\mathbf{F}_{i+\frac{1}{2},j}^x(t)-\mathbf{F}_{i-\frac{1}{2},j}^x(t)\right)
			-
			\frac{1}{\Delta y}  \left(\mathbf{F}_{i,j+\frac{1}{2}}^y(t)-\mathbf{F}_{i,j-\frac{1}{2}}^y(t)\right).
		\end{equation*}
		The spatial derivatives are approximated using central differencing of suitable order (see section \eqref{sec:res}). The system of ODE~\eqref{fullydiscrete} can be integrated in time in several ways and we use explicit time discretization.
		%
		
		\subsection{Explicit schemes}
		We use explicit strong stability preserving Runge Kutta (SSP-RK) methods explained in \cite{gottlieb2001strong} for the time discretization of the SSW model. The second and third-order accurate SSP-RK schemes have the following structure for one time step.
		
		\begin{enumerate}
			\item Set $\con^{0} \ = \ \con^n$.
			\item For $k$ in $\{1,\dots,m+1\}$, compute
			{\small
			\begin{eqnarray*}
				\con_{i,j}^{(k)} \
				= \
				\sum_{l=0}^{k-1}\gamma_{kl}\con_{i,j}^{(l)}
				+
				\delta_{kl}\Delta t \Big(\mathcal{L}_{i,j}(\con^{(l)})-{\B^x}(\con_{i,j}^{(l)})\left(\frac{\partial h}{\partial x}\right)_{i,j}-{\B^y(\con_{i,j}^{(l)})}\left(\frac{\partial h}{\partial y}\right)_{i,j} + \s(\con_{i,j}^{(l)}) \Big),
			\end{eqnarray*}}
			where $\gamma_{kl}$ and $\delta_{kl}$ are given in Table~\eqref{table:ssp}.
			\item Finally, $\con_{i,j}^{n+1} \ =  \ \con_{i,j}^{(m+1)}$.
		\end{enumerate}
		
		\begin{table}[h]
			\caption[h]{Coefficients for Explicit SSP-Runge-Kutta time stepping:}
			\centering
			\begin{tabular}{ l|ccc|ccc }
				\hline
				Order              & \multicolumn{3}{|c|}{$\gamma_{il}$} & \multicolumn{3}{|c}{$\delta_{il}$}                       \\
				\hline
				\multirow{2}{*}{2} & 1                                   &                                   &     & 1 &     &     \\
				& 1/2                                 & 1/2                               &     & 0 & 1/2 &     \\
				\hline
				\multirow{3}{*}{3} & 1                                   &                                   &     & 1 &     &     \\
				& 3/4                                 & 1/4                               &     & 0 & 1/4 &     \\
				& 1/3                                 & 0                                 & 2/3 & 0 & 0   & 2/3 \\
				\hline
			\end{tabular}
			\label{table:ssp}
		\end{table}
		The fourth order RK-SSP scheme \cite{gottlieb2001strong} has the following structure:
		\begin{subequations}
			\begin{align*}
				\con^{(1)} &= \textbf{U}^n + 0.39175222700392 \Delta t \big(\mathcal{M}(\con^n) \big) \\
				\con^{(2)} &= 0.44437049406734 \con^n + 0.55562950593266 \con^{(1)}+0.36841059262959  \Delta t \big( \mathcal{M}(\con^1) \big) \\
				\con^{(3)} &= 0.62010185138540 \con^n + 0.37989814861460 \con^{(2)} +0.25189177424738  \Delta t \big( \mathcal{M}(\con^2) \big) \\
				\con^{(4)} &= 0.17807995410773 \con^n + 0.82192004589227 \con^{(3)} + 0.54497475021237  \Delta t \big( \mathcal{M}(\con^3) \big) \\
				\con^{n+1} &= 0.00683325884039 \con^n + 0.51723167208978 \con^{(2)} + 0.12759831133288 \con^{(3)}\\
				&+ 0.34833675773694 \con^{(4)}+ 0.08460416338212  \Delta t \big( \mathcal{M}(\con^3) \big)\\
				&+ 0.22600748319395  \Delta t \big( \mathcal{M}(\con^4) \big).
			\end{align*}
		\end{subequations}
		where $\mathcal{M}(\con^n)=\mathcal{L}(\con^n)-{\B^x}(\con^n)\left(\frac{\partial h}{\partial x}\right)^n - {\B^y(\con^n)}\left(\frac{\partial h}{\partial y}\right)^n + \s(\con^n)$. Here, we have ignored the subscripts $\{i,j\}$.

		\section{Numerical results}\label{sec:res}
		We test the fully discrete schemes on some 1-D and 2-D test cases and present the results for O1\_ES, O2\_ES, O3\_ES, and O4\_ES schemes. Here,
		\begin{enumerate}
			\item O1\_ES: the Euler time-stepping with first-order entropy stable flux and second-order central difference approximation for the derivatives in the non-conservative terms.
			\item O2\_ES: the explicit second-order scheme with second-order entropy stable flux and second-order central difference approximation for the derivatives in the non-conservative terms.
			\item O3\_ES: the third-order explicit scheme with third-order entropy stable flux and fourth-order central difference approximation for the derivatives in the non-conservative terms.
			\item O4\_ES: the fourth-order explicit SSP RK scheme with fourth-order entropy stable flux and fourth-order central difference approximation for the derivatives in the non-conservative terms.
		\end{enumerate}
		\rev{We take the acceleration due to gravity as $g = 9.81 \ m/s^2$. To compute the time step, we use 
  $$
 \Delta t =\text{CFL}\frac{1}{\max_{ij} \left( \frac{|\lambda^x(\con^n_{i,j})|}{\Delta x} + \frac{|\lambda^y(\con^n_{i,j})|}{\Delta y} \right)},
  $$
from \cite{Chandrashekar2020}. Here $\lambda^x$ and $\lambda^y$ are the maximum eigenvalues in $x$ and $y$ directions, respectively. We take CFL to be $0.45$.}

\rev{For the Riemann problem test, we consider the Neumann boundary conditions at both boundaries. In effect, we copy the value in the last cell to the ghost cells. The final time is chosen in all the Riemann problem test cases so the waves do not reach the boundary.}  

		We set the source term $\s$ to be zero for all the test cases except for the 1-D roll wave test in Section~\ref{test11} and the 2-D roll wave test in Section~\ref{test13}.

		
		\subsection{One-dimensional test problems}
		\subsubsection{Accuracy test}\label{test:accuracy}
		We consider the shear shallow water model without source term $(\s=0)$ but instead, add an artificial source term $\mathcal{S}$ so that we can manufacture an exact solution. Following~\cite{biswas2021entropy}, we add the forcing term $\mathcal{S}(x,t)$ in the right-hand side of the SSW model as follows,
		\begin{equation*}
			\frac{\partial \con}{\partial t}+\frac{\partial\fx}{\partial x} +\B^x\frac{\partial h}{\partial x}= \mathcal{S}(x,t),
		\end{equation*}
		where,
		\begin{eqnarray*}
			\mathcal{S}(x,t)=\left(0,2\pi\cos(2\pi(x-t))(1+2 g+g\sin(2 \pi (x-t))), 0, \right.\\ \left. 2\pi\cos(2\pi(x-t))(1+2 g+g\sin(2 \pi (x-t))),0,0 \right)^\top.
		\end{eqnarray*} The computational domain is $[-0.5,\,0.5]$ with periodic boundary conditions. The exact solution is given by
		\begin{equation*}
			h(x,t)=2+\sin(2\pi (x-t)),\qquad  v_1(x,t)=1, \qquad v_2(x,t)=0,
		\end{equation*}
		\begin{equation*}
			\p_{11}(x,t)=\p_{22}(x,t)=1, \qquad \p_{12}(x,t)=0.
		\end{equation*}
		The computations are performed up to the final time $T=0.5$.
		\begin{table}[ht]
			\centering
			\begin{tabular}{c|c|c|c|c|c|c|}
				\hline Number of cells  & \multicolumn{2}{|c}{{O2\_ES}} &  \multicolumn{2}{|c}{O3\_ES} & \multicolumn{2}{|c}{O4\_ES}  \\
				\hline   & $L^1$ error  &  Order &  $L^1$ error      & Order & $L^1$ error      & Order \\
				\hline 50 & 4.58e-03 & -- & 2.26e-04 & -- & 1.92e-05 & -- \\
				100 & 1.39e-03 & 1.72 & 2.92e-05 & 2.94 &1.56e-06 & 3.62 \\
				200 & 4.67e-04 & 1.57 & 3.70e-06 & 2.98 & 1.14e-07 & 3.77 \\
				400 & 1.35e-04 & 1.79 & 4.63e-07 & 2.99 & 7.83e-09 & 3.86 \\
				800 & 3.67e-05 & 1.88 & 5.80e-08 & 2.99 & 5.32e-10 & 3.88 \\
				1600 & 9.71e-06 & 1.92 & 7.25e-09 & 2.99 & 4.17e-11 & 3.68 \\
				\hline
			\end{tabular}
			\caption{Accuracy test: $L^1$ errors and order of accuracy for the water depth $h$.}
			\label{tab:acc2}
		\end{table}
		We present the $L^1$ errors and order of accuracy for the water depth $h$ in Table~\ref{tab:acc2}  using the schemes O2\_ES, O3\_ES, and O4\_ES. We observe that the schemes have reached the designed order of accuracy.

		
		\subsubsection{Dam break problem}\label{test2}
		This is a Riemann problem from \cite{Nkonga2022}, which models a dam break problem. The domain is taken to be $[-0.5,0.5]$ with Neumann boundary conditions. The initial discontinuity is placed at $x=0$, and the initial conditions are given by
		{\small
		\begin{equation*}
			(h,\ v_1, v_2, \ \p_{11}, \p_{12}, \p_{22}) = \begin{cases}
				\big(0.02 ,\  0,\ 0, \ 4.0\times10^{-2}, \ 0, \ 4.0\times10^{-2} \big)  & \text{if } x < 0.0, \\
				\big(0.01,\ 0, \ 0, \ 4.0\times10^{-2}, \ 0, \ 4.0\times10^{-2} \big) & \text{if } x > 0.0.   \end{cases}
		\end{equation*}}
		The computations are performed up to the final time $T=0.5.$ The numerical solutions for the schemes O1\_ES, O2\_ES, O3\_ES, and O4\_ES at 500 and 2000 cells are presented in Fig.~\ref{fig:test2a}.
		We have plotted the water depth $h$, velocity $v_1$ and stress component $\p_{11}$. The numerical solution has been compared with the exact solution given in~\cite{Nkonga2022}. We can observe the convergence of the schemes. The result in Fig.~\ref{fig:test2b} \rev{shows} the entropy decay of the proposed numerical scheme.
		\begin{figure}
			\centering
			\begin{subfigure}[b]{0.45\textwidth}
				\includegraphics[width=\textwidth]{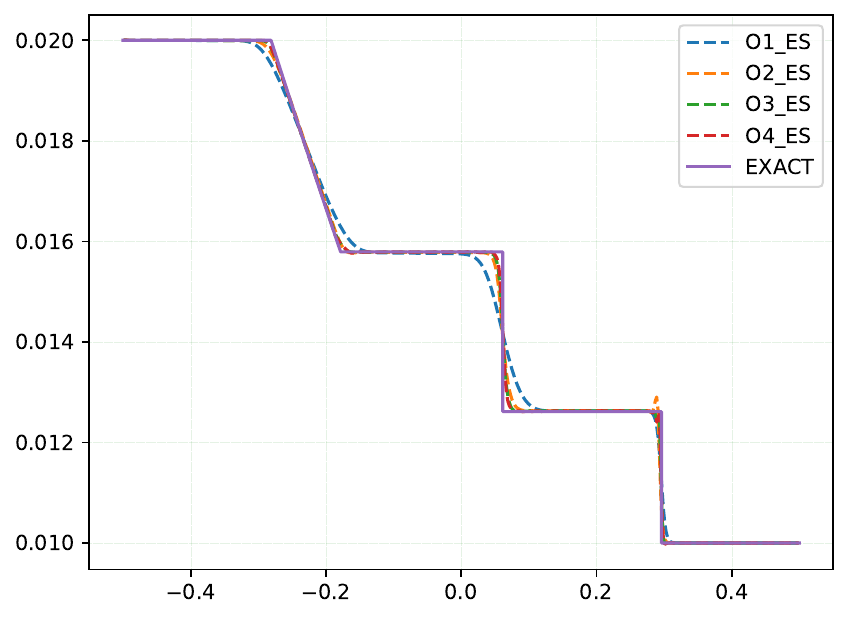}
				\caption{$h,~\text{500 cells}$}
			\end{subfigure}
			\begin{subfigure}[b]{0.45\textwidth}
				\includegraphics[width=\textwidth]{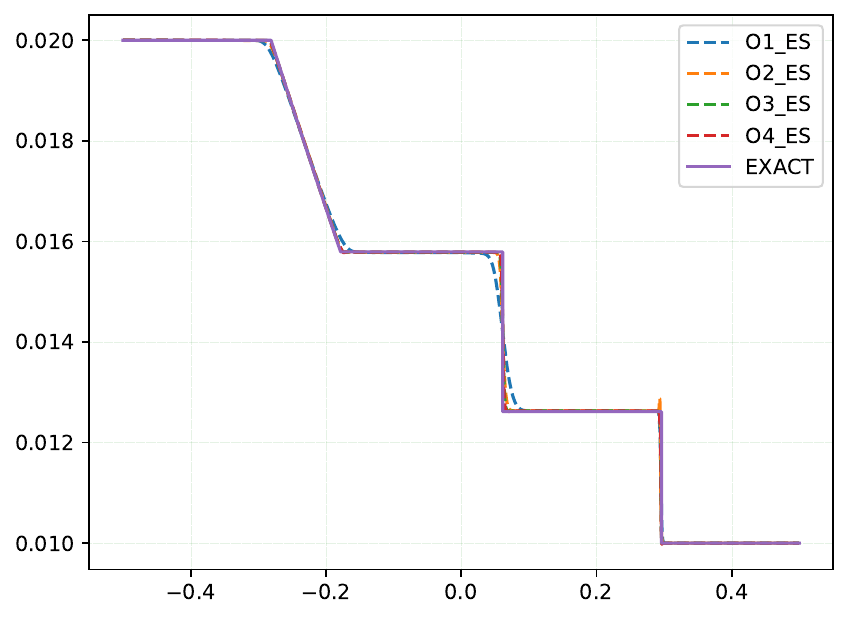}
				\caption{$h,~\text{2000 cells}$}
			\end{subfigure}
			\begin{subfigure}[b]{0.45\textwidth}
				\includegraphics[width=\textwidth]{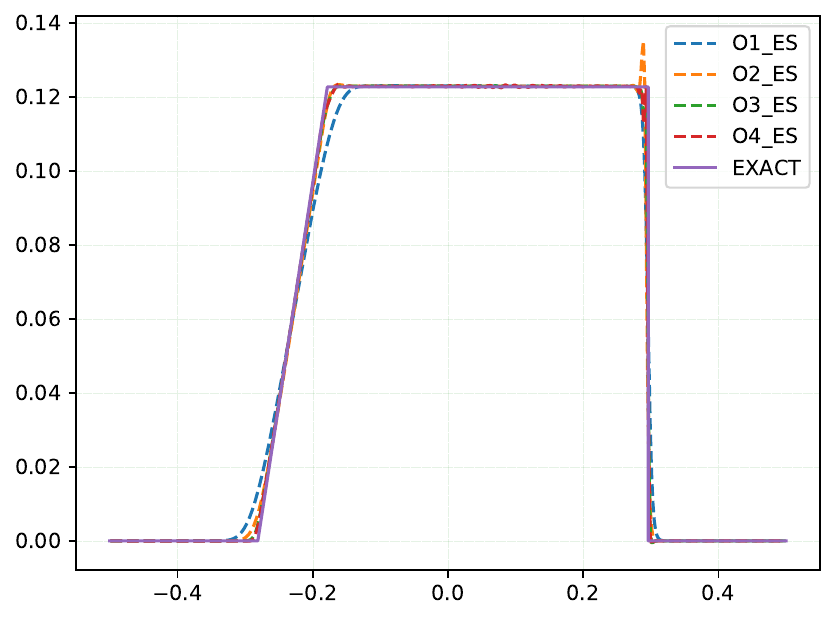}
				\caption{$v_1,~\text{500 cells}$}
			\end{subfigure}
			\begin{subfigure}[b]{0.45\textwidth}
				\includegraphics[width=\textwidth]{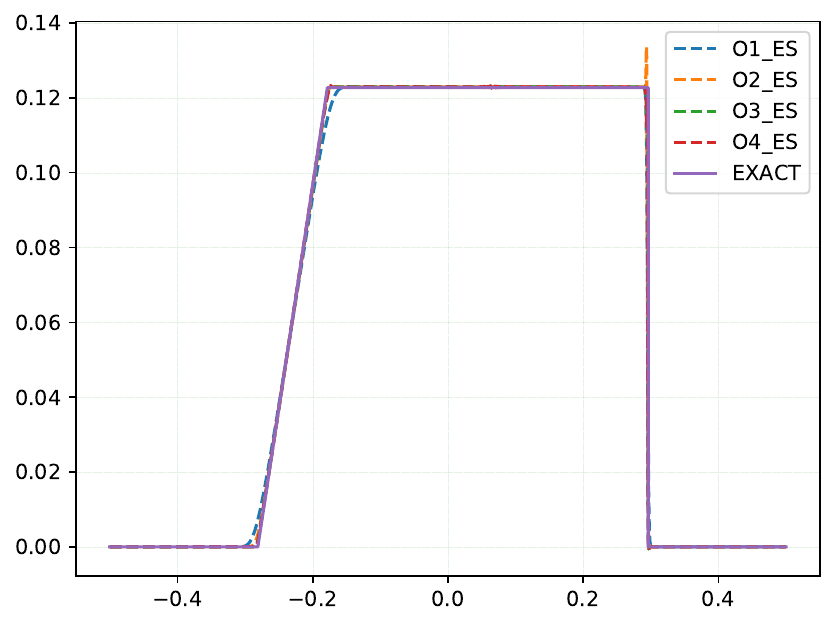}
				\caption{$v_1,~\text{2000 cells}$}
			\end{subfigure}
			\begin{subfigure}[b]{0.45\textwidth}
				\includegraphics[width=\textwidth]{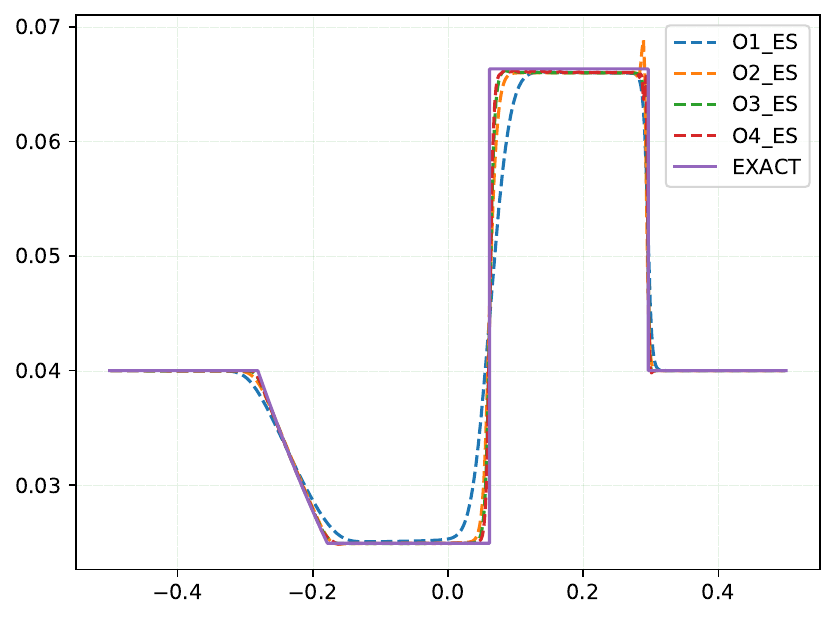}
				\caption{$\p_{11},~\text{500 cells}$}
			\end{subfigure}
			\begin{subfigure}[b]{0.45\textwidth}
				\includegraphics[width=\textwidth]{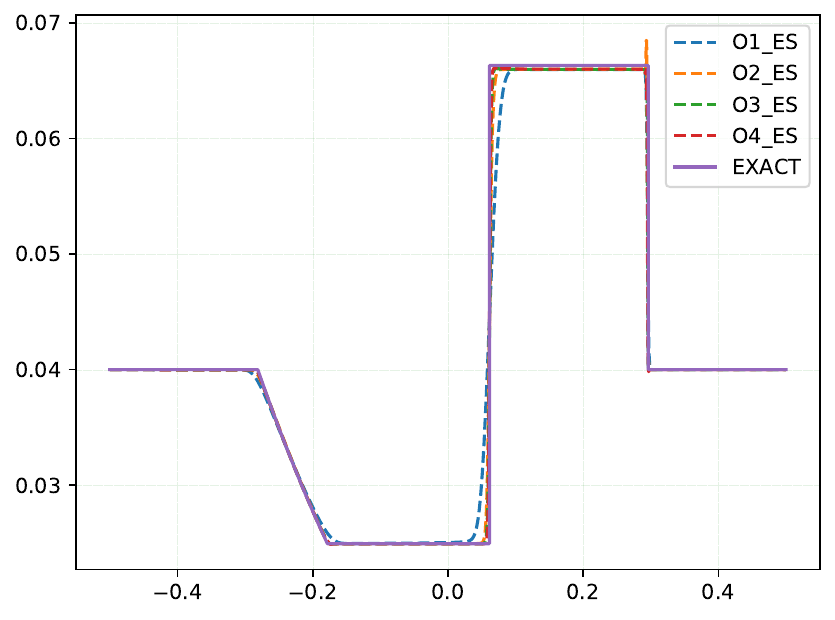}
				\caption{$\p_{11},~\text{2000 cells}$}
			\end{subfigure}
			\caption{\nameref{test2}: Plot of water depth $h$, velocity $v_1$ and stress tesnor component $\p_{11}$ using 500 and 2000 cells.}
			\label{fig:test2a}
		\end{figure}
		\begin{figure}
			\centering
			\begin{subfigure}[b]{0.45\textwidth}
				\includegraphics[width=\textwidth]{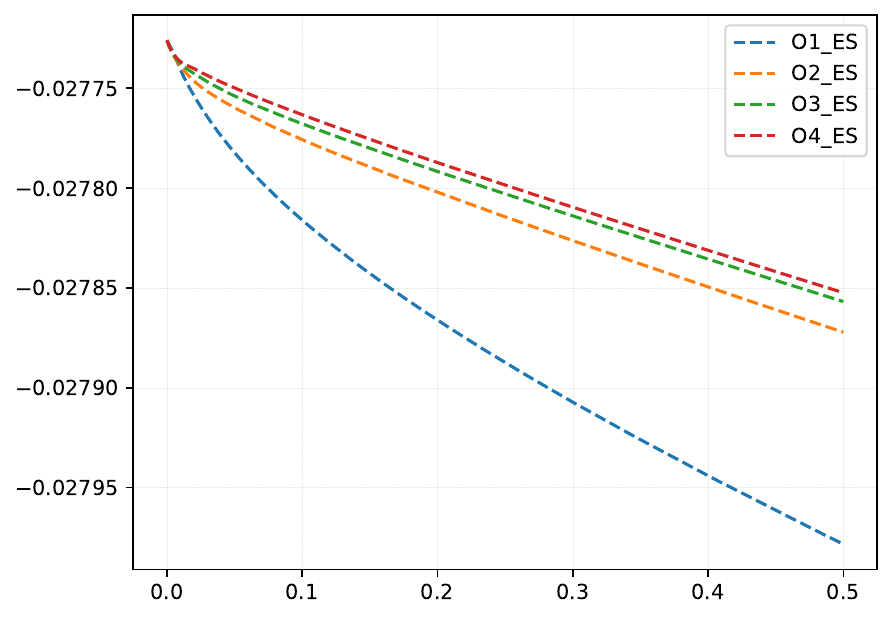}
				\caption{Entropy decay,~$\text{500 cells}$}
			\end{subfigure}
			\caption{\nameref{test2}: Plot of entropy decay using 500 cells.}
			\label{fig:test2b}
		\end{figure}
		

	Next, we test another dam break problem~\cite{Nkonga2022}, where $\p_{12}$ is set to be $10^{-8}$, and the other initial conditions are kept \rev{the} same. The numerical solutions are presented in Fig.~\ref{fig:test3a} and Fig.~\ref{fig:test3b} using 500 and 2000 cells. In this test problem, along with the water depth $h$, velocity $v_1$, stress component $\p_{11}$, we have also plotted the stress component $\p_{12}$. The $\p_{12}$ profile is able to capture all the five waves of the SSW model. The numerical solution has been compared with the exact solution from~\cite{Nkonga2022}, and we note that all the schemes converge towards the exact solution. The result in Fig.~\ref{fig:test3c} shows the entropy decay for the different numerical schemes using 500 cells; all schemes show monotonic decay of total entropy, with higher-order schemes showing smaller decay.
	\begin{figure}
		\centering
		\begin{subfigure}[b]{0.45\textwidth}
			\includegraphics[width=\textwidth]{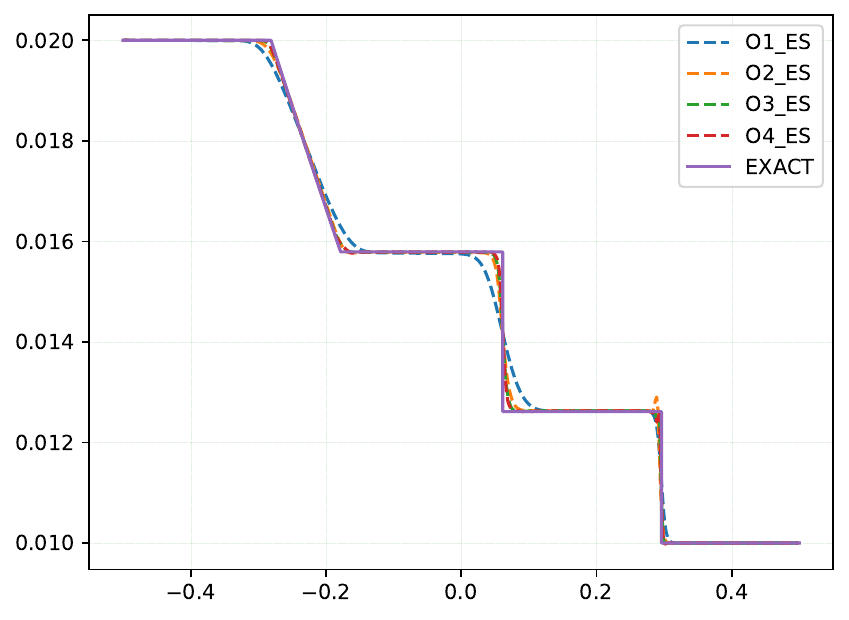}
			\caption{$h,~\text{500 cells}$}
		\end{subfigure}
		\begin{subfigure}[b]{0.45\textwidth}
			\includegraphics[width=\textwidth]{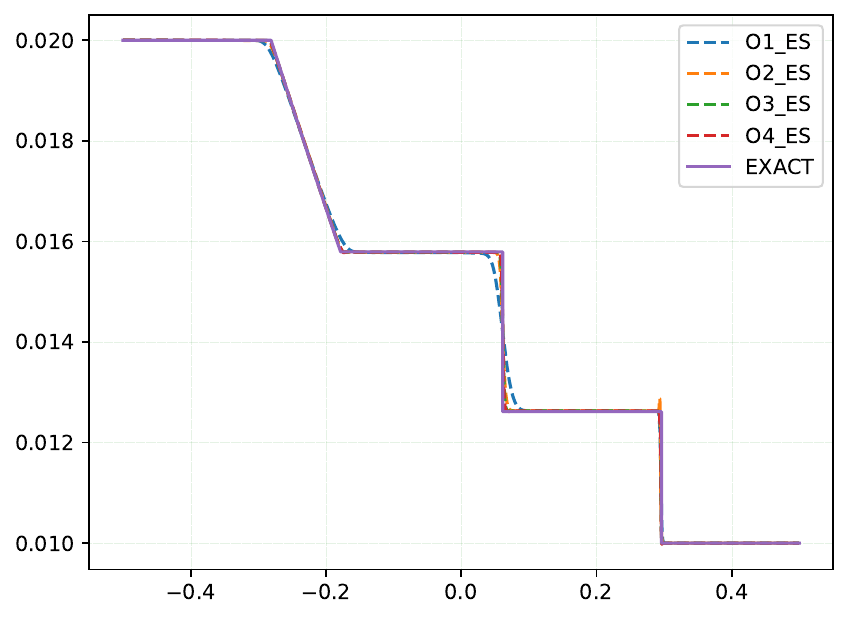}
			\caption{$h,~\text{2000 cells}$}
		\end{subfigure}
		\begin{subfigure}[b]{0.45\textwidth}
			\includegraphics[width=\textwidth]{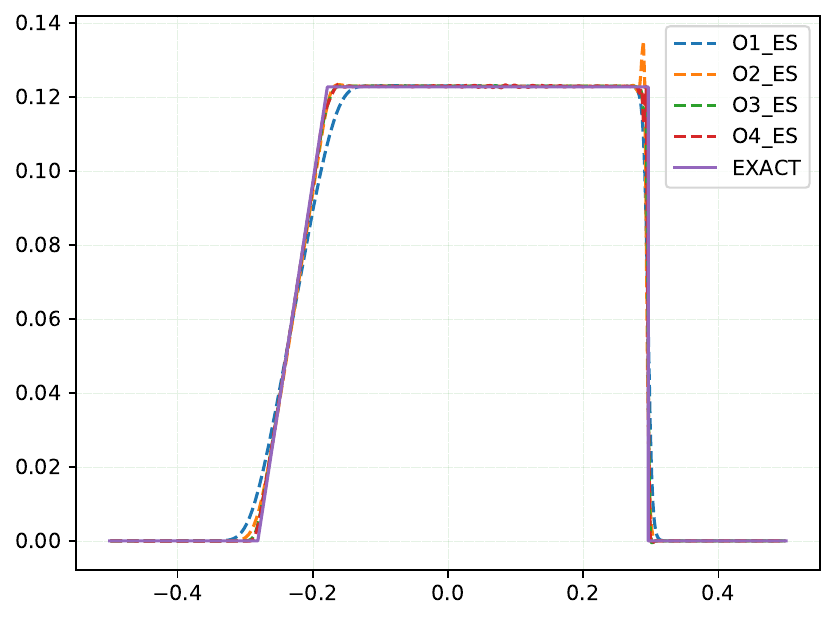}
			\caption{$v_1,~\text{500 cells}$}
		\end{subfigure}
		\begin{subfigure}[b]{0.45\textwidth}
			\includegraphics[width=\textwidth]{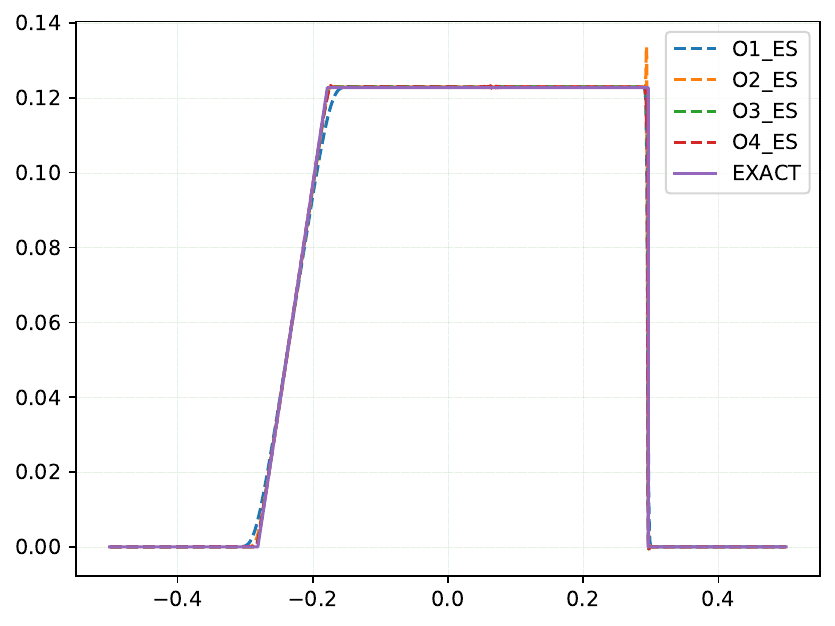}
			\caption{$v_1,~\text{2000 cells}$}
		\end{subfigure}
	\begin{subfigure}[b]{0.45\textwidth}
		\includegraphics[width=\textwidth]{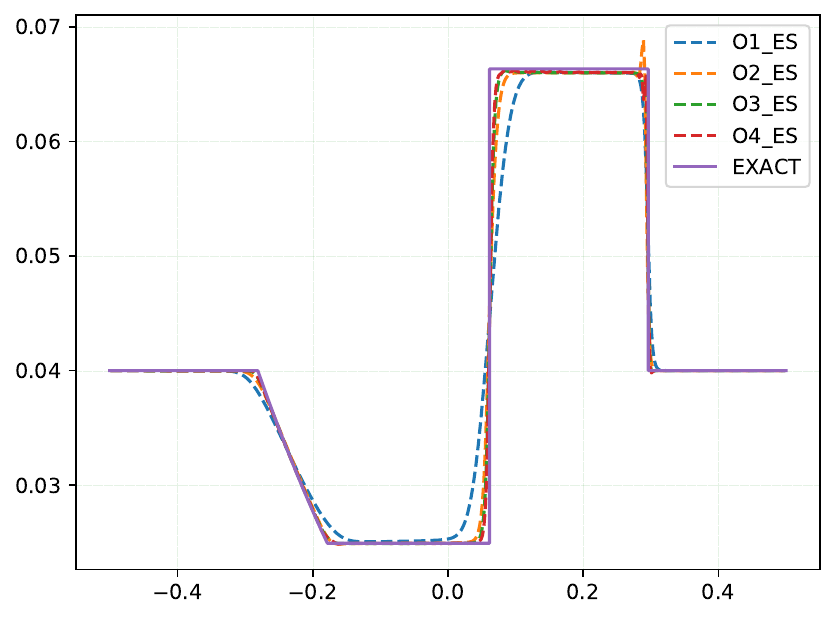}
		\caption{$\p_{11},~\text{500 cells}$}
	\end{subfigure}
	\begin{subfigure}[b]{0.45\textwidth}
		\includegraphics[width=\textwidth]{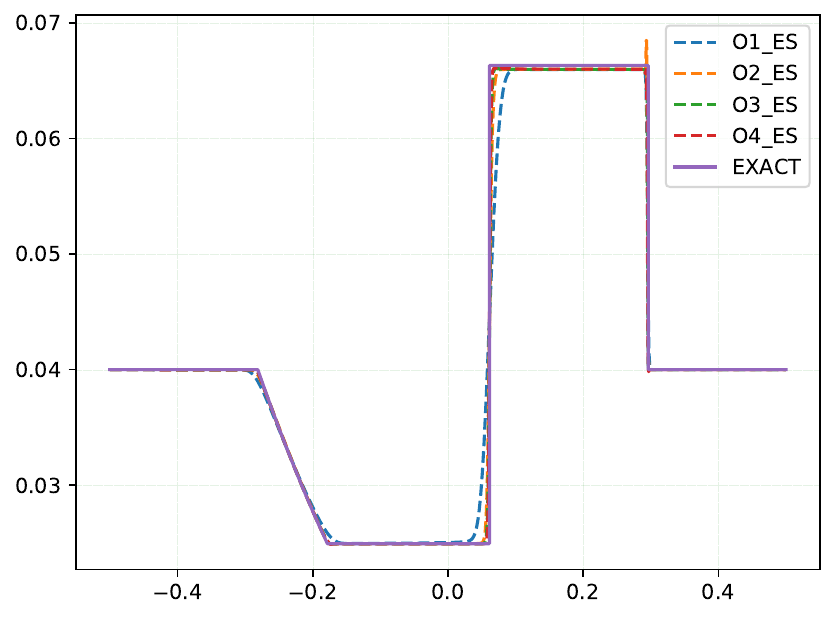}
		\caption{$\p_{11},~\text{2000 cells}$}
	\end{subfigure}
	\caption{\nameref{test2}: Plot of water depth $h$, velocity components $v_1$,  stress components $\p_{11}$ using 500 and 2000 cells.}
	\label{fig:test3a}
\end{figure}
\begin{figure}
	\centering
	\begin{subfigure}[b]{0.45\textwidth}
		\includegraphics[width=\textwidth]{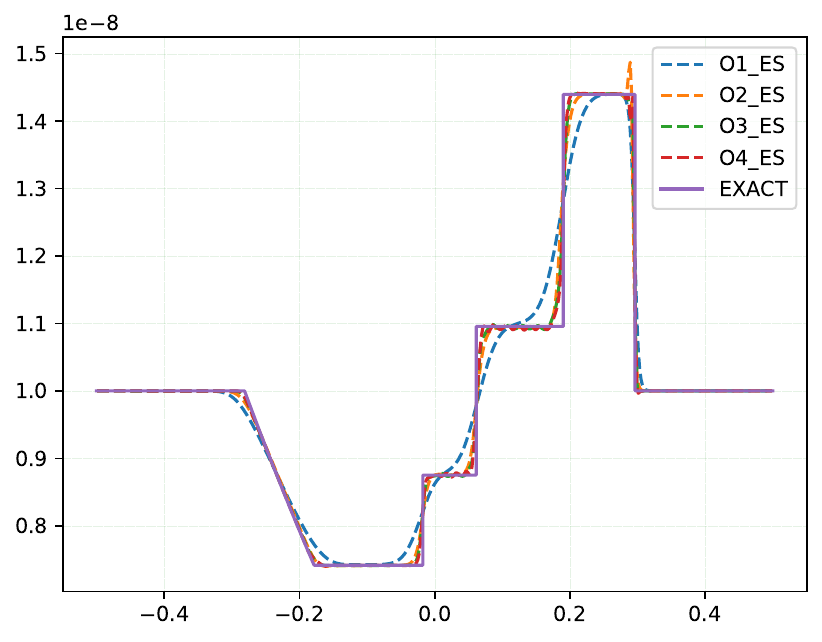}
		\caption{$\p_{12},~\text{500 cells}$}
	\end{subfigure}
	\begin{subfigure}[b]{0.45\textwidth}
		\includegraphics[width=\textwidth]{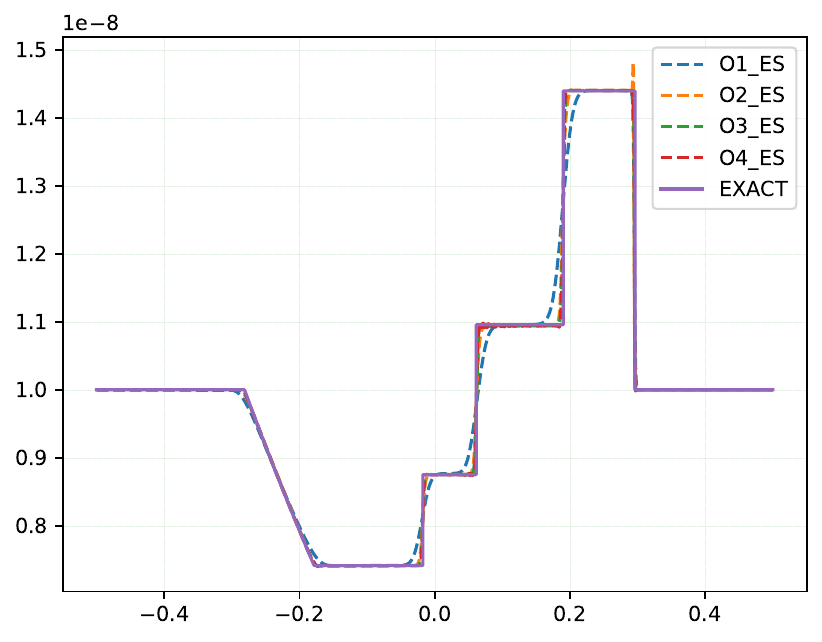}
		\caption{$\p_{12},~\text{2000 cells}$}
	\end{subfigure}
	\caption{\nameref{test2}: Plot of stress components $\p_{12}$ using 500 and 2000 cells.}
	\label{fig:test3b}
\end{figure}
\begin{figure}
	\centering
	\begin{subfigure}[b]{0.45\textwidth}
		\includegraphics[width=\textwidth]{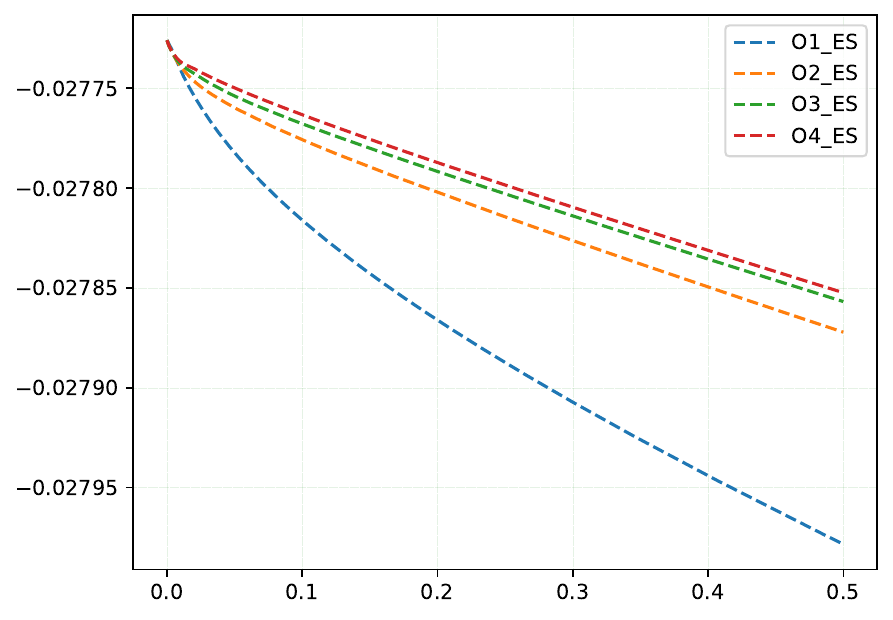}
		\caption{Entropy evolution, 500 cells}
	\end{subfigure}
	\caption{\nameref{test2}: Plot of entropy decay using 500 cells.}
	\label{fig:test3c}
\end{figure}

\subsubsection{Five wave dam break problem}\label{test6}
This is a Riemann problem~\cite{Chandrashekar2020,Nkonga2022}, which gives rise to all five waves in the solution. The computational domain is $[-0.5,0.5]$ with Neumann boundary conditions. The initial discontinuity is placed at $x=0$, and initial conditions are given by
{\small
\begin{equation*}
	(h,\ v_1, v_2, \ \p_{11}, \p_{12}, \p_{22}) = \begin{cases}
		\big(0.01 ,\  0.1,\ 0.2, \ 4.0\times10^{-2}, \ 10^{-8}, \ 4.0\times10^{-2} \big)  & \text{if } x < 0.0, \\
		\big(0.02,\ 0.1, \ -0.2, \ 4.0\times10^{-2}, \ 10^{-8}, \ 4.0\times10^{-2} \big) & \text{if } x > 0.0.   \end{cases}
\end{equation*}}
The numerical solutions are computed up to the final time $T=0.5$.
\begin{figure}
	\centering
	\begin{subfigure}[b]{0.45\textwidth}
		\includegraphics[width=\textwidth]{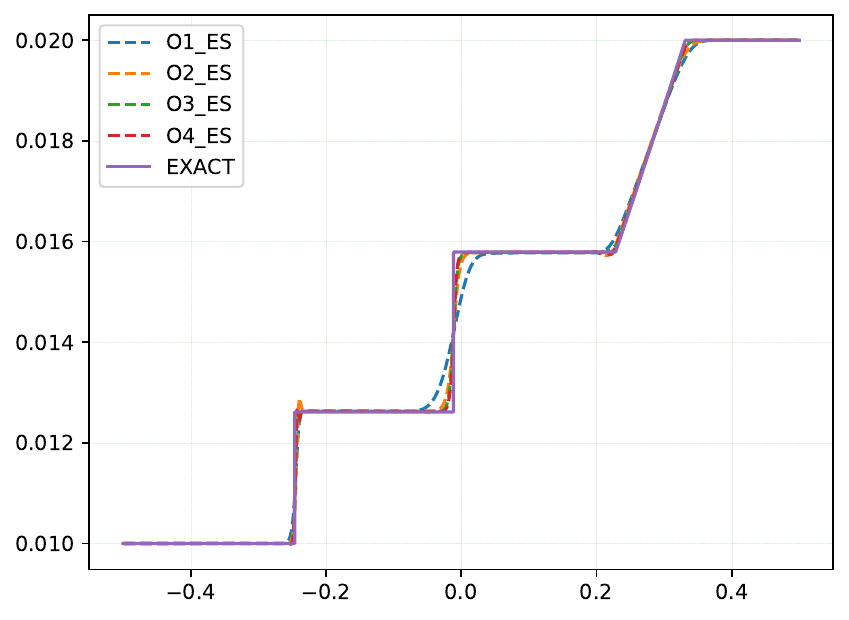}
		\caption{$h,~\text{200 cells}$}
	\end{subfigure}
	\begin{subfigure}[b]{0.45\textwidth}
		\includegraphics[width=\textwidth]{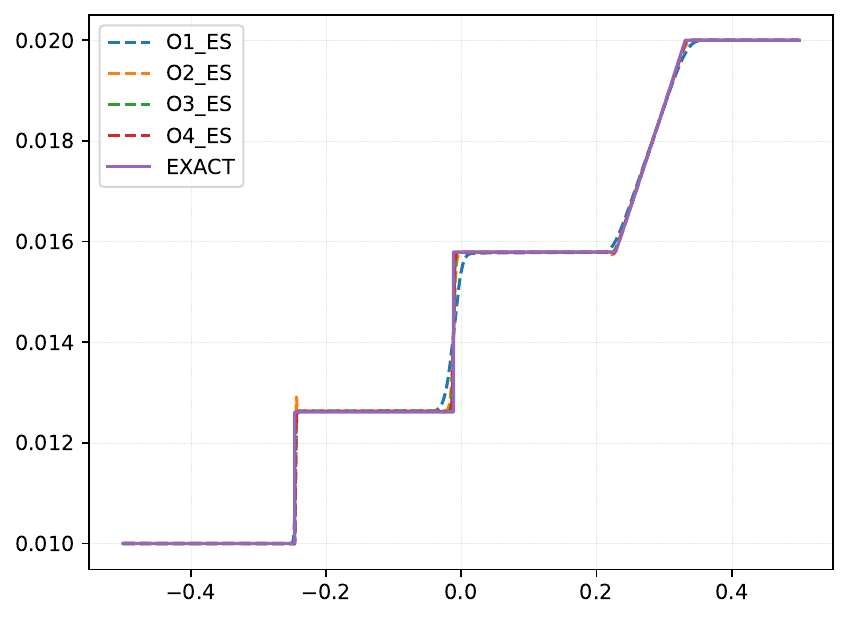}
		\caption{$h,~\text{2000 cells}$}
	\end{subfigure}
	\begin{subfigure}[b]{0.45\textwidth}
		\includegraphics[width=\textwidth]{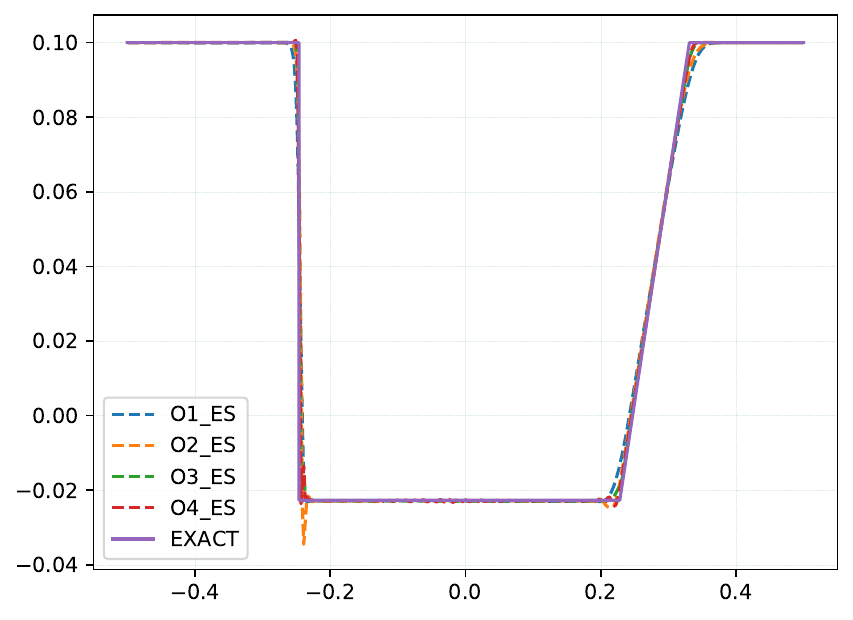}
		\caption{$v_1,~\text{200 cells}$}
	\end{subfigure}
	\begin{subfigure}[b]{0.45\textwidth}
		\includegraphics[width=\textwidth]{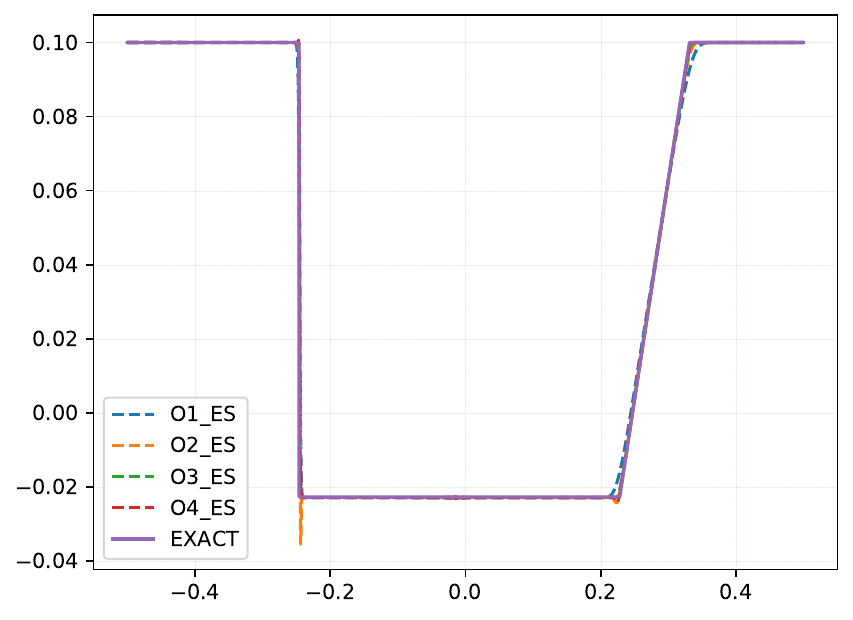}
		\caption{$v_1,~\text{2000 cells}$}
	\end{subfigure}
	\begin{subfigure}[b]{0.45\textwidth}
		\includegraphics[width=\textwidth]{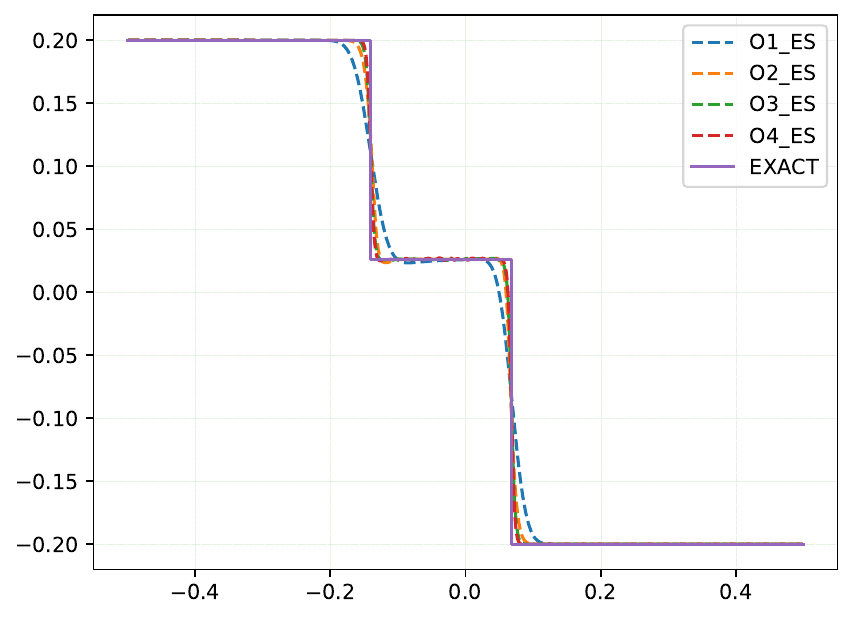}
		\caption{$v_2,~\text{200 cells}$}
	\end{subfigure}
	\begin{subfigure}[b]{0.45\textwidth}
		\includegraphics[width=\textwidth]{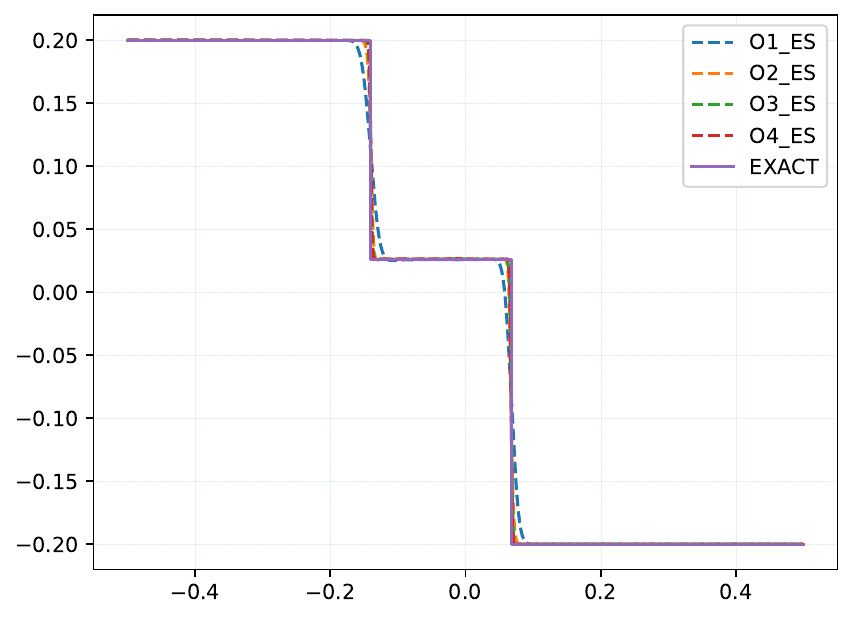}
		\caption{$v_2,~\text{2000 cells}$}
	\end{subfigure}
	\caption{\nameref{test6}: Plot of water depth $h$ and velocity components $v_1, ~v_2$ using 200 cells and 2000 cells.}
	\label{fig:test6a}
\end{figure}
We have plotted all the primitive variables in Fig.~\ref{fig:test6a} and Fig.~\ref{fig:test6b} obtained using 200 and 2000 cells. The numerical solutions have been compared with the exact solution~\cite{Nkonga2022}. We observe that the schemes O1\_ES, O2\_ES, O3\_ES, and O4\_ES converge toward \rev{the} exact solution. The result in Fig.~\ref{fig:test6c} shows the entropy decay for the different numerical schemes at 500 cells, which shows monotonic decay with time.
\begin{figure}
	\centering
	\begin{subfigure}[b]{0.45\textwidth}
		\includegraphics[width=\textwidth]{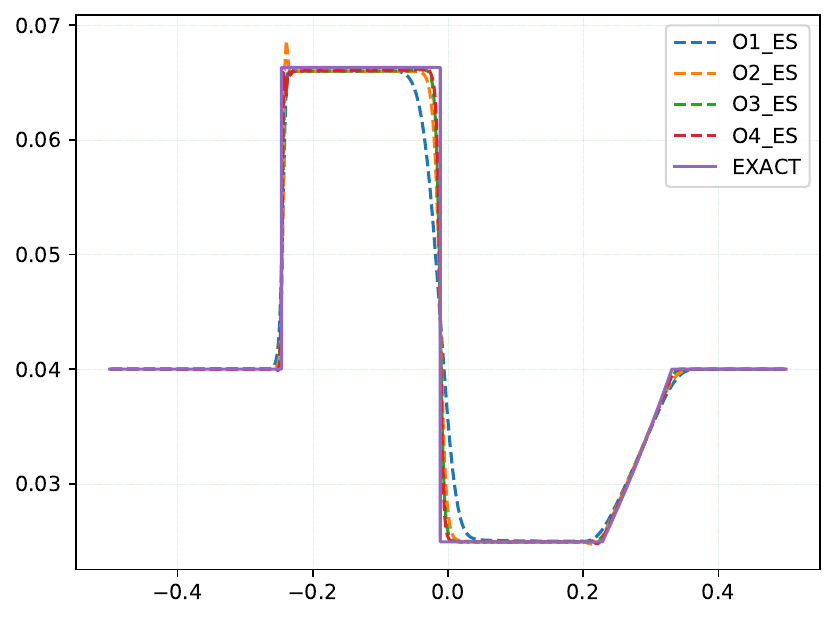}
		\caption{$\p_{11},~\text{200 cells}$}
	\end{subfigure}
	\begin{subfigure}[b]{0.45\textwidth}
		\includegraphics[width=\textwidth]{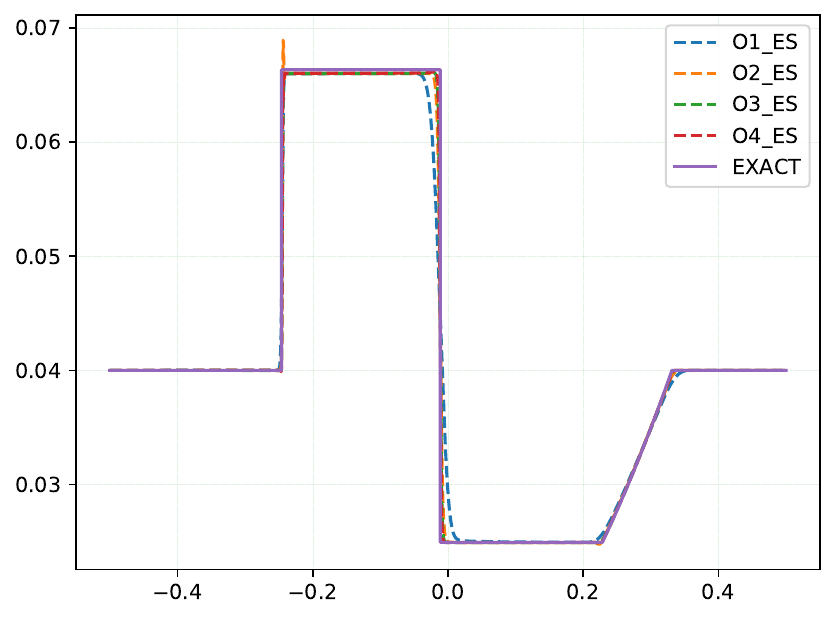}
		\caption{$\p_{11},~\text{2000 cells}$}
	\end{subfigure}
	\begin{subfigure}[b]{0.45\textwidth}
		\includegraphics[width=\textwidth]{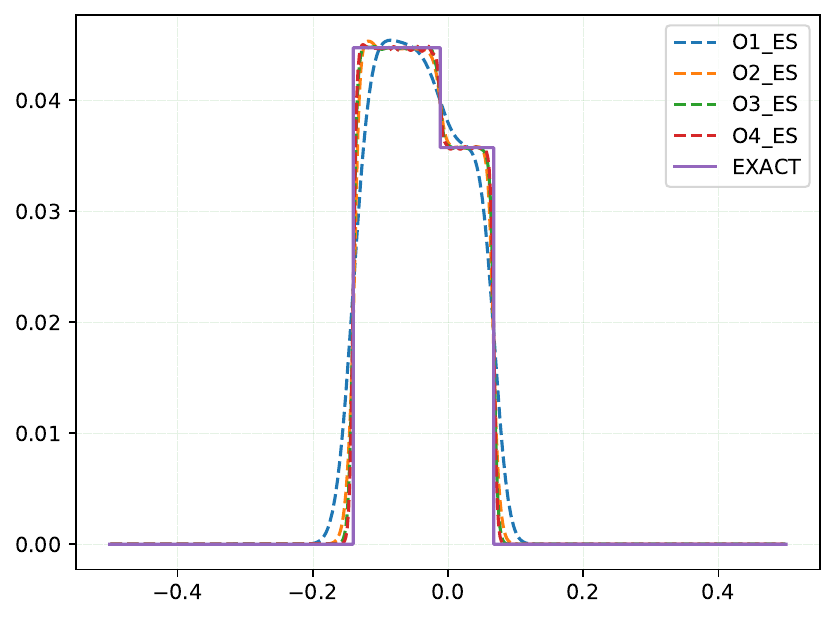}
		\caption{$\p_{12},~\text{200 cells}$}
	\end{subfigure}
	\begin{subfigure}[b]{0.45\textwidth}
		\includegraphics[width=\textwidth]{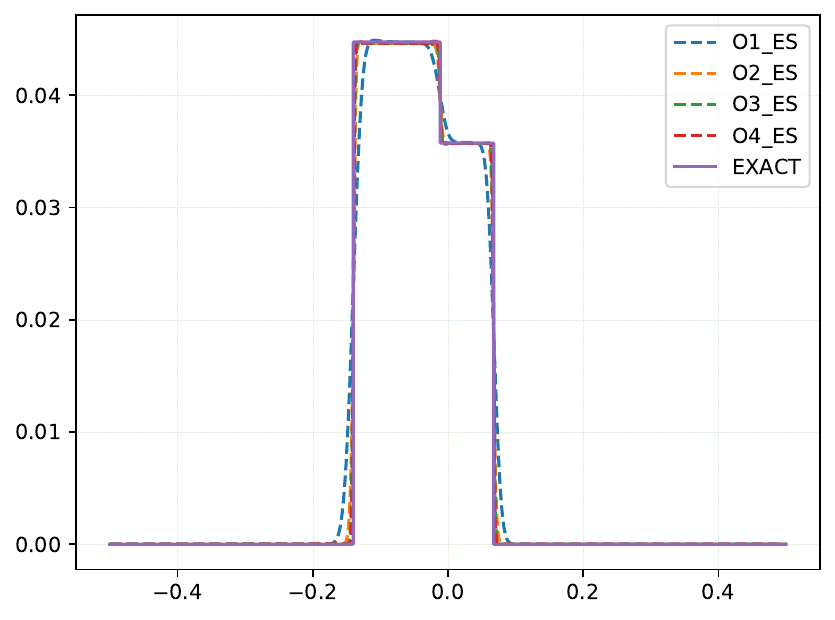}
		\caption{$\p_{12},~\text{2000 cells}$}
	\end{subfigure}
	\begin{subfigure}[b]{0.45\textwidth}
		\includegraphics[width=\textwidth]{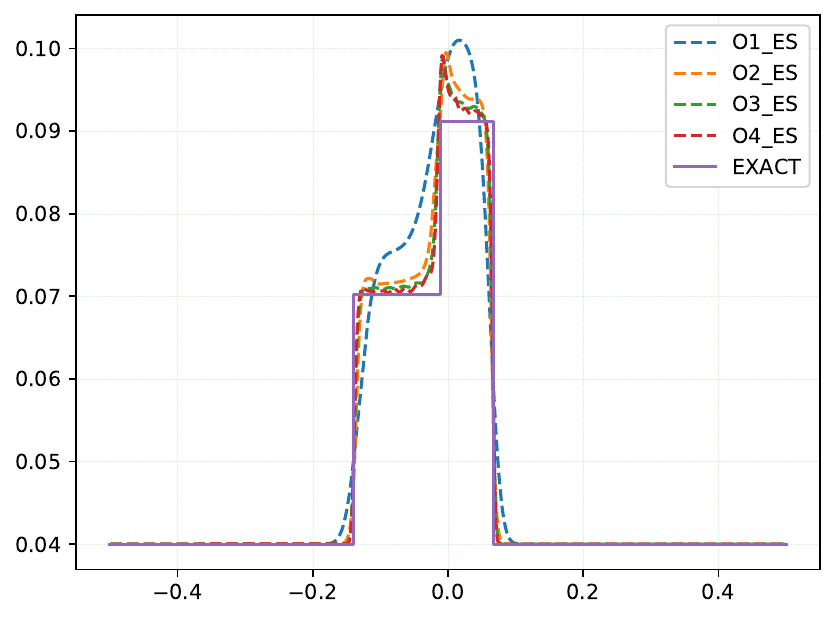}
		\caption{$\p_{22},~\text{200 cells}$}
	\end{subfigure}
	\begin{subfigure}[b]{0.45\textwidth}
		\includegraphics[width=\textwidth]{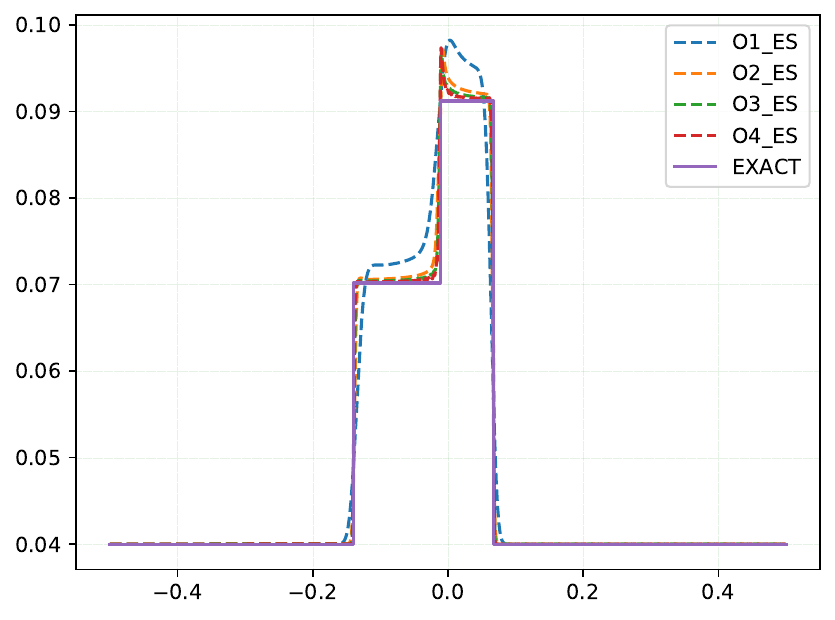}
		\caption{$\p_{22},~\text{2000 cells}$}
	\end{subfigure}
	\caption{\nameref{test6}: Plot of stress components $\p_{11},~\p_{12},~\p_{22}$ using 200 cells and 2000 cells.}
	\label{fig:test6b}
\end{figure}
\begin{figure}
	\centering
	\begin{subfigure}[b]{0.45\textwidth}
		\includegraphics[width=\textwidth]{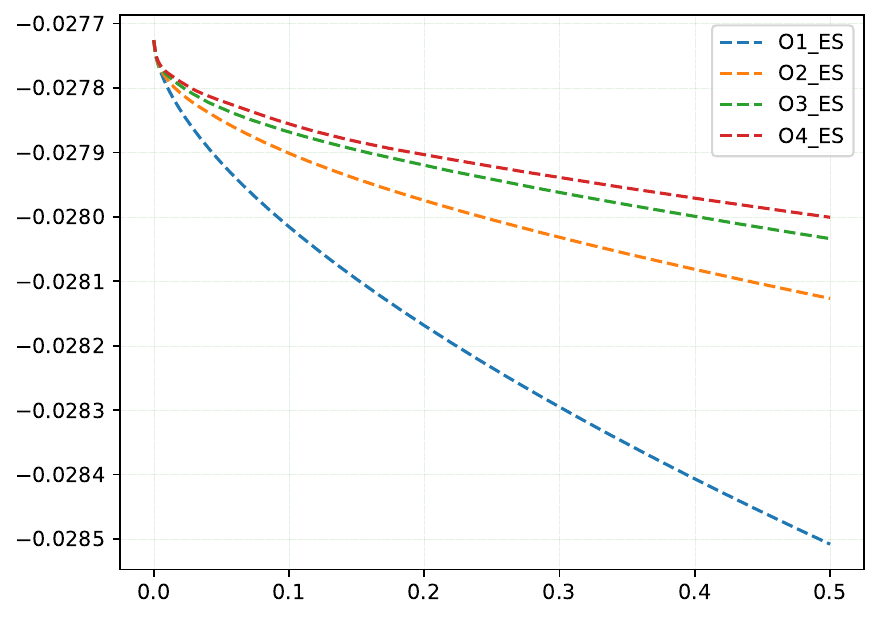}
		\caption{Entropy evolution,~$\text{200 cells}$}
	\end{subfigure}
	\caption{\nameref{test6}: Plot of entropy decay using 200 cells.}
	\label{fig:test6c}
\end{figure}

\subsubsection{1-D shear test problem}\label{test1}
This is a Riemann problem from \cite{Gavrilyuk2018,bhole2019fluctuation,Chandrashekar2020,Nkonga2022}, which gives rise to two shear waves. The domain is $[-0.5,0.5]$ with Neumann boundary conditions. The initial discontinuity is placed at $x=0$, and the initial conditions are given by
{\small \begin{equation*}
	(h,\ v_1, v_2, \ \p_{11}, \p_{12}, \p_{22}) = \begin{cases}
		\big(0.01 ,\  0,\ 0.2, \ 1.0\times10^{-4}, \ 0, \ 1.0\times10^{-4} \big)  & \text{if } x < 0.0, \\
		\big(0.01,\ 0, \ -0.2, \ 1.0\times10^{-4}, \ 0, \ 1.0\times10^{-4} \big) & \text{if } x > 0.0.   \end{cases}
\end{equation*}}
The computations are performed up to the final time $T=10.0.$
\begin{figure}
	\centering
	\begin{subfigure}[b]{0.45\textwidth}
		\includegraphics[width=\textwidth]{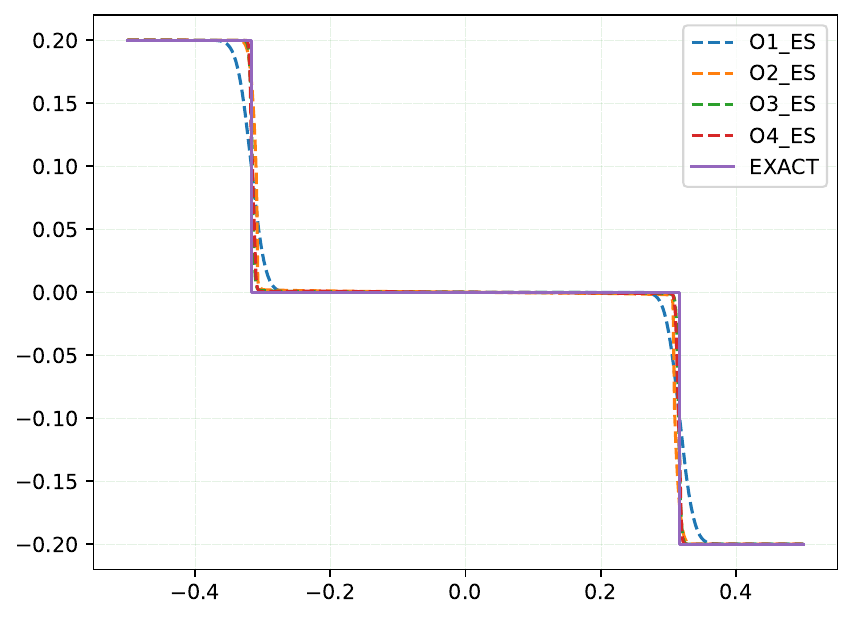}
		\caption{$v_2,~\text{500 cells}$}
	\end{subfigure}
	\begin{subfigure}[b]{0.45\textwidth}
		\includegraphics[width=\textwidth]{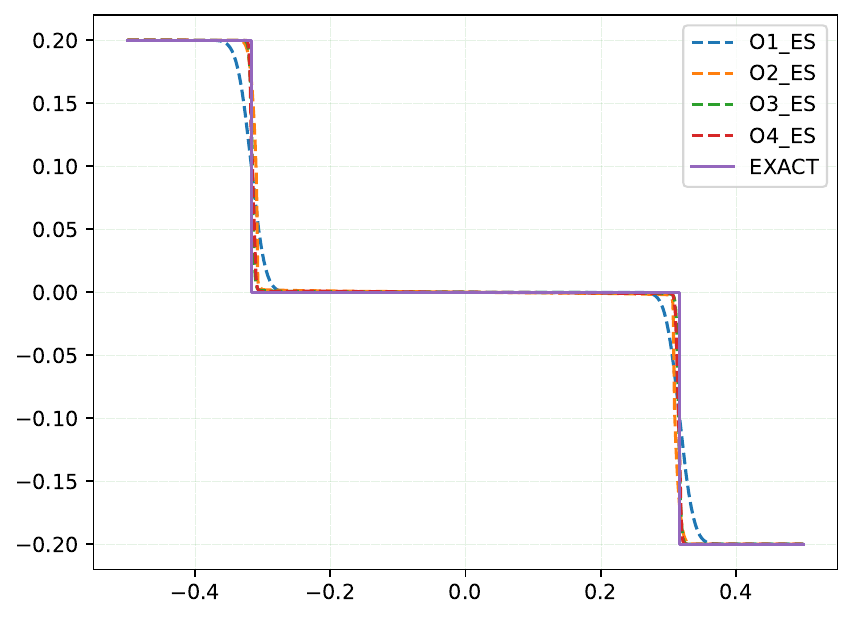}
		\caption{$v_2,~\text{2000 cells}$}
	\end{subfigure}
	\begin{subfigure}[b]{0.45\textwidth}
		\includegraphics[width=\textwidth]{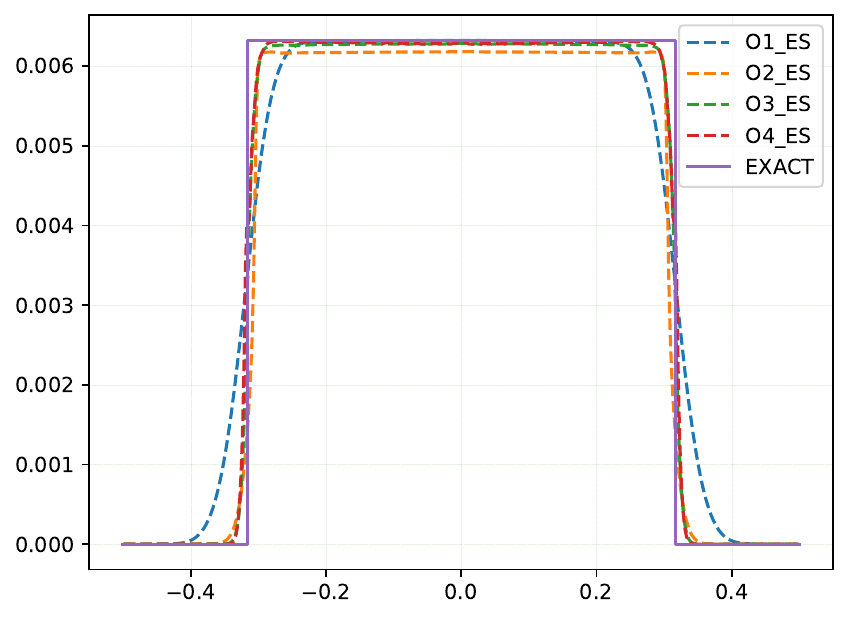}
		\caption{$\p_{12},~\text{500 cells}$}
	\end{subfigure}
	\begin{subfigure}[b]{0.45\textwidth}
		\includegraphics[width=\textwidth]{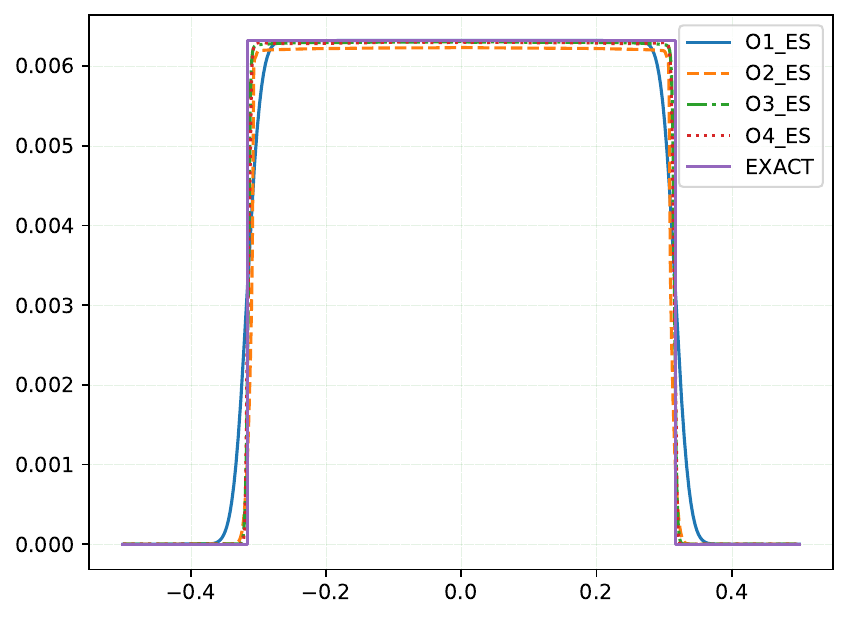}
		\caption{$\p_{12},~\text{2000 cells}$}
	\end{subfigure}
	\begin{subfigure}[b]{0.45\textwidth}
		\includegraphics[width=\textwidth]{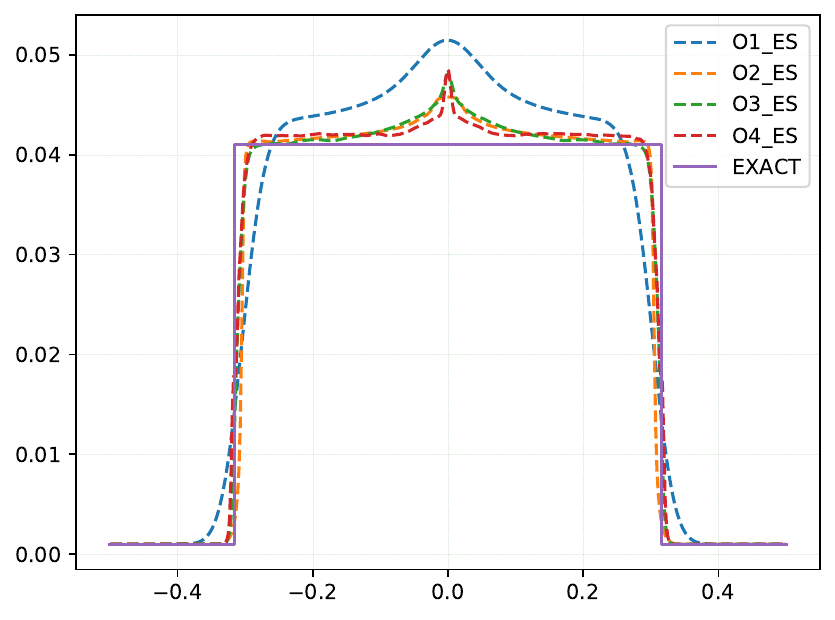}
		\caption{$\p_{22},~\text{500 cells}$}
	\end{subfigure}
	\begin{subfigure}[b]{0.45\textwidth}
		\includegraphics[width=\textwidth]{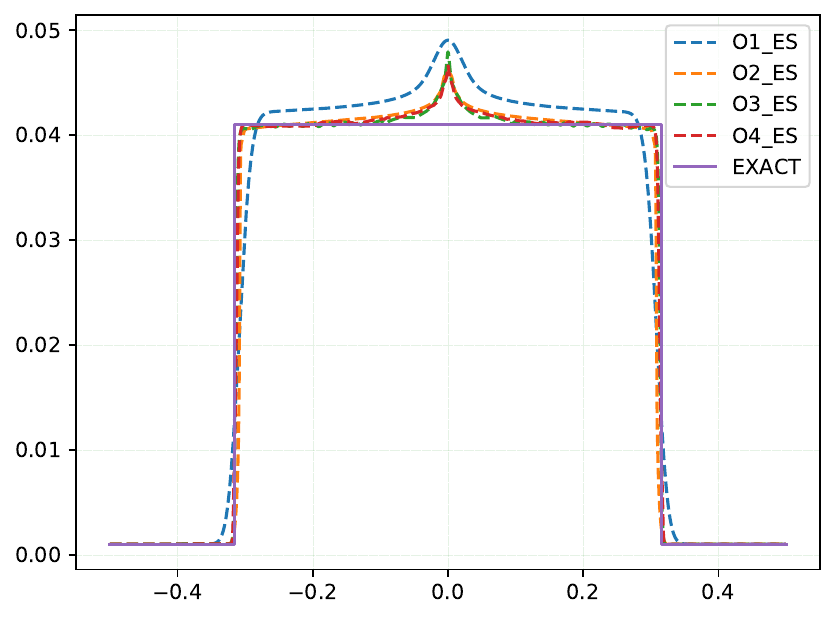}
		\caption{$\p_{22},~\text{2000 cells}$}
	\end{subfigure}
	\caption{\nameref{test1}: Plot of velocity $v_2$, stress components $\p_{12},~\p_{22}$ using 500 and 2000 cells.}
	\label{fig:test1b}
\end{figure}
The numerical solutions for the schemes O1\_ES, O2\_ES, O3\_ES, and O4\_ES using 200 and 2000 cells are presented in Fig.~\ref{fig:test1b}. We have plotted the transverse velocity $v_2$, $\p_{12}$ and $\p_{22}$ component of the stress tensor. The numerical solution has been compared with the exact solution from~\cite{Nkonga2022}. The exact solution of this Riemann problem consists of two shear waves. We observe that all the schemes are able to capture shear waves, and as expected, O4\_ES, O3\_ES, and O2\_ES are more accurate than O1\_ES. However, there are spurious spikes found at the center in $\p_{22}$, and this behavior is similar to what is observed with other numerical methods~\cite{Gavrilyuk2018,bhole2019fluctuation,Chandrashekar2020,Nkonga2022}. The result in Fig.~\ref{fig:test1c} shows the entropy decay behavior of the numerical scheme, which confirms the entropy stability of the scheme. 
\begin{figure}
	\centering
	\begin{subfigure}[b]{0.55\textwidth}
		\includegraphics[width=\textwidth]{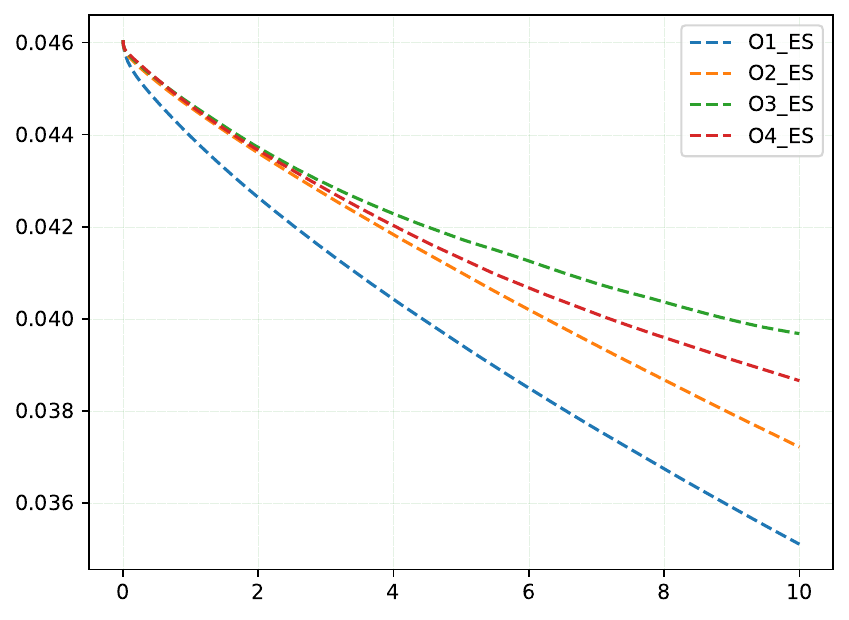}
		\caption{Entropy decay,~$\text{500 cells}$}
	\end{subfigure}
	\caption{\nameref{test1}: Plot of entropy decay using 500 cells.}
	\label{fig:test1c}
\end{figure}


\subsubsection{Single shock wave problem}\label{test4}
This Riemann problem from~\cite{Nkonga2022} should have a single shock wave according to the exact solution derived there. The computational domain is $[-0.5,0.5]$ with the Neumann boundary conditions. The initial discontinuity is placed at $x=0$, and the initial conditions are given by
{\small
\begin{equation*}
	(h,\ v_1, v_2, \ \p_{11}, \p_{12}, \p_{22}) = \begin{cases}
		\big(0.02 ,\  0,\ 0, \ 1.0\times10^{-1}, \ 0, \ 1.0\times10^{-1} \big) , \\
		\big(0.03,\ -7.010706099, \ 0, \ 16.616666666666658, \ 0, \ 1.0\times10^{-1} \big)   \end{cases}
\end{equation*}}
The numerical solutions are computed up to the final time $T=0.015811388$ with gravitational constant $g=9.81\times10^{3}$. The numerical solutions for the schemes O1\_ES, O2\_ES, O3\_ES, and O4\_ES are presented in Fig.~\ref{fig:test4a} using 500 and 2000 cells.
We have plotted the water depth $h$, velocity $v_1$, $\p_{11}$ components of the stress tensor and compare the numerical results with the exact solution provided in~\cite{Nkonga2022}. The exact solution of this Riemann problem consists of a single shock wave but we have observed that the computed numerical solutions exhibit an extra contact wave that is not present in the exact solution, and this is seen even with mesh refinement. Similar results were observed for the HLL-type schemes in~\cite{Nkonga2022}, which is a consequence of the sensitivity of solutions of non-conservative systems to numerical dissipation. The result in Fig.~\ref{fig:test4b} shows the entropy decay for the different numerical schemes using 500 cells.
\begin{figure}
	\centering
	\begin{subfigure}[b]{0.45\textwidth}
		\includegraphics[width=\textwidth]{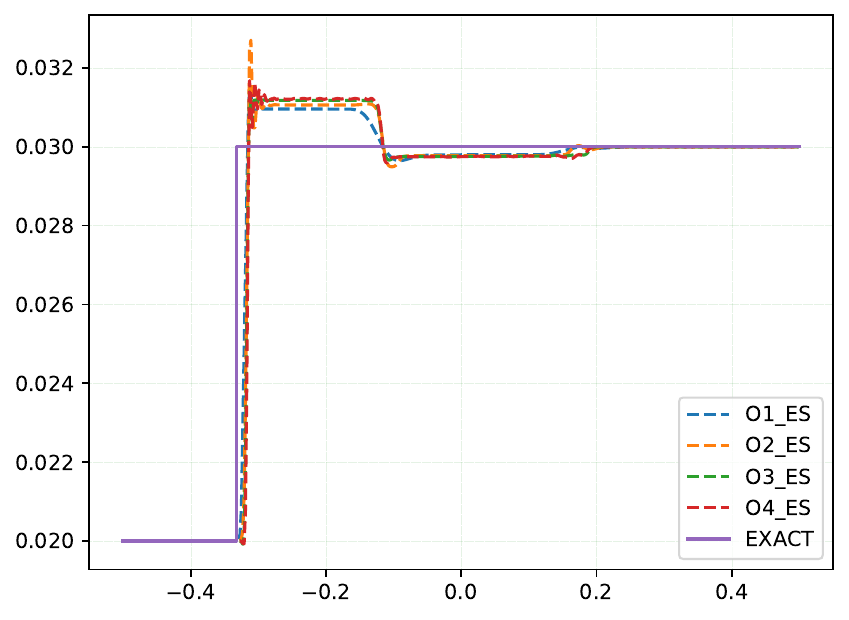}
		\caption{$h,~\text{500 cells}$}
	\end{subfigure}
	\begin{subfigure}[b]{0.45\textwidth}
		\includegraphics[width=\textwidth]{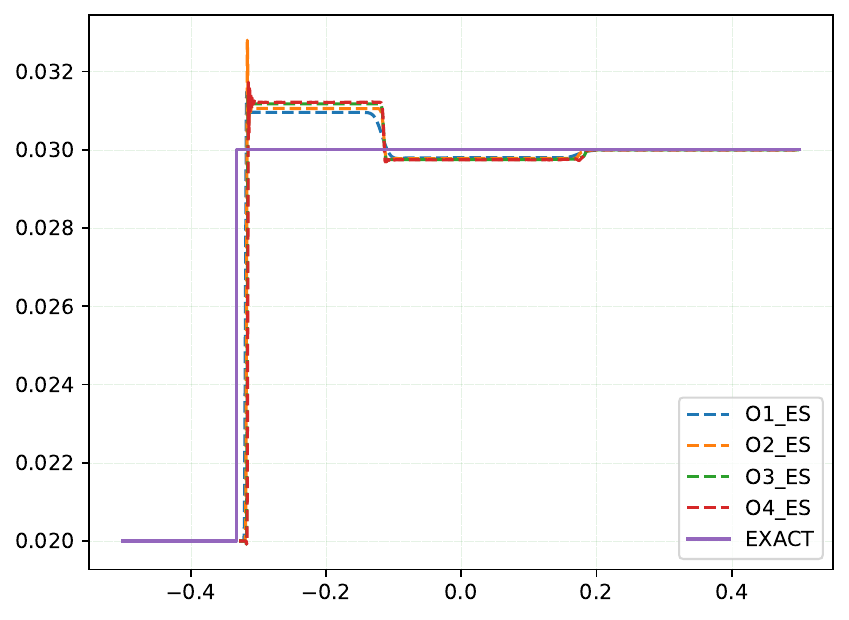}
		\caption{$h,~\text{2000 cells}$}
	\end{subfigure}
	\begin{subfigure}[b]{0.45\textwidth}
		\includegraphics[width=\textwidth]{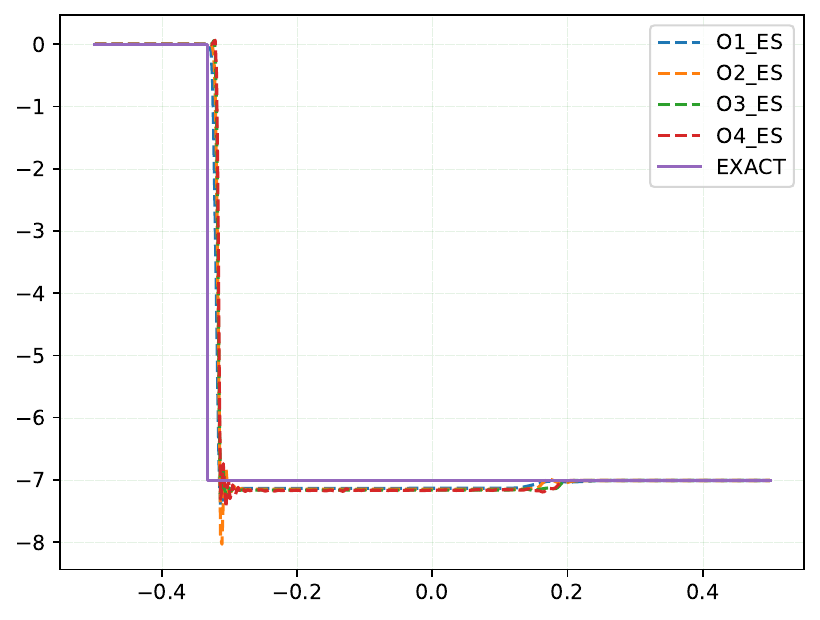}
		\caption{$v_1,~\text{500 cells}$}
	\end{subfigure}
	\begin{subfigure}[b]{0.45\textwidth}
		\includegraphics[width=\textwidth]{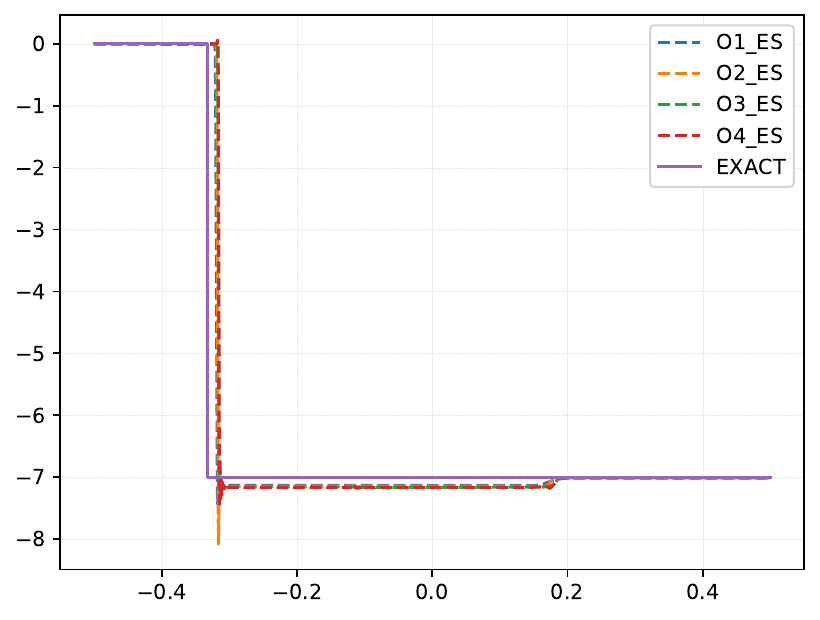}
		\caption{$v_1,~\text{2000 cells}$}
	\end{subfigure}
	\begin{subfigure}[b]{0.45\textwidth}
		\includegraphics[width=\textwidth]{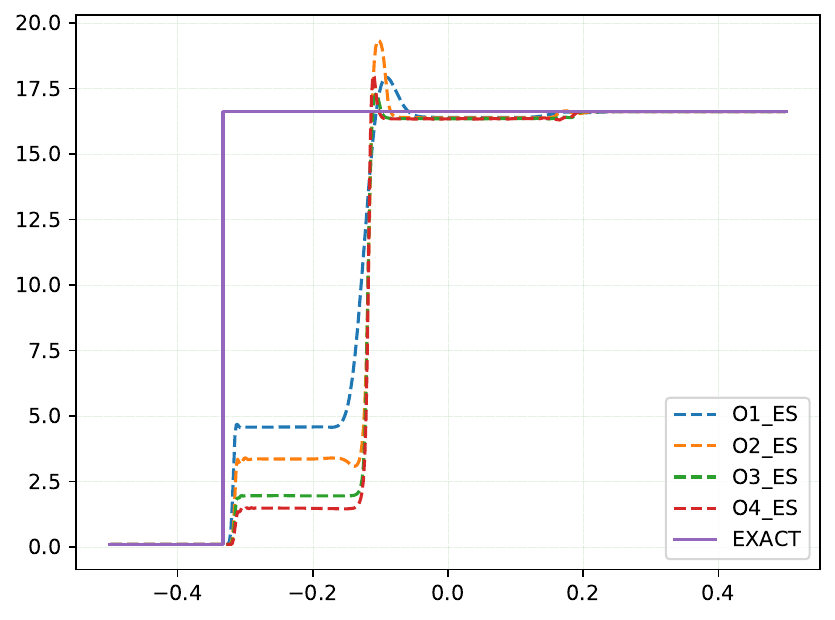}
		\caption{$\p_{11},~\text{500 cells}$}
	\end{subfigure}
	\begin{subfigure}[b]{0.45\textwidth}
		\includegraphics[width=\textwidth]{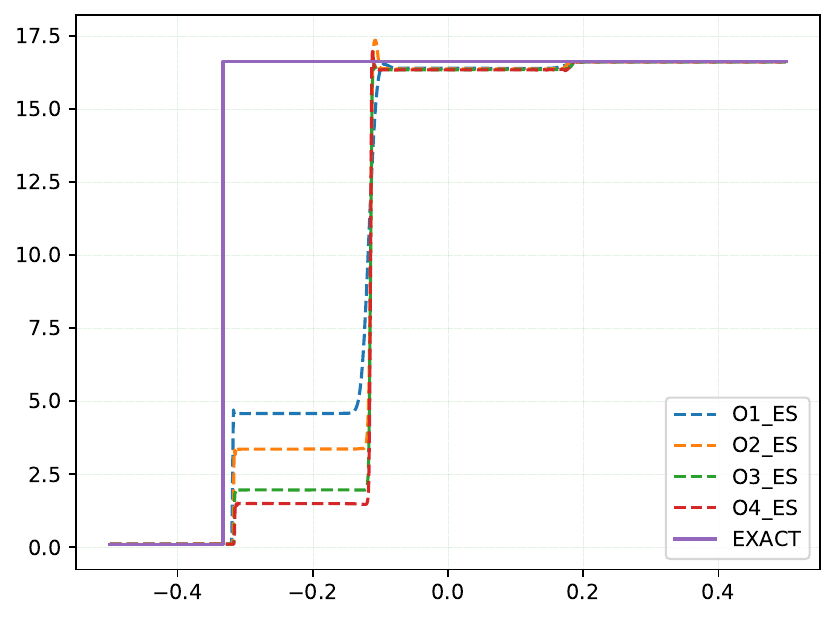}
		\caption{$\p_{11},~\text{2000 cells}$}
	\end{subfigure}
	\caption{\nameref{test4} Plot of water depth $h$, velocity $v_1$ and stress tensor components $\p_{11}$ using 500 and 2000 cells.}
	\label{fig:test4a}
\end{figure}
\begin{figure}
	\centering
	\begin{subfigure}[b]{0.45\textwidth}
		\includegraphics[width=\textwidth]{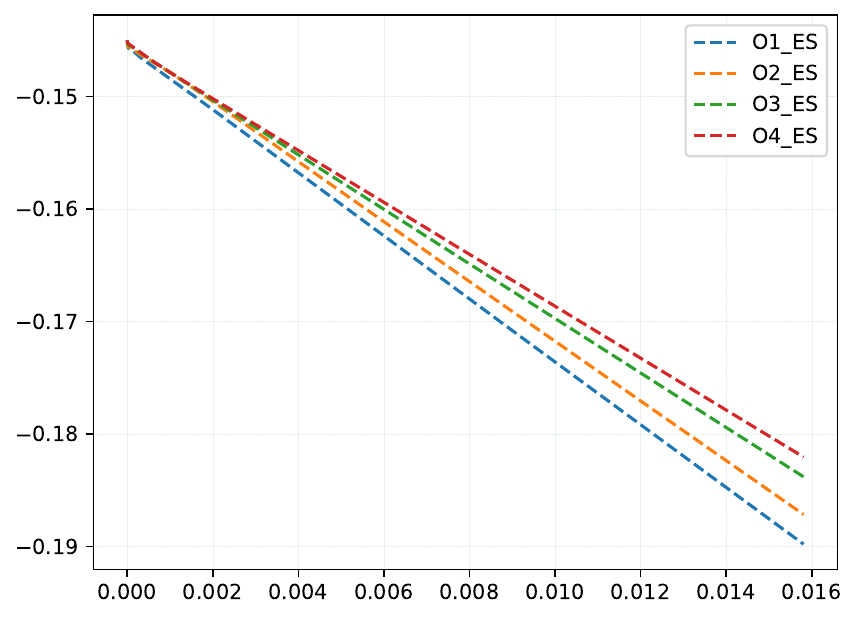}
		\caption{Entropy evolution, 500 cells}
	\end{subfigure}
	\caption{\nameref{test4} Plot of velocity $v_1$, stress tensor components $\p_{11}$ using 500 and 2000 cells and entropy evolution using 500 cells.}
	\label{fig:test4b}
\end{figure}

\subsubsection{1-D roll wave problem}\label{test11}
This problem models the flow of a thin layer of liquid flowing down an inclined bottom and results in the formation of hydraulic jump and roll waves. We use periodic boundary conditions and the initial conditions are taken from \cite{Gavrilyuk2018,bhole2019fluctuation,Chandrashekar2020} and given by
$$
h(x,0)=h_0[1+a\sin(2\pi x/L_x)],~~~~v_1(x,0)=\sqrt{g h_0 \tan{\theta}/C_f},~~~v_2(x,0)=0,
$$
$$
\p_{11}(x,0)=\p_{22}(x,0)=\frac{1}{2}\phi h^2(x,0),~~~\p_{12}(x,0)=0.
$$
The bottom topography is given by $b=-x\tan{\theta}$ and we consider two sets of parameters as given in~\cite{Ivanova2017}. In case $1$, the parameters are $\theta=0.05011 ,~C_f=0.0036,~h_0=7.98\times 10^{-3}$m, $a=0.05,~\phi=22.7s^{-2}, ~ C_r=0.00035,~L_x=1.3$m. In case $2$, the parameters are $\theta=0.11928 ,~C_f=0.0038,~h_0=5.33\times 10^{-3}$m, $a=0.05,~\phi=153.501s^{-2}, ~ C_r=0.002,~L_x=1.8$m. The computations are performed using 500 cells up to the final time $T=25$. The numerical results are presented in Fig.~\ref{fig:test11a}. We have also plotted the water depth $h$  for both the cases with Brock's experimental data \cite{brock1969development,brock1970periodic} in Fig.~\ref{fig:test11b} and observe that the numerical results are comparable with measurements. The classical shallow water model captures the hydraulic jump but is unable to predict the roll wave profile, which is captured by the SSW model.
\begin{figure}
	\centering
	\begin{subfigure}[b]{0.45\textwidth}
		\includegraphics[width=\textwidth]{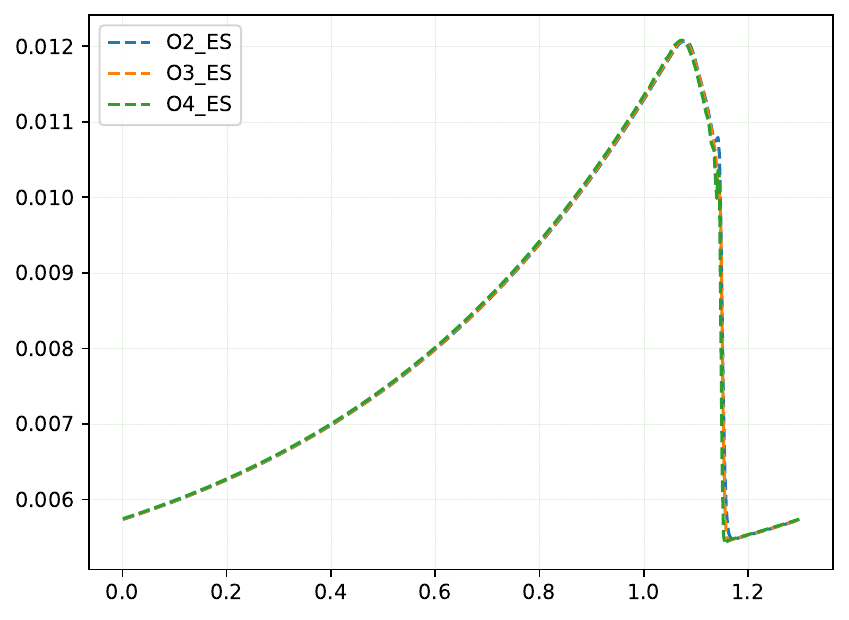}
		\caption{$h$}
	\end{subfigure}
	\begin{subfigure}[b]{0.45\textwidth}
		\includegraphics[width=\textwidth]{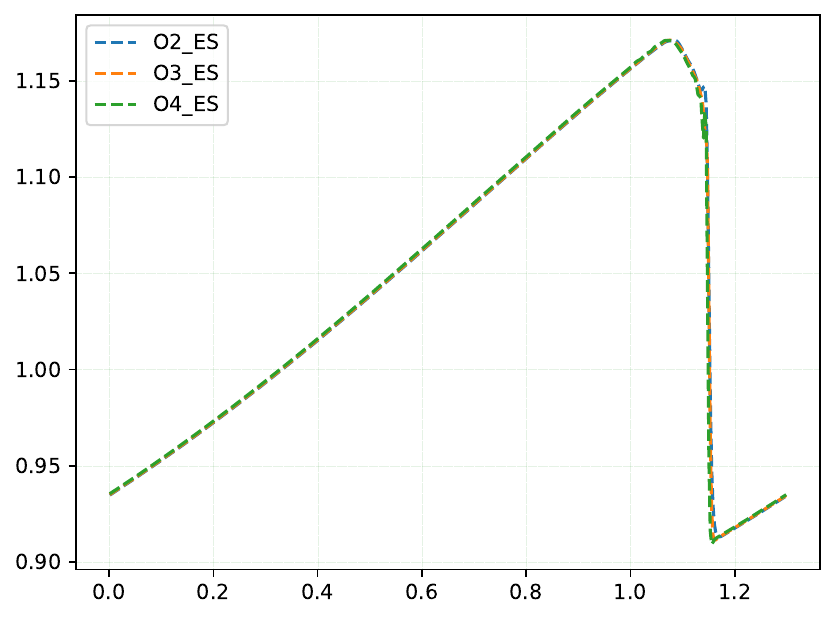}
		\caption{$v_1$}
	\end{subfigure}
	\begin{subfigure}[b]{0.45\textwidth}
		\includegraphics[width=\textwidth]{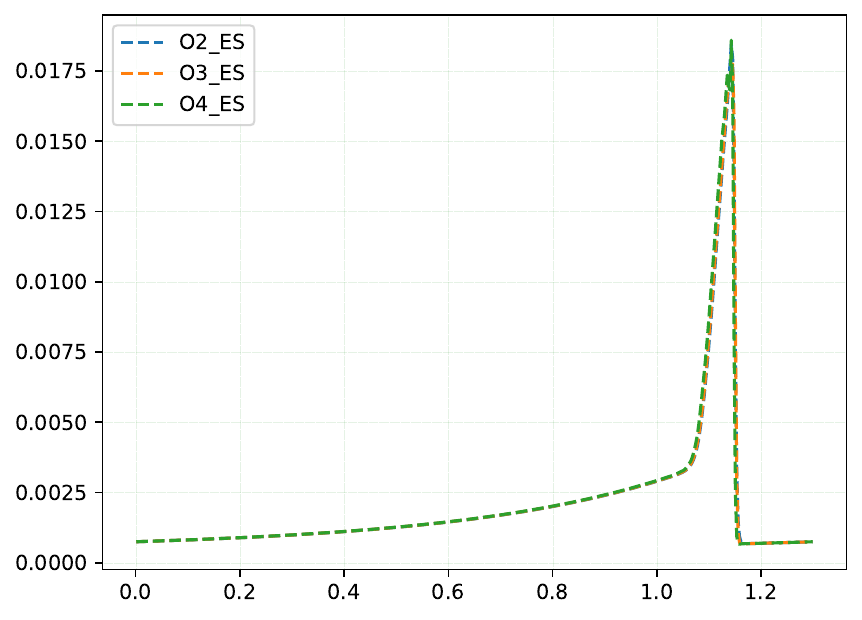}
		\caption{$\p_{11}$}
	\end{subfigure}
	\begin{subfigure}[b]{0.45\textwidth}
		\includegraphics[width=\textwidth]{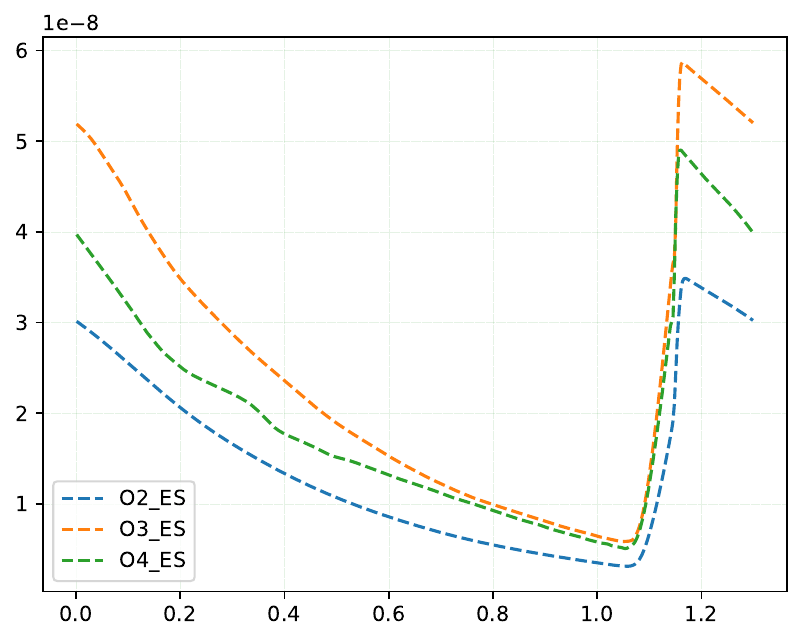}
		\caption{$\p_{22}$}
	\end{subfigure}
	\caption{\nameref{test11} Plot of water depth $h$, velocity component $v_1$ and  stress components $\p_{11},~\p_{22}$ using 500 cells.}
	\label{fig:test11a}
\end{figure}
\begin{figure}
	\centering
	\begin{subfigure}[b]{0.45\textwidth}
		\includegraphics[width=\textwidth]{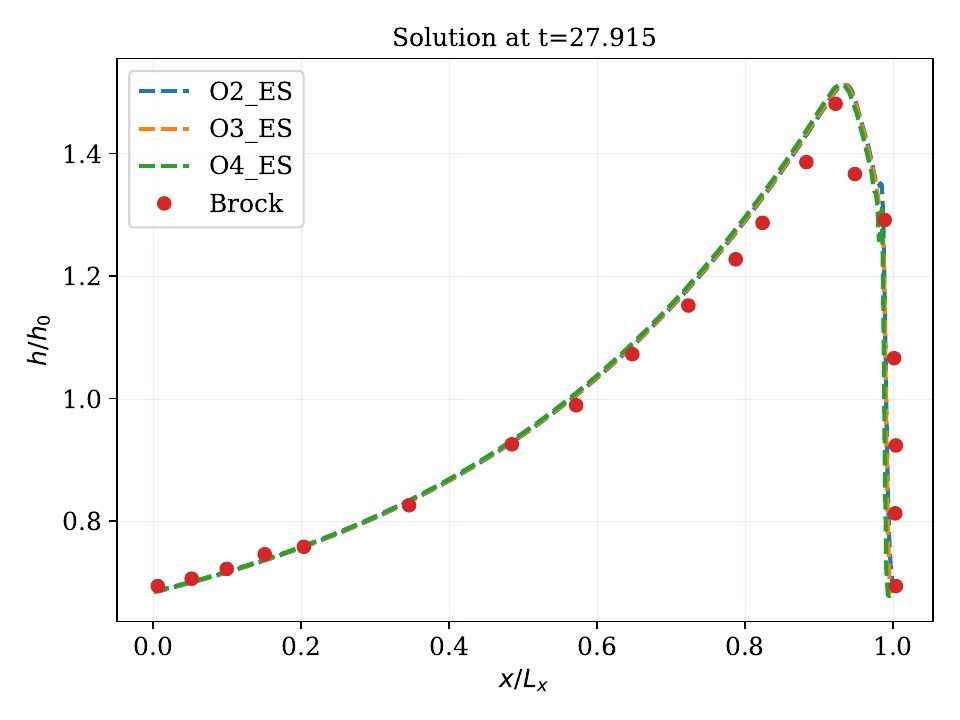}
		\caption{$h$}
	\end{subfigure}
	\begin{subfigure}[b]{0.45\textwidth}
		\includegraphics[width=\textwidth]{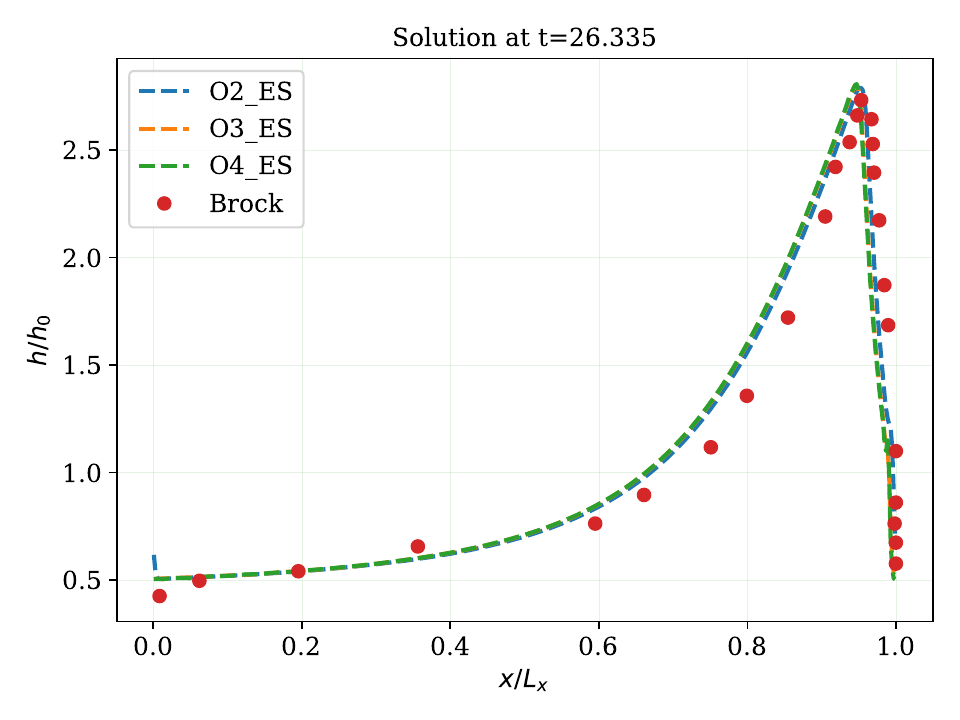}
		\caption{$h$}
	\end{subfigure}
	\caption{\nameref{test11} Comparison of water depth $h$ with Brock's experimental data for O2\_ES, O3\_ES, O4\_ES (a) Case 1 at time $t=27.915$ s (b) Case 3 at time $t=26.335$ s.}
	\label{fig:test11b}
\end{figure}
\subsection{Two-dimensional test problems}
\subsubsection{2-D accuracy test}
This is a two-dimensional extension of the smooth problem \eqref{test:accuracy}, which was solved in 1-D. The test case is used to check the formal order and accuracy of the proposed scheme in two dimensions. The forcing term $\mathcal{S}(x,y,t)$ is given by,
\begin{equation*}
	\mathcal{S}(x,y,t)=\left(0, 2\alpha, 2\alpha, \alpha, \alpha, \alpha \right)^\top,
\end{equation*}
where $\alpha=\pi\cos(2\pi(x+y-t))(1+2 g+g\sin(2 \pi (x+y-t)))$.
The exact solution with domain $[-0.5,\,0.5]\times[-0.5,0.5]$ is as follows,
\begin{align*}
	h(x,y,t) &= 2+\sin(2\pi (x+y-t)), \qquad v_1(x,y,t)=0.5, \qquad v_2(x,y,t)=0.5,&\\
	&\p_{11}(x,y,t)=\p_{22}(x,y,t)=1, \qquad \p_{12}(x,y,t)=0.
\end{align*}
Periodic boundary conditions are used for the computations, and the error is computed using the exact solution at time $T=0.5$s. We present the $L^1$ errors and order of accuracy for the water depth $h$ in Table~\ref{tab:acc3} using the schemes O2\_ES, O3\_ES, and O4\_ES. We observe that the schemes have reached the designed order of accuracy.

\begin{table}[htb!]
	\centering
	\begin{tabular}{c|c|c|c|c|c|c|}
		\hline Number of cells  & \multicolumn{2}{|c}{{O2\_ES}} &  \multicolumn{2}{|c}{O3\_ES} & \multicolumn{2}{|c}{O4\_ES}  \\
		\hline   & $L^1$ error  &  Order &  $L^1$ error      & Order & $L^1$ error      & Order \\
		\hline 40 & 1.10e-02 & -- & 6.76e-04 & -- & 4.68e-05 & -- \\
		80 & 2.42e-03 & 2.19 & 9.05e-05 & 2.90 &4.29e-06 & 3.45 \\
		160 & 8.14e-04 & 1.57 & 1.16e-05 & 2.96 & 3.31e-07 & 3.70 \\
		320 & 2.40e-04 & 1.78 & 1.46e-06 & 2.992 & 2.30e-08 & 3.85 \\
		640 & 6.63e-05 & 1.86 & 1.82e-07 & 2.998 & 1.54e-09 & 3.90 \\
		1280& 1.78e-05 & 1.90 & 2.28e-08 & 2.999 & 1.01e-10 & 3.93\\
		\hline
	\end{tabular}
	\caption{Accuracy test: $L^1$ errors and order of accuracy for \rev{the} water depth $h$.}
	\label{tab:acc3}
\end{table}

\subsubsection{2-D roll wave problem}\label{test13}
This is a two-dimensional extension of the 1-D roll wave test from  Section~\ref{test11}. The initial conditions are given by
$$h(x,y,0)=h_0[1+a\sin(2\pi x/L_x)+a\sin(2\pi y/L_y)],$$
$$ v_1(x,y,0)=\sqrt{g h_0 \tan{\theta}/C_f}, \qquad v_2(x,y,0)=0,$$
$$p_{11}(x,y,0)=\p_{22}(x,y,0)=\frac{1}{2}\phi h^2(x,0), \qquad \p_{12}(x,y,0)=0.$$
The computational domain is \rev{$[0,1.3]\times[0,0.5]$} with the periodic boundary conditions. This problem includes the source term with bottom topography given by  $b=-x\tan{\theta}$. Here, $\theta=0.05011 ,~C_f=0.0036,~h_0=7.98\times 10^{-3}$m, $a=0.05,~\phi=22.7s^{-2}, ~ C_r=0.00035,~L_x=1.3$m,\ $L_y=0.5$m as given in \cite{Gavrilyuk2018,bhole2019fluctuation,Chandrashekar2020}. The computations are performed up to the final time $T=36$s, and the numerical results are presented in Figures~\ref{fig:test13b}, \ref{fig:test13c}, \ref{fig:test13d}. The elevation of the water surface shown in Fig.~\ref{fig:test13b} indicates the formation of hydraulic jump and roll waves, but the solutions do not look smooth. This type of solution has been observed in previous studies~\cite{bhole2019fluctuation,Chandrashekar2020} using different numerical schemes. Fig.~\ref{fig:test13c} shows the projection of the $h$ profile onto the plane $y=0$, and its $y$-average is shown as a red line. While the profile varies in the $y$ direction and looks random/turbulent, the average profile shows the characteristic roll wave and hydraulic jump that is also seen in the 1-D simulations. The higher order schemes exhibit more fluctuations about the average and also give a better resolution of the roll wave than the first order scheme. Fig.~\ref{fig:test13d} shows the contour lines of the $h$ field at time $T=36$ units which show carbuncle-like structures that are seen in some compressible flow problems~\cite{Elling2009}. The first-order scheme shows a somewhat smooth solution \rev{similar to~\cite{bhole2019fluctuation,Chandrashekar2020}}, while the higher-order schemes show more small-scale structures which have been observed in previous studies also~\cite{Chandrashekar2020}. The solutions qualitatively look similar to those obtained using the five-wave HLLC solver, while the two-wave and three-wave HLL-type schemes show more smooth solutions~\cite{Chandrashekar2020}. This indicates that the present schemes are able to more accurately model the five waves in the solution, like the sophisticated multi-wave approximate Riemann solvers. The similarity of solutions obtained for this problem from different numerical schemes suggests that they may not be purely numerical artifacts.
\begin{figure}
	\centering
	\begin{subfigure}[b]{0.49\textwidth}
		\includegraphics[width=\textwidth]{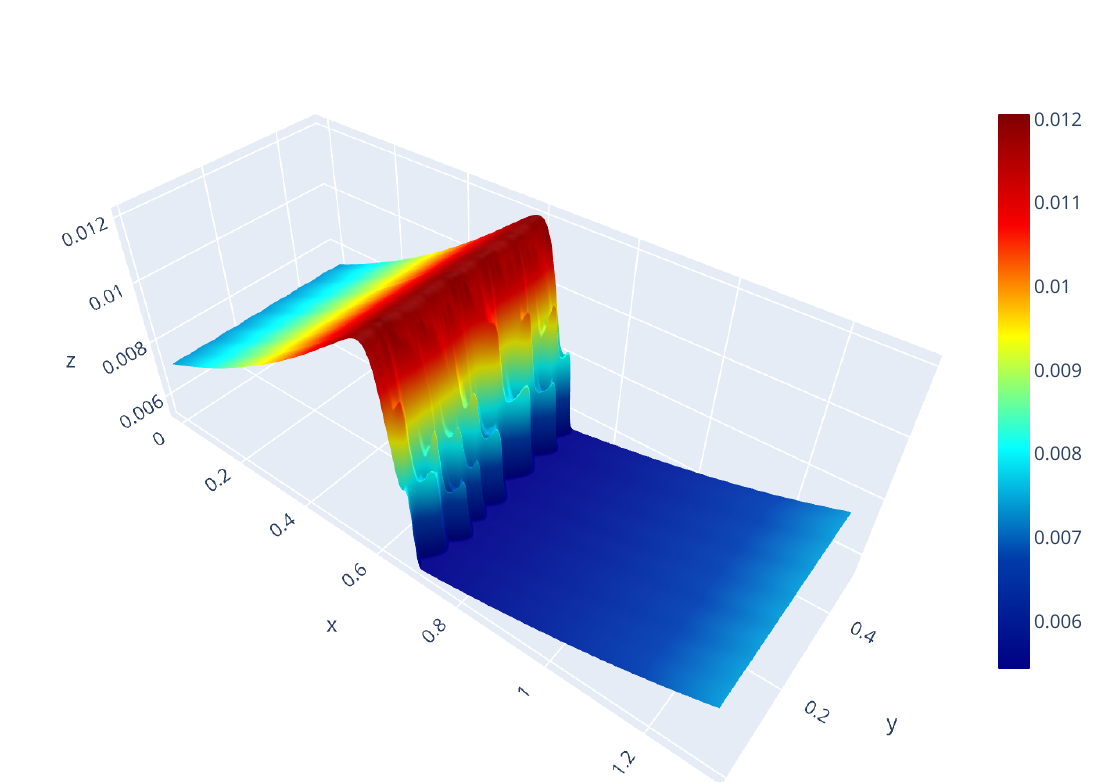}
		\caption{$h$, O1\_ES}
	\end{subfigure}
	\begin{subfigure}[b]{0.49\textwidth}
		\includegraphics[width=\textwidth]{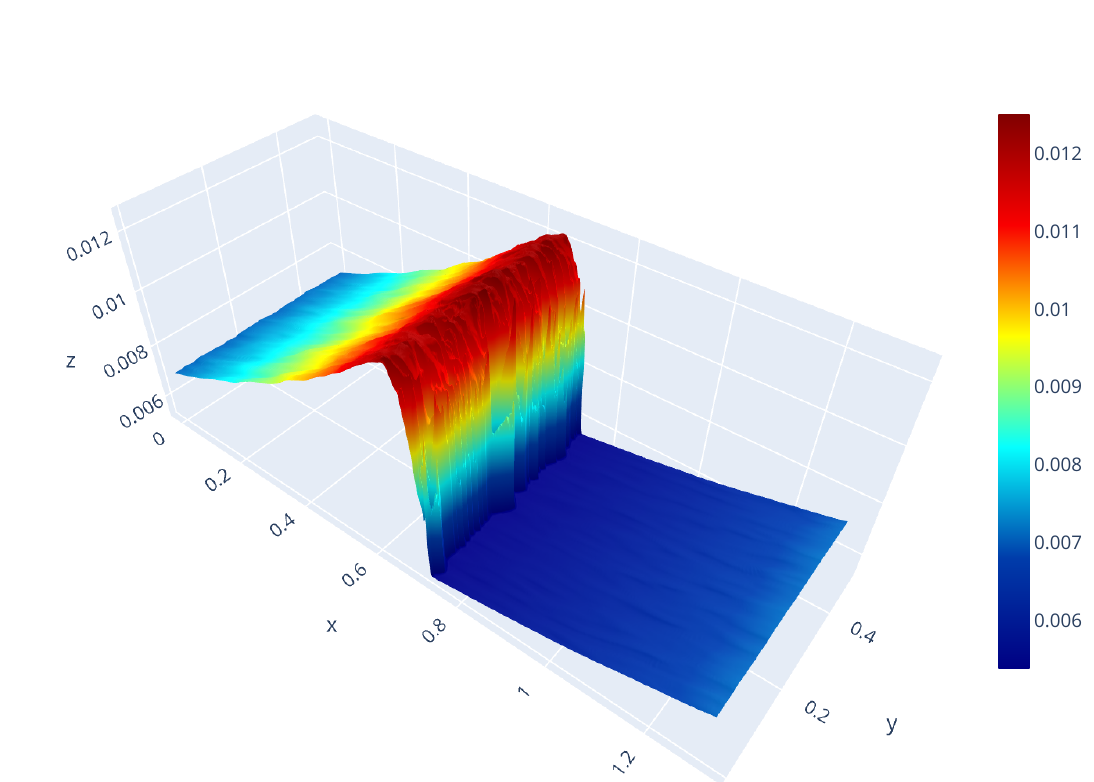}
		\caption{$h$, O2\_ES}
	\end{subfigure}
	\begin{subfigure}[b]{0.49\textwidth}
		\includegraphics[width=\textwidth]{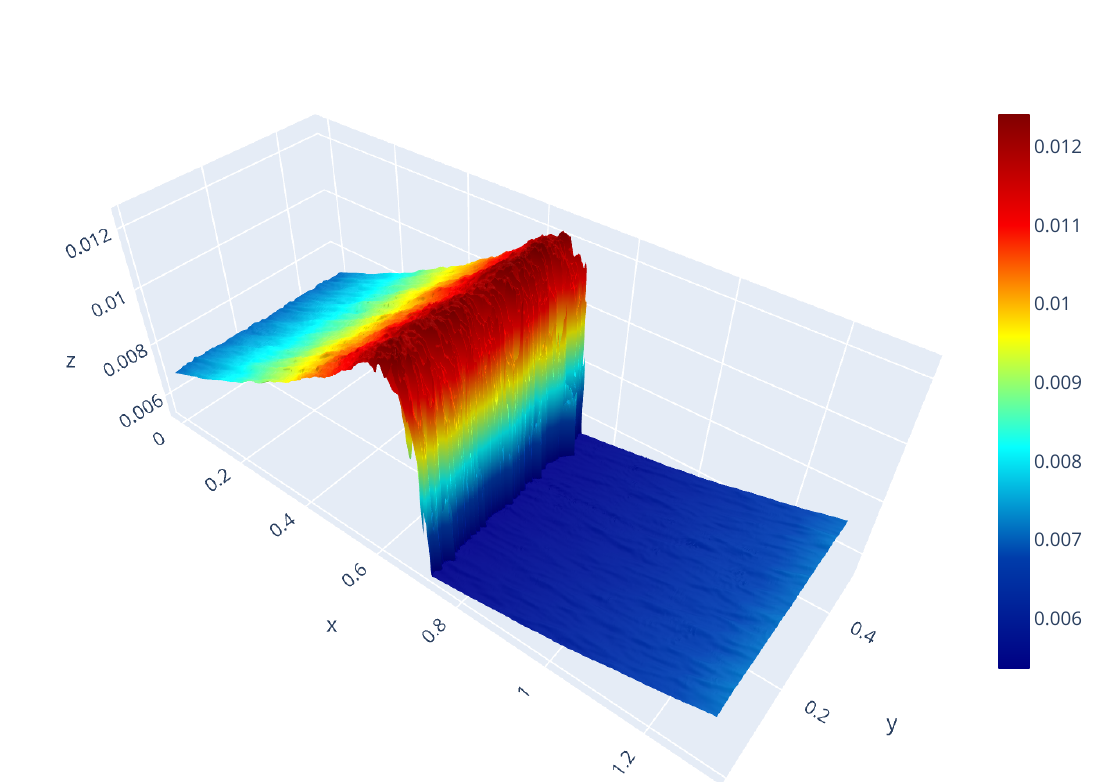}
		\caption{$h$, O3\_ES}
	\end{subfigure}
	\begin{subfigure}[b]{0.49\textwidth}
		\includegraphics[width=\textwidth]{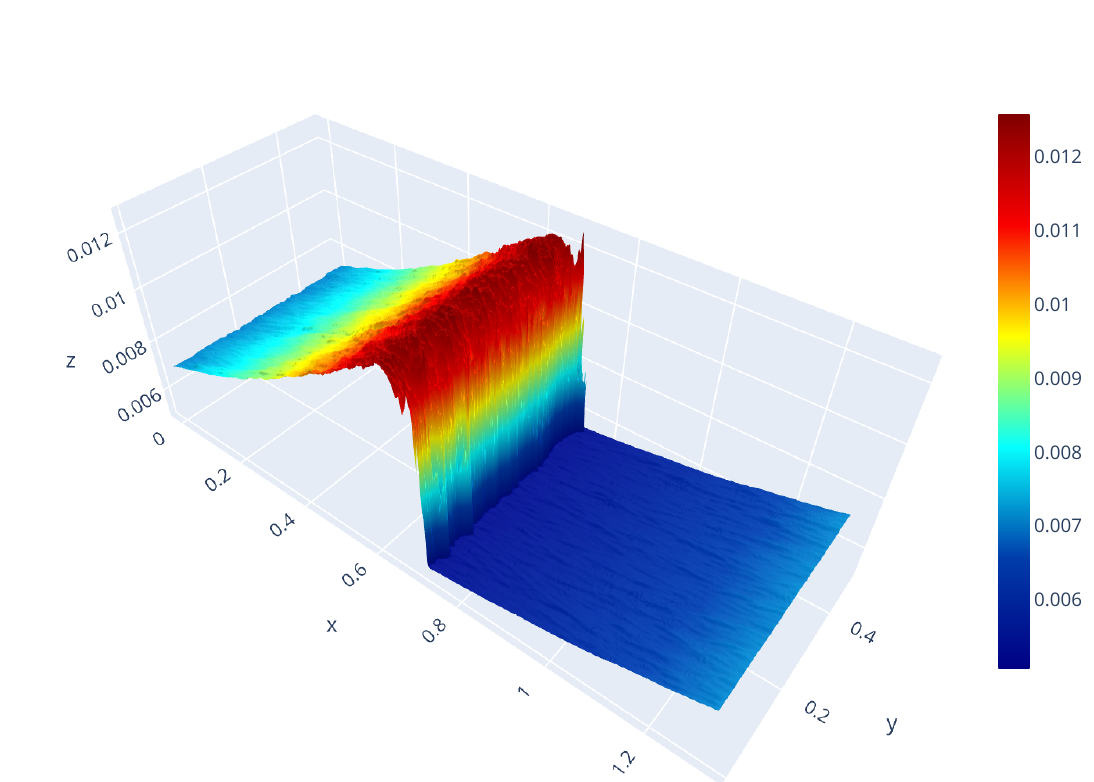}
		\caption{$h$, O4\_ES}
	\end{subfigure}
	\caption{\nameref{test13} Plot of water depth $h$ at 1040$\times$400 cells at time $T=36$ unit.}
	\label{fig:test13b}
\end{figure}

\begin{figure}
	\centering
	\begin{subfigure}[b]{0.45\textwidth}
		\includegraphics[width=\textwidth]{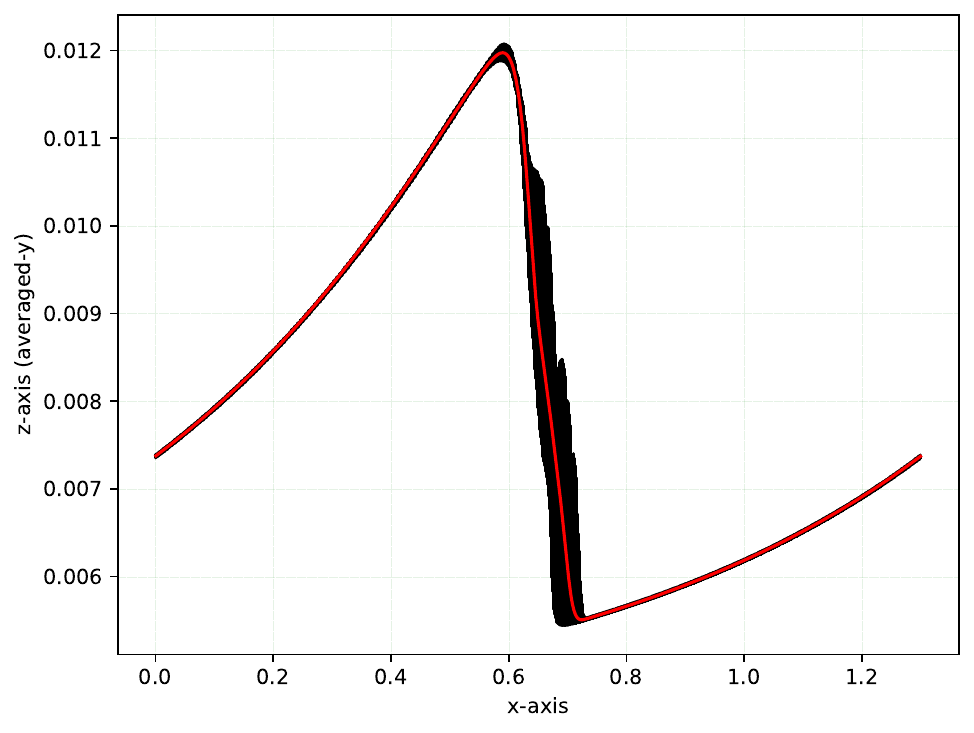}
		\caption{$h$, O1\_ES}
	\end{subfigure}
	\begin{subfigure}[b]{0.45\textwidth}
		\includegraphics[width=\textwidth]{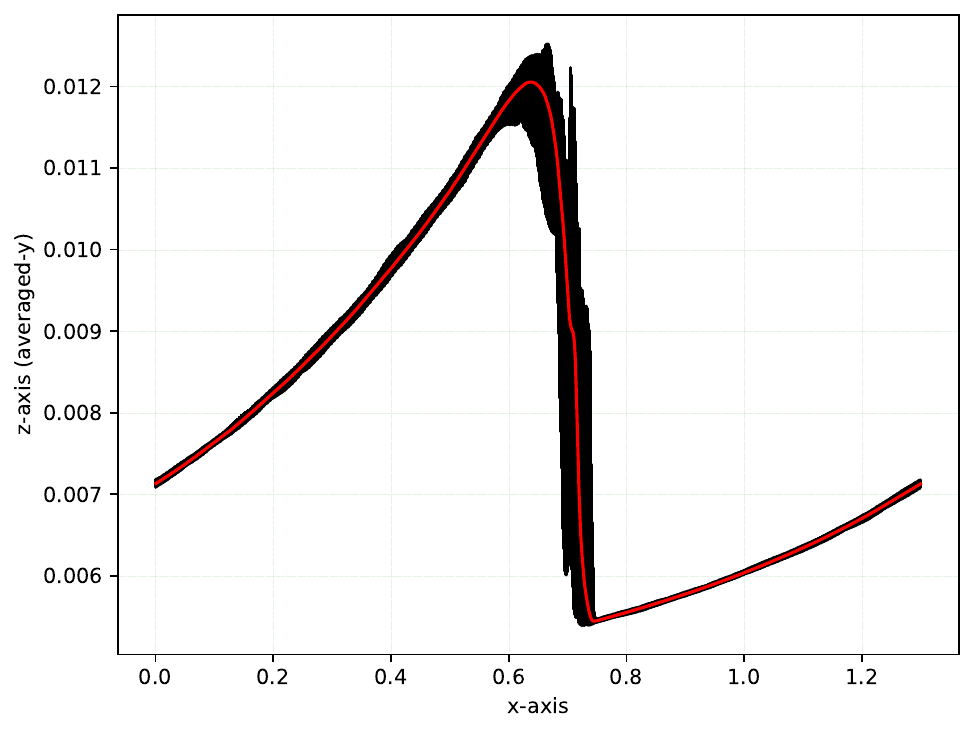}
		\caption{$h$, O2\_ES}
	\end{subfigure}
	\begin{subfigure}[b]{0.45\textwidth}
		\includegraphics[width=\textwidth]{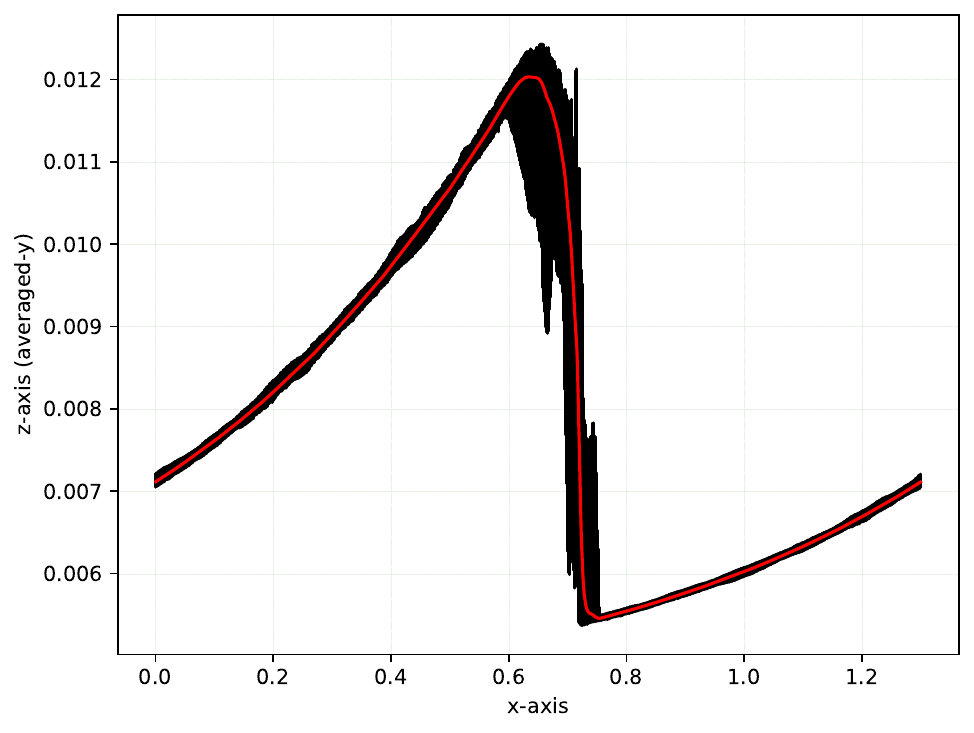}
		\caption{$h$, O3\_ES}
	\end{subfigure}
	\begin{subfigure}[b]{0.45\textwidth}
		\includegraphics[width=\textwidth]{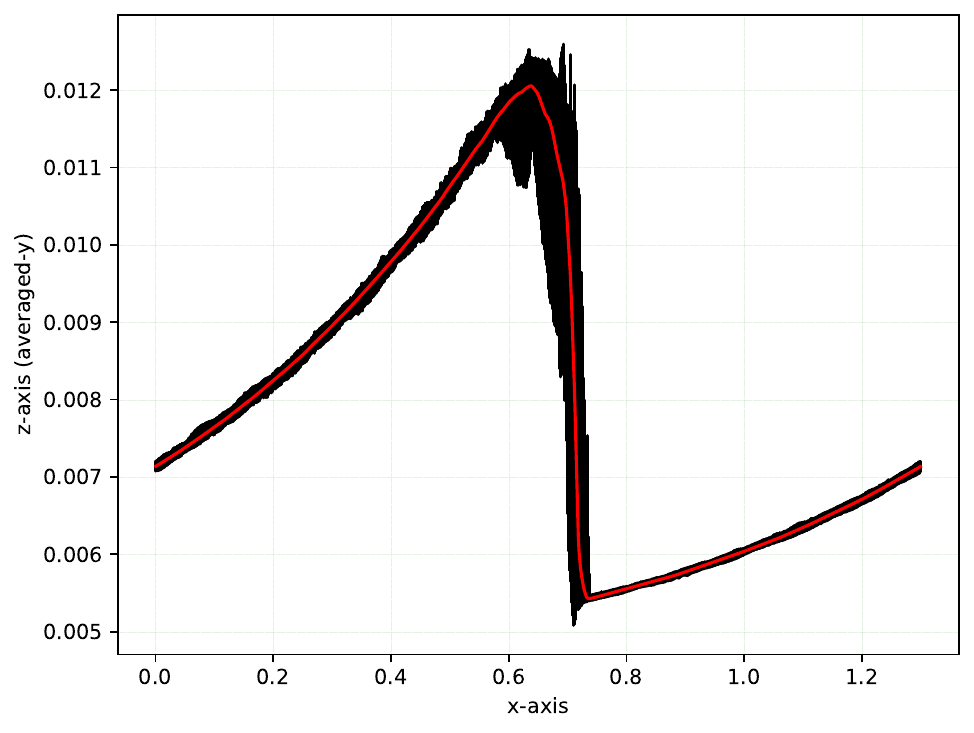}
		\caption{$h$, O4\_ES}
	\end{subfigure}
	\caption{\nameref{test13} Y-average of depth field for different schemes at time $T=36$ unit on $1040\times400$ mesh.}
	\label{fig:test13c}
\end{figure}

\begin{figure}
	\centering
	\begin{subfigure}[b]{0.45\textwidth}
		\includegraphics[width=\textwidth]{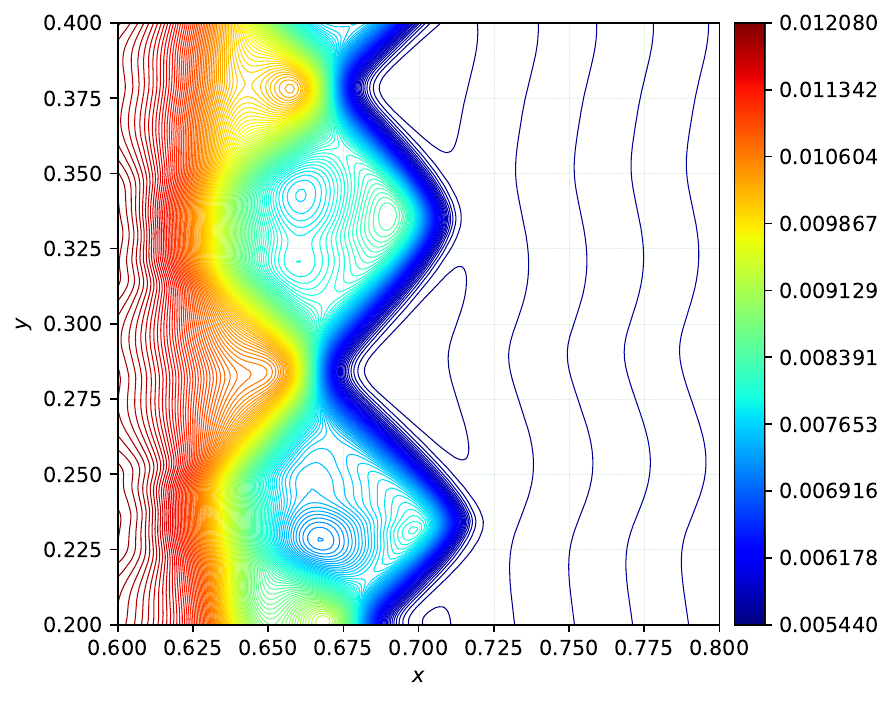}
		\caption{$h$, O1\_ES}
	\end{subfigure}
	\begin{subfigure}[b]{0.45\textwidth}
		\includegraphics[width=\textwidth]{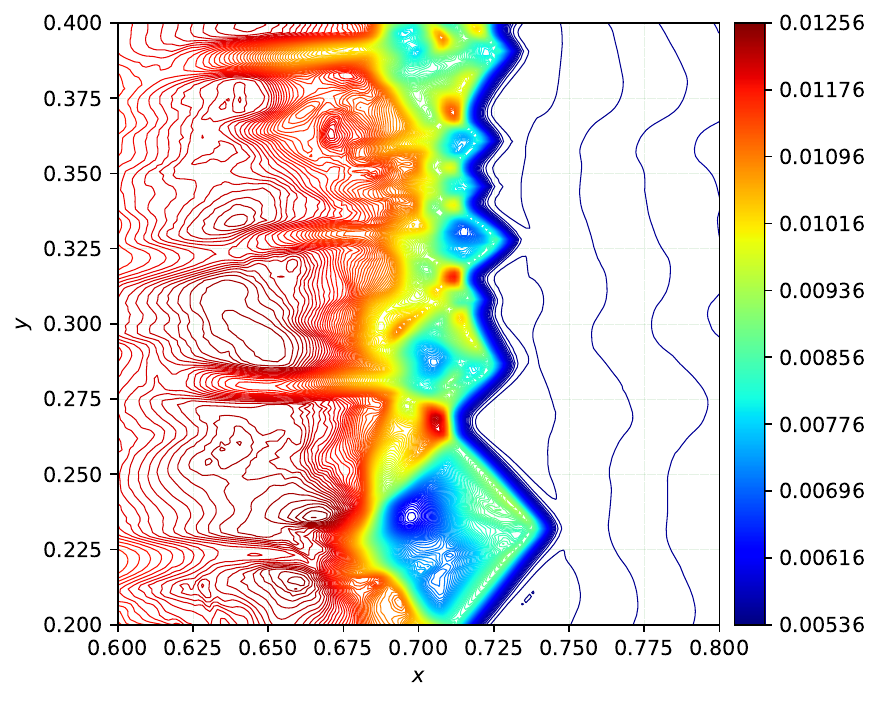}
		\caption{$h$, O2\_ES}
	\end{subfigure}
	\begin{subfigure}[b]{0.45\textwidth}
		\includegraphics[width=\textwidth]{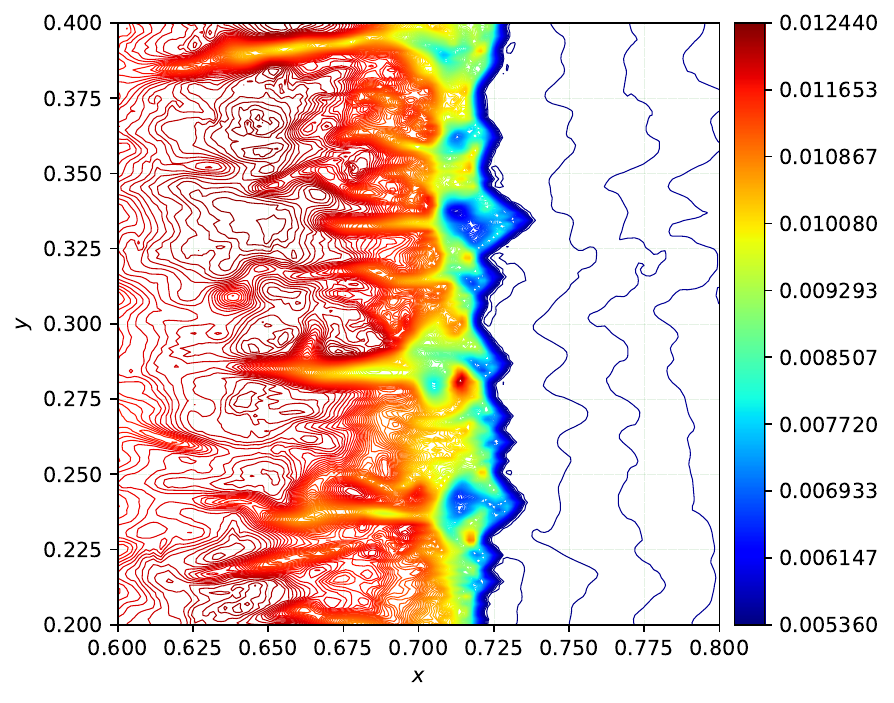}
		\caption{$h$, O3\_ES}
	\end{subfigure}
	\begin{subfigure}[b]{0.45\textwidth}
		\includegraphics[width=\textwidth]{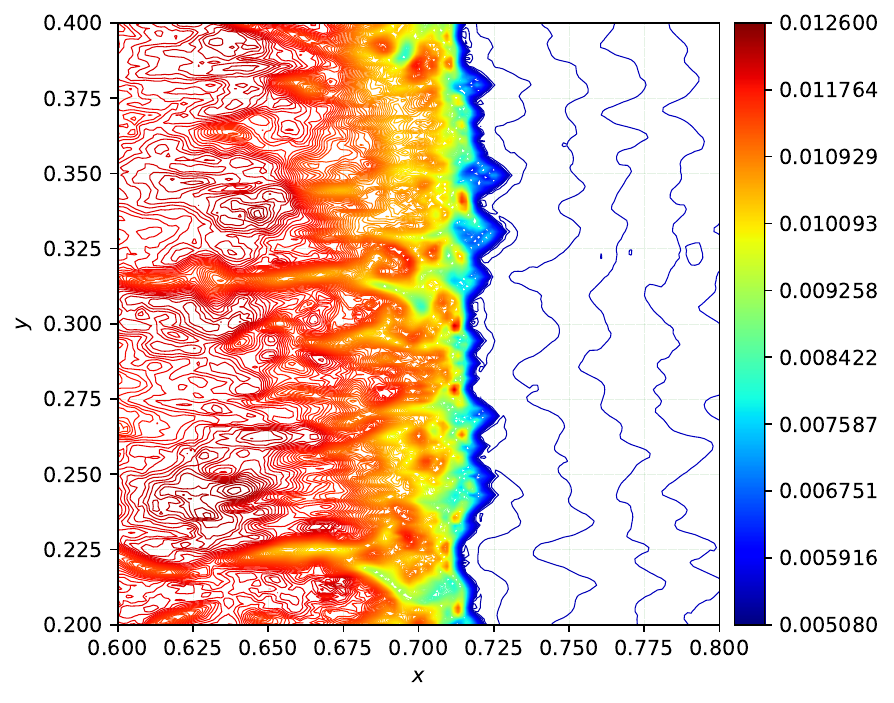}
		\caption{$h$, O4\_ES}
	\end{subfigure}
	\caption{\nameref{test13} Contour plot of depth field for different schemes at time $T=36$ unit on $1040\times400$ mesh. }
	\label{fig:test13d}
\end{figure}

\section{Summary and conclusions}\label{sec:sum}

We have developed semi-discretely entropy stable schemes for the shear shallow water model which is a non-conservative hyperbolic system modeling shallow flows but including horizontal vorticity effects. The conservative part of the model is identical to the Ten-moment model of gas dynamics, and the non-conservative terms are due to gravity. For conservative systems, the existence of an entropy condition is related to the symmetrizability of the system, but this is not sufficient for non-conservative systems. In fact, the SSW model does not become symmetric when written in terms of entropy variables. However, we can exploit the symmetrizability of the conservative part to construct entropy conservative and entropy stable schemes since the non-conservative terms do not contribute to the entropy equation.

We have constructed up to fourth-order finite difference schemes which satisfy the entropy inequality. The inequality is obtained due to the addition of carefully designed dissipative fluxes based on entropy scaled eigenvectors. The fully discrete schemes obtained with RK time stepping have been applied to several test problems like dam break and roll waves and shown to yield stable solutions that compare well with some exact solutions. The fully discrete schemes are observed to satisfy the entropy inequality in the numerical results. The roll wave solutions are able to match the experimental results of Brock. In multi-dimensions, the roll waves also generate turbulent like solutions and carbuncle like features that have been observed from other numerical techniques based on approximate Riemann solvers that include five waves in their model. Thus, the proposed schemes are expected to be similar to such accurate Riemann solver models in their wave resolution capabilities.

%

\begin{acknowledgements}
The work of Praveen Chandrashekar is supported by the Department of Atomic Energy,  Government of India, under project no.~12-R\&D-TFR-5.01-0520. The work of Harish Kumar is supported in parts by DST-SERB, MATRICS grant with file No. MTR/2019/000380.
\end{acknowledgements}

 \section*{Conflict of interest}
 The authors declare that they have no conflict of interest.
 \section*{Data Availability Declaration}
Data will be made available on reasonable request. 
\bibliographystyle{spmpsci}      
\bibliography{main}
\appendix

\section{A note on non-symmetrizability of shear shallow water model}\label{symmetrizability}
In this section, we will discuss the symmetrizability of the following SSW model in one dimension, i.e., we consider,
\begin{equation}\label{eq:ssw1d}
	\frac{\partial \con}{\partial t}+\frac{\partial \fx(\con)}{\partial x}+\tilde\Bx(\con)\frac{\partial \con}{\partial x}=0,
\end{equation}
where $\con, \fx$ and $\tilde\Bx$ are defined in Section \eqref{sec:entropy}.
Additionally, this system has the entropy pair $(\eta,q)$ \eqref{entropy-pair}, such that in addition to \eqref{eq:ssw} the following equality holds,
\begin{equation*}
	\frac{\partial \eta}{\partial t}+\frac{\partial q}{\partial x}=0.
\end{equation*}
for smooth solutions. For detailed proof, refer to Lemma \eqref{prop:entropy}.  In the standard symmetrization theory \cite{godunov1961,lax1973hyperbolic,harten1983symmetric,mock1980systems}, one seeks a change of variable $\mathbf{U} \to \mathbf{\evar}$ applied to \eqref{eq:ssw1d} so that when transformed
\begin{align*}
	\frac{\partial \mathbf{U}}{\partial \evar}\frac{\partial \evar}{\partial t}+\bigg(\frac{\partial \fx}{\partial \mathbf{U}}+\tilde\Bx\bigg)\frac{\partial \mathbf{U}}{\partial \evar}\frac{\partial \evar}{\partial x}=0,
\end{align*}
the matrix $\frac{\partial \mathbf{U}}{\partial \evar}$ is symmetric, positive definite and the matrix $\tilde A_1 = \bigg(\frac{\partial \fx}{\partial \mathbf{U}}+\tilde\Bx\bigg)\frac{\partial \mathbf{U}}{\partial \evar}$ is symmetric. For the SSW system \eqref{eq:ssw1d}, we calculate the matrix $\tilde A_1$ to check it's symmetry, which yields

\begin{align*}
	\tilde A_1 - \tilde A_1^\top = \begin{pmatrix}
		0 & -\alpha & 0 & -\alpha v_1 & -\frac{1}{2} \alpha v_2 & 0 \\
		\alpha & 0 & \alpha v_2 & -\beta_1 & \frac{1}{2} \alpha \p_{12} & \beta_2 \\
		0 & -\alpha v_2 & 0 & -\alpha v_1 v_2 & -\frac{1}{2} \alpha v_2^2 & 0 \\
		\alpha v_1 & \beta_1 & \alpha v_1 v_2 & 0 & a & v_1 \beta_2 \\
		\frac{1}{2}\alpha v_2 & -\frac{1}{2} \alpha \p_{12} & \frac{1}{2} \alpha v_2^2 & -a & 0 & \frac{1}{2}v_2 \beta_2\\
		0 & -\beta_2 & 0 & -\beta_2 & -\frac{1}{2} v_2 \beta_2 & 0
	\end{pmatrix}.
\end{align*}
where 
$$
\alpha=\frac{gh^2}{2}, \quad \beta_1=\frac{1}{4} g h^2 \left(v_1^2-\p_{11}\right),\quad \beta_2=\frac{1}{4} g h^2 \left(v_2^2+\p_{22}\right)
$$
$$
a=\frac{1}{8} g h^2 \left(2 \p_{12} v_1+\left(v_1^2-\p_{11}\right) v_2\right)
$$
Hence, $\tilde A_1$ is not a symmetric matrix unless $g=0$, in which case the non-conservative terms vanish from the SSW model. Furthermore, we recall the following result presented in \cite{godlewski1996numerical} which gives the necessary and sufficient condition for a non-linear system of conservation laws to admit a strictly convex entropy.
\begin{theorem}
	A necessary and sufficient condition for the conservative system,
	\begin{align}\label{eq:sys_conservative}
		\df{\con}{t} + \df{\fx}{ x}=0,
	\end{align} to posses a strictly convex entropy $\eta$ is that there exists a change of dependent variables $\con=\con(\evar)$ that symmetrizes~\eqref{eq:sys_conservative}.
\end{theorem}
Analogously, we extend the above result for the case of non-conservative hyperbolic systems of the form~\eqref{eq:ssw1d}.

\begin{proposition}
	If $\eta$ is a strictly convex entropy for the non-conservative system of the form~\eqref{eq:ssw1d} and $\eta{'}(\con)\tilde\Bx(\con)=0$, then the change of variable $\con\rightarrow\evar$ with $\evar^\top  = \eta{'}(\con)$ symmetrizes the non-conservative system if and only if
	$\tilde\Bx(\con)\con{'}(\evar)$ is symmetric.
\end{proposition}
\begin{proof}
	Define the conjugate functions
	\begin{align*}
		\eta^{*}(\evar)=\evar^\top\con(\evar)-\eta(\con(\evar)), \qquad q^{*}(\evar)=\evar^\top\f(\con(\evar))-q(\con(\evar)).
	\end{align*}
	Differentiating with respect to $\evar$ gives,
	\begin{align*}
		{\eta^{*}}^{'}(\evar)=\con(\evar)^\top-\evar^\top\con'(\evar)-\eta'(\con(\evar))\con'(\evar)=\con(\evar)^\top,
	\end{align*}
	and
	\begin{align*}
		{q^{*}}^{'}(\evar)&=\f(\con(\evar))^\top+\evar^\top\f'(\con(\evar))\con'(\evar)-q'(\con(\evar))\con'(\evar)&\\
		&=\f(\con(\evar))^\top+[\evar^\top\f'(\con(\evar))-q'(\con(\evar))]\con'(\evar)&\\
		&=\f(\con(\evar))^\top
	\end{align*}
	since,
	\begin{align*}
		\evar^\top\f'(\con(\evar))-q'(\con(\evar))=0.
	\end{align*}
	Hence, the matrices $\con'(\evar)={\eta^{*}}^{''}(\evar)$ and $\f'(\con(\evar))\con'(\evar)={q^{*}}^{''}(\evar)$ are symmetric. Moreover, the matrix $\con'(\evar)=\eta{''}(\con(\evar))^{-1}$ is positive definite.\\
	The change of variable yields
	\begin{align*}
		\con'(\evar)\evar_t+[\f'(\con(\evar)) + \tilde\Bx(\con(\evar))]\con'(\evar)\evar_x=0
	\end{align*}
	We need $[\f'(\con(\evar))+\tilde\Bx(\con(\evar))]\con'(\evar)$ to be symmetric, since, $\f'(\con(\evar))\con'(\evar)$ is symmetric we need $\tilde\Bx(\con(\evar))\con'(\evar)$ to be symmetric.
\end{proof}

\begin{remark}
	The matrix $\tilde\Bx(\con(\evar))\con'(\evar)$ for system \eqref{eq:ssw} is not symmetric, since,
	\begin{align*}
		\tilde\Bx(\mathbf{U}(\textbf{V}))\mathbf{U}'(\textbf{V}) - \left[ \tilde\Bx(\mathbf{U}(\textbf{V}))\mathbf{U}'(\textbf{V}) \right]^\top =
		\begin{pmatrix}
			0 & -\alpha & 0 & -\alpha v_1 & -\frac{1}{2} \alpha v_2 & 0 \\
			\alpha & 0 & \alpha v_2 & -\beta_1 & \frac{1}{2} \alpha \p_{12} & \beta_2 \\
			0 & -\alpha v_2 & 0 & -\alpha v_1 v_2 & -\frac{1}{2} \alpha v_2^2 & 0 \\
			\alpha v_1 & \beta_1 & \alpha v_1 v_2 & 0 & a & v_1 \beta_2 \\
			\frac{1}{2}\alpha v_2 & -\frac{1}{2} \alpha \p_{12} & \frac{1}{2} \alpha v_2^2 & -a & 0 & \frac{1}{2}v_2 \beta_2\\
			0 & -\beta_2 & 0 & -\beta_2 & -\frac{1}{2} v_2 \beta_2 & 0
		\end{pmatrix},
	\end{align*}
	where $\alpha=\frac{gh^2}{2},~\beta_1=\frac{1}{4} g h^2 \left(v_1^2-\p_{11}\right),~\beta_2=\frac{1}{4} g h^2 \left(v_2^2+\p_{22}\right),$ \\
	$a=\frac{1}{8} g h^2 \left(2 \p_{12} v_1+\left(v_1^2-\p_{11}\right) v_2\right)$. 
	
	The above matrix is identical to $\tilde A_1 - \tilde A_1^\top$ which we derived explicitly and shown above.
\end{remark}
From the above discussion, we observe that the existence of entropy pair does not guarantee the symmetrizability of the non-conservative hyperbolic systems. In particular, we have seen that the shear shallow water model has the entropy pair $(\eta,q)$ but it is not symmetrizable.
\section{Entropy scaled right eigenvectors for shear shallow water model}\label{scaledrev}
In this section, we will calculate the entropy scaled right eigenvectors for the case of $x-$direction. Consider the conservative part of the SSW system \eqref{eq:ssw},
\begin{equation}
	\df{\con}{t} + \df{\fx}{x} =\df{\con}{t} + A_1\df{\con}{x} = 0, \label{eq:jacobi}
\end{equation}
where $A_1$ is jacobian matrix of the flux function $\fx$. To derive the eigenvalues and right eigenvectors, it is useful to transform the system \eqref{eq:jacobi} in terms of the primitive variables $\bm{W}$. The eigenvalues of the jacobian matrix $A_1$~\cite{biswas2021entropy,sen_entropy_2018} are given by,
\begin{align*}
	v_1-\sqrt{3\p_{11}},\quad v_1-\sqrt{\p_{11}}, \quad v_1,\quad v_1,\quad v_1+\sqrt{\p_{11}}, \quad v_1+\sqrt{3\p_{11}}.
\end{align*}
We observe that if $\p_{11}>0$ then all eigenvalues are real.
The right eigenvector matrix ${R}^x$ for the matrix $A_1$ is given by the relation
\begin{align*}
	{R}^x=\dfrac{\partial \con}{\partial \bm{W}}{R}_{\bm{W}}^x,
\end{align*}
where $\dfrac{\partial \con}{\partial \bm{W}}$ is the jacobian matrix for the change of variable, given by
\begin{align*}
	\dfrac{\partial \con}{\partial \bm{W}}=\begin{pmatrix}
		1                                     & 0               & 0               & 0           & 0           & 0           \\
		v_1                                   & h               & 0               & 0           & 0           & 0           \\
		v_2                                   & 0               & h               & 0           & 0           & 0           \\
		\frac{1}{2}(\p_{11}+v_1^2)   & h v_1           & 0               & \frac{h}{2} & 0           & 0           \\
		\frac{1}{2}(\p_{12}+v_1 v_2) & \frac{h v_2}{2} & \frac{h v_1}{2} & 0           & \frac{h}{2} & 0           \\
		\frac{1}{2}(\p_{22}+v_2^2)   & 0               & h v_2           & 0           & 0           & \frac{h}{2}
	\end{pmatrix},
\end{align*}
and the matrix $R^x_{\bm{W}}$ is given by
\begin{align*}
	R_{\bm{W}}^x=\begin{pmatrix}
		h \p_{11}                        & 0                        & -h               & 0 & 0                       & h \p_{11}                       \\
		-\sqrt{3\p_{11}}\p_{11} & 0                        & 0                & 0 & 0                       & \sqrt{3\p_{11}}\p_{11} \\
		-\sqrt{3\p_{11}}\p_{12} & -\sqrt{\p_{11}} & 0                & 0 & \sqrt{\p_{11}} & \sqrt{3\p_{11}}\p_{12} \\
		2\p_{11}^2                       & 0                        & \p_{11} & 0 & 0                       & 2\p_{11}^2                      \\
		2\p_{11}\p_{12}         & \p_{11}         & \p_{12} & 0 & \p_{11}        & 2\p_{11}\p_{12}        \\
		2\p_{12}^2                       & 2\p_{12}        & 0                & 1 & 2 \p_{12}      & 2\p_{12}^2
	\end{pmatrix}.
\end{align*}
We need to find a scaling matrix $T^x$ such that the scaled right eigenvector matrix $\tilde{R}^x=R^x T^x$ satisfies
\begin{align}\label{eq1}
	\frac{\partial \bm{U}}{\partial \evar} ={{\tilde{R}}^x} {{}{{\tilde{R}}^x}}^\top.
\end{align}
where $\evar$ is the entropy variable vector as in Eqn.~\eqref{entvar}.
We follow Barth scaling process \cite{barth1999numerical} to scale the right eigenvectors. The scaling matrix $T^x$ is the square root of $Y^x$ where $Y^x$ has the expression
\begin{align*}
	{Y}^x= \left(\tilde{R}^x_{\bm{W}} \right)^{-1}
	\frac{\partial \bm{W}}{\partial \evar}
	\left( \frac{\partial \bm{U}}{\partial \bm{W}}\right)^{-\top}
	\left(\tilde{R}^x_{\bm{W}}\right)^{-\top},
\end{align*}
which results in
\begin{align*}
	Y^x=\begin{pmatrix}
		\frac{1}{12 h \p_{11}^2} & 0                                                                                   & 0                                               & 0                                                                                                          & 0                                                                                   & 0                                 \\
		0                                 & \frac{\p_{11} \p_{22}-\p_{12}^2}{4 h \p_{11}^2} & 0                                               & 0                                                                                                          & 0                                                                                   & 0                                 \\
		0                                 & 0                                                                                   & \frac{1}{3h}                                    & \frac{\p_{12}^2}{3 h \p_{11}}                                                            & 0                                                                                   & 0                                 \\
		0                                 & 0                                                                                   & \frac{\p_{12}^2}{3 h \p_{11}} & \frac{3(\p_{11}\p_{22}-\p_{12}^2)^2+\p_{12}^4}{3 h \p_{11}^2} & 0                                                                                   & 0                                 \\
		0                                 & 0                                                                                   & 0                                               & 0                                                                                                          & \frac{\p_{11} \p_{22}-\p_{12}^2}{4 h \p_{11}^2} & 0                                 \\
		0                                 & 0                                                                                   & 0                                               & 0                                                                                                          & 0                                                                                   & \frac{1}{12 h \p_{11}^2}
	\end{pmatrix}.
\end{align*}
The matrix $Y^x$ is a block diagonal matrix which contains the blocks of order $1$ and $2$. It is straightforward to write the square root of a block matrix of order $1$. Consider the $2\times 2$ block sub-matrix of the matrix $Y^x$ and denote it by $Y^{x}_b$,
\begin{align*}
	Y^{x}_b=\begin{pmatrix}
		\frac{1}{3h}                                    & \frac{\p_{12}^2}{3 h \p_{11}}                                                            \\
		\frac{\p_{12}^2}{3 h \p_{11}} & \frac{3(\p_{11}\p_{22}-\p_{12}^2)^2+\p_{12}^4}{3 h \p_{11}^2}
	\end{pmatrix}.
\end{align*}
We need to find matrix $T^{x}_b=\sqrt{Y^{x}_b}$. To obtain formula for the matrix $T^x_b$ we first consider the characteristic polynomial of $T^{x}_b$,
\begin{align}\label{characteristic1}
	{T^{x}_b}^{2}-\trace(T^{x}_b)T^{x}_b+\det(T^{x}_b)I=0,
\end{align}
where $\det(T^{x}_b)=\pm\sqrt{det(Y^{x}_b)}=r_1$,  say, with $r_1$ being the positive square root, given by
\begin{align*}
	r_1 & =\sqrt{\frac{3(\p_{11}\p_{22}-\p_{12}^2)^2+\p_{12}^4}{9 h^2 \p_{11}^2}-\frac{\p_{12}^4}{9 h^2 \p_{11}^2}} & \\
	& =\sqrt{\frac{3(\p_{11}\p_{22}-\p_{12}^2)^2}{9 h^2 \p_{11}^2}}                                                                        & \\
	& =\frac{\p_{11}\p_{22}-\p_{12}^2}{\sqrt{3}h \p_{11}},
\end{align*}
and $I_{2\times2}$ is the identity matrix. Observe from Eqn.~\eqref{characteristic1} that
\begin{align}\label{charateristic2}
	\trace(T^{x}_b)T^{x}_b={T^{x}_b}^2+r_1I=Y^{x}_b+r_1I,
\end{align}
and,
\begin{align*}
	(\trace(T^{x}_b))^2=\trace(\trace(T^{x}_b)T^{x}_b)=\trace(Y^{x}_b+r_1I)=\trace(Y^{x}_b)+2r_1.
\end{align*}
Simultaneously solving Eqns.~\eqref{characteristic1},~\eqref{charateristic2} we obtain
\begin{align} \label{block_T}
	T^{x}_b=\frac{1}{\sqrt{\trace(Y^{x}_b)+2r_1}}(Y^{x}_b+r_1I).
\end{align}
Observe that
\begin{align*}
	{\trace(Y^{x}_b)+2r_1} & =\frac{1}{3h}+\frac{3(\p_{11}\p_{22}-\p_{12}^2)^2+\p_{12}^4}{3 h \p_{11}^2}+\frac{2(\p_{11}\p_{22}-\p_{12}^2)}{\sqrt{3}h \p_{11}}   \\
	& =\frac{\p_{11}^2+{3(\p_{11}\p_{22}-\p_{12}^2)^2+\p_{12}^4}+2\sqrt{3}\p_{11}(\p_{11}\p_{22}-\p_{12}^2)}{3 h \p_{11}^2}      \\
	& =\frac{3(\p_{11}\p_{22}-\p_{12}^2)^2+2\sqrt{3}\p_{11}(\p_{11}\p_{22}-\p_{12}^2)+\p_{11}^2+\p_{12}^4}{3 h \p_{11}^2}      & \\
	& =\frac{(\sqrt{3}(\p_{11}\p_{22}-\p_{12}^2)+\p_{11})^2+\p_{12}^4}{3 h \p_{11}^2}.
\end{align*}
Since $h>0$, we have $\trace(Y^{x}_b)+2r_1>0$. We use notation $\alpha_1 = \sqrt{\trace(Y^{x}_b)+2r_1}$. A long simplification using the block matrix $T^x_b$ as in Eqn.~\eqref{block_T} results in
\begin{align*}
	T^x=\begin{pmatrix}
		\sqrt{\frac{1}{12 h \p_{11}^2}} & 0                                                                                          & 0                                                       & 0                                                                                           & 0                                                                                          & 0                                        \\
		0                                        & \sqrt{\frac{\p_{11} \p_{22}-\p_{12}^2}{4 h \p_{11}^2}} & 0                                                       & 0                                                                                           & 0                                                                                          & 0                                        \\
		0                                        & 0                                                                                          & \frac{\frac{1}{3h}+r_1}{\alpha_1}                       & \frac{\p_{12}^2}{3 h \p_{11}\alpha_1}                                     & 0                                                                                          & 0                                        \\
		0                                        & 0                                                                                          & \frac{\p_{12}^2}{3 h \p_{11}\alpha_1} & \frac{\beta_1(\beta_1+\p_{11})+\p_{12}^4}{3 h \p_{11}^2\alpha_1} & 0                                                                                          & 0                                        \\
		0                                        & 0                                                                                          & 0                                                       & 0                                                                                           & \sqrt{\frac{\p_{11} \p_{22}-\p_{12}^2}{4 h \p_{11}^2}} & 0                                        \\
		0                                        & 0                                                                                          & 0                                                       & 0                                                                                           & 0                                                                                          & \sqrt{\frac{1}{12 h \p_{11}^2}}
	\end{pmatrix}.
\end{align*}

\begin{remark}
	We proceed similarly in the $y-$direction, the eigenvalues for the jacobian matrix $A_2=\df{\f_2}{\con}$ are given by
	
	\begin{align*}
		v_2-\sqrt{3\p_{22}},\quad v_2-\sqrt{\p_{22}},\quad v_2,\quad v_2,\quad v_2+\sqrt{\p_{22}},\quad v_2+\sqrt{3\p_{22}}.
	\end{align*}
	and right eigenvector matrix is given by the relation
	\begin{align*}
		{R}^y=\dfrac{\partial \con}{\partial \bm{W}}{R}_{\bm{W}}^y,
	\end{align*}
	where
	\begin{align*}
		R_{\bm{W}}^y=\begin{pmatrix}
			h \p_{22}                        & 0                        & -h               & 0 & 0                       & h \p_{22}                       \\
			-\sqrt{3\p_{22}}\p_{12} & -\sqrt{\p_{22}} & 0                & 0 & \sqrt{\p_{22}} & \sqrt{3\p_{22}}\p_{12} \\
			-\sqrt{3\p_{22}}\p_{22} & 0                        & 0                & 0 & 0                       & \sqrt{3\p_{22}}\p_{22} \\
			2\p_{12}^2                       & 2 \p_{12}       & 0                & 1 & 2\p_{12}       & 2\p_{12}^2                      \\
			2\p_{22}\p_{12}         & \p_{22}         & \p_{12} & 0 & \p_{22}        & 2\p_{22}\p_{12}        \\
			2\p_{22}^2                       & 0                        & \p_{22} & 0 & 0                       & 2\p_{22}^2
		\end{pmatrix}.
	\end{align*}
	Accordingly, we obtain the scaling matrix $T^y$ as,
	\begin{align*}
		T^y=\begin{pmatrix}
			\sqrt{\frac{1}{12 h \p_{22}^2}} & 0                                                                                          & 0                                                       & 0                                                                                           & 0                                                                                          & 0                                        \\
			0                                        & \sqrt{\frac{\p_{11} \p_{22}-\p_{12}^2}{4 h \p_{22}^2}} & 0                                                       & 0                                                                                           & 0                                                                                          & 0                                        \\
			0                                        & 0                                                                                          & \frac{\frac{1}{3h}+r_2}{\alpha_2}                       & \frac{\p_{12}^2}{3 h \p_{22}\alpha_2}                                     & 0                                                                                          & 0                                        \\
			0                                        & 0                                                                                          & \frac{\p_{12}^2}{3 h \p_{22}\alpha_2} & \frac{\beta_1(\beta_1+\p_{22})+\p_{12}^4}{3 h \p_{22}^2\alpha_2} & 0                                                                                          & 0                                        \\
			0                                        & 0                                                                                          & 0                                                       & 0                                                                                           & \sqrt{\frac{\p_{11} \p_{22}-\p_{12}^2}{4 h \p_{22}^2}} & 0                                        \\
			0                                        & 0                                                                                          & 0                                                       & 0                                                                                           & 0                                                                                          & \sqrt{\frac{1}{12 h \p_{22}^2}}
		\end{pmatrix},
	\end{align*}
	where $r_2=\frac{\p_{11}\p_{22}-\p_{12}^2}{\sqrt{3}h \p_{22}}$ and $\alpha_2=\sqrt{\frac{(\sqrt{3}(\p_{11}\p_{22}-\p_{12}^2)+\p_{22})^2+\p_{12}^4}{3 h \p_{22}^2}}$.
\end{remark}


\end{document}